\documentclass[12pt]{amsart}

\usepackage{amssymb,amsmath,amsthm}
\usepackage{graphicx}
\usepackage{color}
\usepackage{thmtools}
\usepackage{thm-restate}
\usepackage[shortlabels]{enumitem}
\usepackage{mathrsfs}
\usepackage[utf8]{inputenc}
\usepackage[T1]{fontenc}
\usepackage{dsfont}
\usepackage{soul}
\usepackage[left=3cm,right=3cm,top=3cm,bottom=3cm]{geometry}

\calclayout

\newtheorem{theorem}{Theorem}[section]
\newtheorem{lemma}[theorem]{Lemma}
\newtheorem{proposition}[theorem]{Proposition}

\newtheorem{question}[theorem]{Question}
\newtheorem{corollary}[theorem]{Corollary}
\newtheorem{fact}[theorem]{Fact}

\theoremstyle{definition}
\newtheorem{definition}[theorem]{Definition}
\newtheorem{example}[theorem]{Example}

\theoremstyle{remark}
\newtheorem{remark}[theorem]{Remark}

\numberwithin{equation}{section}
\newcommand{\frakc}{\mathfrak{c}}

\newcommand{\eps}{\varepsilon}

\newcommand{\N}{\mathbb{N}}
\newcommand{\R}{\mathbb{R}}
\newcommand{\Q}{\mathbb{Q}}

\newcommand{\G}{\mathbb{G}}

\newcommand{\aA}{\mathcal{A}}
\newcommand{\bB}{\mathcal{B}}
\newcommand{\cC}{\mathcal{C}}
\newcommand{\dD}{\mathcal{D}}

\newcommand{\fF}{\mathcal{F}}

\newcommand{\jJ}{\mathcal{J}}
\newcommand{\kK}{\mathcal{K}}

\newcommand{\sS}{\mathcal{S}}

\newcommand{\uU}{\mathcal{U}}

\newcommand{\zZ}{\mathcal{Z}}

\newcommand{\concat}{{^\smallfrown}}

\DeclareMathOperator{\der}{\mathrm{d}}

\newcommand{\ol}{\overline}

\newcommand{\rstr}{\restriction}

\DeclareMathOperator{\sgn}{sgn}
\newcommand{\sm}{\setminus}
\newcommand{\sub}{\subseteq}
\DeclareMathOperator{\supp}{supp}

\newcommand{\wh}{\widehat}

\newcommand{\seq}[2]{\big\langle#1\colon\ #2\big\rangle}
\newcommand{\seqn}[1]{\big\langle#1\colon\ n\io\big\rangle}

\newcommand{\seqk}[1]{\big\langle#1\colon\ k\io\big\rangle}

\newcommand{\ctblsub}[1]{\left[#1\right]^\omega}
\newcommand{\finsub}[1]{\left[#1\right]^{<\omega}}

\newcommand{\iA}{\in\aA}

\newcommand{\io}{\in\omega}%{\in\N}

\newcommand{\wo}{{\wp(\omega)}}%{{\wp(\N)}}
\newcommand{\bo}{{\beta\omega}}%{{\beta\N}}
\newcommand{\cso}{\ctblsub{\omega}}%{\ctblsub{\N}}
\newcommand{\fso}{\finsub{\omega}}%{\finsub{\N}}

\newcommand{\elli}{{\ell_\infty}}

\newcommand{\Cantor}{2^\omega}

\newcommand{\noproof}{\hfill$\Box$}

\newcommand{\grp}{the Grothendieck property}
\newcommand{\gr}{Grothendieck\ }

\newcommand{\nik}{Nikodym\ }

\newcommand{\jn}{Josefson--Nissenzweig\ }

\begin{document}

\title[The Josefson--Nissenzweig theorem, Grothendieck property, and...]{The Josefson--Nissenzweig theorem, Grothendieck property, and finitely supported measures on compact spaces}
\author[J. K\k{a}kol]{J. K\k{a}kol}
\address{Faculty of Mathematics and Computer Science, A. Mickiewicz University, Pozna\'n, Poland, and Institute of Mathematics, Czech Academy of Sciences, Prague, Czech Republic.}
\email{kakol@amu.edu.pl}
\author[D.\ Sobota]{D. Sobota}
\address{Universit\"at Wien, Institut f\"ur Mathematik,
Kurt G\"odel Research Center, Augasse 2-6, UZA 1 --- Building 2, 1090
Wien, Austria.}
\email{ein.damian.sobota@gmail.com}
\urladdr{www.logic.univie.ac.at/~{}dsobota}
\author[L. Zdomskyy]{L. Zdomskyy}
\address{Universit\"at Wien, Institut f\"ur Mathematik,
Kurt G\"odel Research Center, Augasse 2-6, UZA 1 --- Building 2, 1090
Wien, Austria.}
\email{lzdomsky@gmail.com}
\urladdr{www.logic.univie.ac.at/~{}lzdomsky}
%\thanks{The authors were supported by the Austrian Science Fund FWF, Grant I 2374-N35.}
\thanks{The research of the first  named author is supported by the GA\v{C}R project 20-22230L and
RVO: 67985840. The second and third named authors were supported by the Austrian Science Fund FWF, Grants I 2374-N35, I 3709-N35, M 2500-N35.}

\begin{abstract}
The celebrated Josefson--Nissenzweig theorem implies that for a Banach space $C(K)$ of continuous real-valued functions on an infinite compact space $K$ there exists a sequence of Radon measures $\langle\mu_n\colon\ n\in\omega\rangle$ on $K$ which is weakly* convergent to the zero measure on $K$ and such that $\big\|\mu_n\big\|=1$ for every $n\io$. We call such a sequence of measures \textit{a Josefson--Nissenzweig sequence}. In this paper we study the situation when the space $K$ admits a Josefson--Nissenzweig sequence of measures such that its every element has finite support. We prove among the others that $K$ admits such a Josefson--Nissenzweig sequence if and only if $C(K)$ does not have the Grothendieck property restricted to functionals from the space $\ell_1(K)$. We also investigate miscellaneous analytic and topological properties of finitely supported Josefson--Nissenzweig sequences on general Tychonoff spaces.%, and in the case when such a sequence exists we can find one with pairwise disjoint supports. We also show that in this case either there is a Josefson--Nissenzweig sequence with all supports of cardinality $2$, or every finitely supported Josefson--Nissenzweig sequence $\seqn{\mu_n}$ on $K$ satisfies the equality $\lim_{n\to\infty}\big|\supp\big(\mu_n\big)\big|=\infty$.

We prove that various properties of compact spaces guarantee the existence of finitely supported Josefson--Nissenzweig sequences. One such property is, e.g., that a compact space can be represented as the limit of an inverse system of compact spaces based on simple extensions. An immediate consequence of this result is that many classical consistent examples of Efimov spaces, i.e. spaces being counterexamples to the famous Efimov problem, admit such sequences of measures.

Similarly, we show that if $K$ and $L$ are infinite compact spaces, then their product $K\times L$ always admits a finitely supported Josefson--Nissenzweig sequence. As a corollary we obtain a constructive proof that the space $C_p(K\times L)$ contains a complemented copy of the space $c_0$ endowed with the pointwise topology---this generalizes results of Cembranos and Freniche.

Finally, we provide a direct proof of the Josefson--Nissenzweig theorem for the case of Banach spaces $C(K)$.% as well as we study general properties of finitely supported Josefson--Nissenzweig sequences on Tychonoff spaces.
\end{abstract}

%\begin{keyword}
%Rosenthal's lemma \sep sequences of measures \sep selective ultrafilters \sep P-points \sep Q-points
%\subjclass[2010]{Primary: 28A33, 03E75. Secondary: 28E15.}
%\keywords{Grothendieck property, Grothendieck space, Boolean algebras, convergence of measures, weak topologies}
%\end{keyword}

\maketitle

\section{Introduction}

The celebrated Josefson--Nissenzweig theorem, stating that for every infinite-dimensional Banach space there is a sequence $\seqn{x_n^*}$ in the dual space $X^*$ which is weakly* convergent to $0$ and such that $\big\|x_n^*\big\|=1$ for every $n\io$, is one of the most fundamental results in Banach space theory and has found many remarkable applications. For instance, it was used to obtain such structural results as: (1) for every infinite compact space $K$ and infinite-dimensional Banach space $X$ the space $C(K,X)$ of all continuous functions from $K$ to $X$ is not a Grothendieck space (Khurana \cite{Khu78}); (2) for every infinite compact space $K$ the space $C(K\times K)$ of continuous real-valued functions on $K\times K$ contains a complemented copy of the Banach space $c_0$ (Cembranos \cite{Cem84}, Freniche \cite{Fre84}); (3) if $X$ is a Banach space, then $X$ contains a complemented copy of the Banach space $\ell_1$ if and only if for every infinite-dimensional Banach space $Y$ there exists a non-compact bounded linear operator $T\colon X\to Y$ (Bator \cite{Bat92}); etc. Several interesting equivalent statements of the theorem have been also found e.g. by Borwein and Fabian \cite{BF93}.

The original proofs of the theorem due to Josefson \cite{Jos75} and Nissenzweig \cite{Nis75} were rather long and complicated. Alternative or simpler proofs were later obtained by Hagler and Johnson \cite{HJ77}, Bourgain and Diestel \cite{BD84}, Behrends \cite{Beh95}, or Mujica \cite{Muj03}. The conclusion of the theorem was also proved valid separately for several particular classes of Banach spaces, e.g. spaces $C(K)$ of continuous real-valued functions on extremely disconnected compact spaces $K$ (And\^{o} \cite{And61}, 14 years before the works of Josefson and Nissenzweig!), Banach lattices (W\'{o}jtowicz \cite{Woj07}), Banach spaces which are not Asplund but have weakly* sequentially compact dual unit balls (H\'ajek and Talponen \cite{HT14}), etc. In Section \ref{section:plebanek_proof} of our paper we provide a simple measure-theoretic proof of the Josefson--Nissenzweig theorem for the class of all Banach spaces $C(K)$ of continuous real-valued functions on compact spaces $K$.

The Josefson--Nissenzweig theorem does not hold for general topological vector spaces, e.g. it is known that it fails for \textit{Fr\'echet  spaces}, i.e. metrizable and complete locally convex spaces (see Bonet \cite{Bon91}, Lindstr\"om and Schlumprecht \cite{LS93}, Bonet, Lindstr\"om and Valdivia \cite{BLV}). Banakh, K\k{a}kol and \'Sliwa \cite{BKS19} studied the validity of the theorem in the class of \textit{$C_p(X)$-spaces}, i.e. spaces of continuous real-valued functions on Tychonoff spaces $X$ endowed with the product topology, and found its remarkable connection with the Separable Quotient Problem for $C_p(X)$-spaces and complementability of the space $(c_0)_p$, the classical Banach space $c_0$ but equipped with the product topology (see Section \ref{section:notation} for explanation of the notation). Namely, they proved that given a Tychonoff space $X$, there exists a sequence $\seqn{\mu_n}$ of measures with finite supports on $X$ such that $\lim_{n\to\infty}\int_Xf{\der}\mu_n=0$ for every $f\in C_p(X)$ and $\big\|\mu_n\big\|=1$ for every $n\io$ if and only if $C_p(X)$ contains a complemented copy of $(c_0)_p$. %Hence, e.g., if $X$ contains a non-trivial convergent sequence, then it admits such a sequence of measures.
Note that there is a natural one-to-one correspondence between continuous functionals on the topological vector space $C_p(X)$ and measures on $X$ with finite supports, so their result can be considered as a characterization of those $C_p(X)$-spaces for which the Josefson--Nissenzweig theorem holds. Recall also that by the Riesz representation theorem every continuous functional on a space $C(K)$ of continuous real-valued functions on a compact space $K$ endowed with the supremum norm corresponds similarly to a unique Radon measure on $K$. We may thus introduce the following notions (cf. \cite[page 1122]{BF93}).

\begin{definition}\label{def:jn_sequence}
A sequence $\seqn{\mu_n}$ of Radon measures on a Tychonoff space $X$ is \textit{a Josefson--Nissenzweig sequence} (or \textit{a JN-sequence}) on $X$ if $\lim_{n\to\infty}\int_Xf{\der}\mu_n=0$ for every $f\in C(X)$ and $\big\|\mu_n\big\|=1$ for every $n\io$.
\end{definition}

\begin{definition}\label{def:fsjn_sequence}
A sequence $\seqn{\mu_n}$ of Radon measures on a Tychonoff space $X$ is \textit{finitely supported} (resp. \textit{countably supported}) if $\mu_n$ has finite (resp. countable) support for every $n\io$.
\end{definition}

\begin{definition}\label{def:fsjn_sequence}
A finitely supported JN-sequence (resp. countably supported JN-sequence) is called in short \textit{an fsJN-sequence} (resp. \textit{a csJN-sequence}).
\end{definition}

To show that the Josefson--Nissenzweig theorem does not hold for every space $C_p(X)$, Banakh, K\k{a}kol and \'Sliwa \cite{BKS19} proved that $\beta\N$, the \v{C}ech-Stone compactification of the space $\N$ of natural numbers, does not admit any fsJN-sequences of measures, although it clearly admits \textit{some} JN-sequence by the general Josefson--Nissenzweig theorem. This motivated them to introduce the following property.

\begin{definition}[\cite{BKS19}]
Given a Tychonoff space $X$, we say that $C_p(X)$ has \textit{the Josefson--Nissenzweig property} (\textit{the JNP}, in short) if $X$ admits an fsJN-sequence of measures.
\end{definition}

For the sake of precision, we additionally declare the following.

\begin{definition}\label{def:fsjnp}
A Tychonoff space $X$ has \textit{the finitely supported Josefson--Nissenzweig property} (or \textit{the fsJNP}) if $X$ admits an fsJN-sequence. Similarly, $X$ has \textit{the countably supported Josefson--Nissenzweig property} (or \textit{the csJNP)} if $X$ admits a csJN-sequence.
\end{definition}

Thus, given a Tychonoff space $X$, $C_p(X)$ has the JNP if and only if $X$ has the fsJNP, and it follows that e.g. every metric non-discrete space has the fsJNP and that $\beta\N$ does not have the fsJNP. The question hence arises.% (cf. \cite[]{BKS19}).

\begin{question}\label{ques:main}
Which Tychonoff spaces have finitely supported JN-sequences of measures?
\end{question}

There are several motivations standing behind Question \ref{ques:main}, which constitutes the main research problem of this paper. First of all, we would like to continue the line of research presented in \cite{BKS19} and attempt to provide necessary and sufficient conditions for Tychonoff spaces (or, in particular, compact Hausdorff spaces) implying that their $C_p$-spaces will contain a complemented copy of the space $(c_0)_p$. This would allow us, e.g., to better understand the Separable Quotient Problem for topological vector spaces of the form $C_p(X)$ (see \cite{KS18} and \cite{BKS18}). Second, because of the utility of the Josefson--Nissenzweig theorem we would like to understand it more thoroughly. In particular, since the original proofs are purely existential, it still seems necessary to understand how the sequence in the theorem can be obtained in a \textit{constructive} way, what properties it may have and to what extent it can be modified. Third, by answering Question \ref{ques:main} we could better comprehend the nature and behavior of convergent sequences of Radon measures on compact spaces, objects playing a fundamental role e.g. in probability theory. And the last, we would like to know how \textit{simple} JN-sequences on Tychonoff spaces may be. We approach those issues in the following threefold manner. %First, we assume that a Tychonoff space $X$ has the fsJNP and show what kind of fsJN-sequences it may have (e.g. we show that there must exist an fsJN-sequence with measures having disjoint supports). Second, we find the characterization of the fsJNP for compact spaces in terms of the Grothendieck property of their Banach spaces of continuous functions. And third, we present several examples of classes of compact spaces which admit fsJN-sequences (e.g. a large class of Efimov spaces). Let us elaborate more precisely on these three issues.

\medskip

In the first main research part of the paper, i.e. Sections \ref{section:jn_sequences} and \ref{section:sizes_of_supports}, we do assume that a given Tychonoff space has the fsJNP and study what kind of fsJN-sequences it may carry. In particular, starting from a given fsJN-sequence $\seqn{\mu_n}$ of measures on a Tychonoff space $X$, by manipulating its elements and studying special limit subsets of the union of the supports of $\mu_n$'s, we prove that there is another fsJN-sequence $\seqn{\nu_n}$ on $X$ with disjoint supports, that is, $\supp\big(\nu_n\big)\cap\supp\big(\nu_{n'}\big)=\emptyset$ for every $n\neq n'\io$ (Theorem \ref{theorem:disjoint_supps}). Since, by the virtue of the Schur property, fsJN-sequences are never weakly convergent to $0$, this result is related to the Dieudonn\'e--Grothendieck characterization of non-weakly compact subsets of the dual Banach space $C(X)^*$ for $X$ compact (see \cite[Theorem 14, Chapter VII]{Die84}).

It is immediate that if a Tychonoff space contains a non-trivial convergent sequence (e.g. if it is metric and non-discrete), then it admits an fsJN-sequence $\seqn{\mu_n}$ such that $\big|\supp\big(\mu_n\big)\big|=2$. Example \ref{example:schachermayer} shows that the converse does not hold, and in Example \ref{example:plebanek} we study an instance of a compact space with the fsJNP and such that its every fsJN-sequence $\seqn{\mu_n}$ satisfies the conditions $\lim_{n\to\infty}\big|\supp\big(\mu_n\big)\big|=\infty$. These two examples show that the finitely supported Josefson--Nissenzweig property for non-metric compact spaces may be realized in two extremely different ways. However, these two ways are in some sense the only ones. Namely, in Theorem \ref{theorem:sizes_of_supps} we prove that for a given compact space $K$ if there is an fsJN-sequence $\seqn{\mu_n}$ on $K$ for which there exists $M\in\N$ such that $\big|\supp\big(\mu_n\big)\big|\le M$ for every $n\io$, then there is another fsJN-sequence $\seqn{\nu_n}$ on $K$ for which we have $\big|\supp\big(\nu_n\big)\big|=2$ for every $n\io$. Thus either there is a very simple fsJN-sequence on $K$, or every fsJN-sequence on $K$ gets more and more complicated. For metric compact spaces, various examples of  ``complicated'' fsJN-sequences are presented in Proposition \ref{prop:square_fsjnseqs}.

\medskip

The second part of the paper is devoted to the study of relations between the finitely supported Josefson--Nissenzweig property of compact spaces and the Grothendieck property of their Banach spaces of continuous functions. Recall that a Banach space $X$ is \textit{a Grothendieck space} or has \textit{the Grothendieck property} if every weakly* convergent sequence of functionals in the dual space $X^*$ of $X$ is also weakly convergent. Similarly, we say that a compact space $K$ has \textit{the Grothendieck property} if $C(K)$ is a Grothendieck space. Grothendieck \cite{Gro53} proved  that spaces of the form $\ell_\infty(\Gamma)$ are Grothendieck spaces (or, equivalently, spaces $C(K)$ for $K$ compact and extremely disconnected). Later, many other Banach spaces were recognized to be Grothendieck, e.g. von Neumann algebras (Pfitzner \cite{Pfi94}), the space $H^\infty$ of bounded analytic functions on the unit disc (Bourgain \cite{Bou83}), spaces of the form $C(K)$ for $K$ an F-space (Seever \cite{See68}; see also Haydon \cite{Hay81}, Molt\'o \cite{Mol81}, Schachermayer \cite{Sch82} or Freniche \cite{Fre84_vhs}), etc. On the other hand, the space $c_0$ is not Grothendieck, since a separable Banach  space is Grothendieck if and only if it is reflexive. %It follows that closed linear subspaces of Grothendieck spaces need not be  Grothendieck, although  this property is preserved by  complemented subspaces.
In fact, Cembranos \cite{Cem84} proved that a space $C(K)$ is Grothendieck if and only if it does not contain any complemented copy of $c_0$. For more information on Grothendieck $C(K)$-spaces we refer the reader to the papers of Haydon \cite{Hay01}, Koszmider \cite{Kos04}, or Sobota and Zdomskyy \cite{SZ19}.

The Josefson--Nissenzweig theorem has found numerous applications in the study of Grothendieck Banach spaces of continuous functions, see e.g. Khurana \cite{Khu78} and Freniche \cite{Fre84}. Also, since the characterization of Grothendieck $C(K)$-spaces due to Cembranos (which we mentioned in the previous paragraph) looks very similar to the above stated characterization of the Josefson--Nissenzweig property of $C_p(X)$-spaces by Banakh, K\k{a}kol and \'Sliwa, it seemed natural to seek a connection between the finitely supported Josefson--Nissenzweig property of compact spaces and the Grothendieck property of their spaces of continuous functions. To describe such a relation, in Section \ref{section:grothendieck} we introduce \textit{the $\ell_1$-Grothendieck property}. This new property can be described as the restriction of the Grothendieck property of a given space $C(K)$ for $K$ compact to the space of measures on $K$ having countable (equivalently, finite) support, i.e. to the functionals from the subspace $\ell_1(K)$ of the dual $C(K)^*$, see Definition \ref{def:ell1_gr}. Then we prove in Theorem \ref{theorem:ell1_grothendieck_equiv_no_fsjnp} that a compact space $K$ has the fsJNP if and only if its space $C(K)$ does not have the $\ell_1$-Grothendieck property. It follows immediately that if $C(K)$ is a Grothendieck space, then $C_p(K)$ does not have the JNP---this generalizes the result of \cite{BKS19} stating that $C_p(\beta\N)$ does not have the JNP.

The $\ell_1$-Grothendieck property follows from the general Grothendieck property, but the converse is not true. Namely, in Section \ref{section:ell1_gr_no_gr} we construct a separable compact space $K$, in fact a continuous image of $\beta\N$, such that $C(K)$ has the $\ell_1$-Grothendieck property but it does not have the Grothendieck property. The construction was suggested to us by G. Plebanek and it generalizes results from his unpublished note \cite{Ple05}, where he constructed a compact space $L$ such that $C(L)$ does not have the Grothendieck property but for every separable closed subset $L'\sub L$ the space $C(L')$ is Grothendieck.

\medskip

The third main part of the paper deals with various classes of compact spaces having the finitely supported Josefson--Nissenzweig property. The first major class consists of compact spaces obtained as limits of inverse systems based on so-called simple extensions, see Definition \ref{def:inv_sys_simple_ext}. Intuitively speaking, such compact spaces may be thought as inverse limits of sequences of compact spaces such that every successor space in a sequence is obtained from its predecessor by splitting only one point into two new points. E.g. every metrizable compact space can be obtained in such a way; see also Koppelberg \cite{Kop88,Kop89} and Borodulin--Nadzieja \cite{PBN07} for many non-trivial examples coming from the theory of Boolean algebras.

It is a folklore fact that every compact space obtained as the limit of an inverse system based on simple extensions does not have the Grothendieck property. In Corollary \ref{cor:min_gen_no_ell_1_gr} we generalize this result and prove that every such compact space does not have even the $\ell_1$-Grothendieck property, or equivalently, that it admits an fsJN-sequence of measures. As a corollary we obtain that many classical (consistent) examples of Efimov spaces do have the fsJNP, too; consequently, they fail to have the $\ell_1$-Grothendieck property (Corollary \ref{cor:efimov_fsjnp}). Recall that an infinite compact space $K$ is \textit{an Efimov space} if $K$ neither contains any non-trivial convergent sequences nor any copies of $\beta\N$. The famous Efimov problem asks if there exists an Efimov space (in ZFC). So far many consistent examples of Efimov spaces have been found, see e.g. Fedorchuk \cite{Fed77}, Dow \cite{Dow05}, Dow and Fremlin \cite{DF07}, Dow and Shelah \cite{DS13}, Sobota \cite{Sob19}, but no ZFC example is known. For more information concerning the Efimov problem, we refer the reader to Hart's survey \cite{Har07}. A weaker form of the problem in terms of the space $\ell_\infty$ and Grothendieck $C(K)$-spaces was also shortly discussed in Koszmider and Shelah \cite[Section 3]{KS12}.

In Section \ref{section:tau_simple_ext} we provide a generalization of Corollary \ref{cor:min_gen_no_ell_1_gr} to a broader class of compact spaces obtained as the limits of inverse systems. The generalization has interesting connections (Corollary \ref{cor:delavega_fsjnp}) with the Separable Quotient Problem for $C_p(X)$-spaces in the context of \cite{KS18}. An interesting tool proved and used in the section is Proposition \ref{prop:transport_fsjn_seq} asserting that in particular cases we can obtain an fsJN-sequence on a given compact space by ``transporting'' it from the standard Cantor space.

Khurana's result, mentioned in the first paragraph, yields that for every infinite compact space $K$ the Banach space $C(K\times K)\cong C(K,C(K))$ is not Grothendieck. This fact, together with the aforementioned theorem of Cembranos, implies that $C(K\times K)$ contains a complemented copy of the space $c_0$---the result also proved by Freniche \cite{Fre84}. In Section \ref{section:products} we generalize this theorem by proving that given two infinite compact spaces $K$ and $L$ their product $K\times L$ always admits an fsJN-sequence (Theorem \ref{theorem:products_fsjnp}), thus, in particular, $K\times L$ does not have the $\ell_1$-Grothendieck property and the space $C_p(K\times L)$ contains a complemented copy of the space $(c_0)_p$. The result of Cembranos and Freniche follows immediately (see Corollary \ref{cor:complemented_c0p_complemented_c0}). It is worth to mention here that the theorems of Khurana, Cembranos and Freniche are all proved with an aid of the Josefson--Nissenzweig theorem and therefore their original proofs are purely existential. Our proof of Theorem \ref{theorem:products_fsjnp} is different---we provide a direct and simple definition of the required fsJN-sequence and use basic probability tools to demonstrate its properties.

\section*{Acknowledgments}

We would like to thank Grzegorz Plebanek for many valuable comments and discussions which helped us to obtain several of the results contained in the paper, in particular, we are grateful for presenting us the ideas of the proofs provided in Sections \ref{section:plebanek_proof} and \ref{section:ell1_gr_no_gr}.

\section{Preliminaries and notation\label{section:notation}}

The notations and terminology used in the paper are rather standard and follow the books of Kunen \cite{Kun80} (set theory), Engelking \cite{Eng89} (general topology), Koppelberg \cite{Kop89hbk} (Boolean algebras), Diestel \cite{Die84} (Banach space theory), Tkachuk \cite{TkaVol1} ($C_p$-theory), and Bogachev \cite{Bog07} (measure theory).

\medskip

In particular, we use the following standard notions and symbols. If $X$ is a set and $A$ its subset, then $A^c=X\sm A$ and $\chi_A$ denotes the characteristic function of $A$ in $X$. The cardinality of a set $X$ is denoted by $|X|$. $\omega$ denotes the first infinite cardinal number and $\omega_1$ denotes the first uncountable cardinal number. If $\kappa$ is a cardinal, finite or infinite, then by $[X]^\kappa$ we mean the family of all subsets of $X$ of size $\kappa$; in particular, $\ctblsub{X}$ denotes the families of all countable subsets of $X$. The families of all subsets of $X$ and all finite subsets of $X$ are denoted by $\wp(X)$ and $\finsub{X}$, respectively. The continuum, i.e. the size of the real line $\R$, is denoted either by $\frakc$ or $2^\omega$. We also put simply $\R_+=[0,\infty)$ and $\omega_+=\omega\sm\{0\}$.

\medskip

Throughout this paper, we assume that all topological spaces we consider are \textbf{Tychonoff}, so, e.g., every compact space we deal with is Hausdorff. The weight of a topological space $X$ is denoted by $w(X)$. If $X$ is a space and $A$ its subspace, then $\ol{A}^X$ denotes the closure of $A$ in $X$. We will often omit the superscript and write simply $\ol{A}$. $A^\circ$ and $\partial A$ denote the interior and the boundary of $A$ in $X$, respectively. $\beta X$ denotes the \v{C}ech--Stone compactification of $X$. Given two spaces $X$ and $Y$, $X\approx Y$ means that they are homeomorphic. The Cantor space will be usually denoted by $\Cantor$. We also usually identify $\omega$ with the discrete space $\N$ of natural numbers.

\medskip

If $\aA$ is a Boolean algebra, then by $St(\aA)$ we denote its Stone space. Recall that $St(\aA)$ is a totally disconnected compact space and that the Boolean algebra of clopen subsets of $St(\aA)$ is isomorphic to $\aA$. For every element $A\iA$ by $[A]_\aA$ we denote the corresponding clopen subset of $St(\aA)$.% A sequence $\seq{A_i}{i\in I}$ of elements of a Boolean algebra $\aA$ is an antichain if $A_i\wedge A_j=0_\aA$ for every $i\neq j\in I$.

\medskip

If $X$ is a (Tychonoff) space, then by $C_p(X)$ we denote the space of real-valued continuous functions on $X$ endowed with the pointwise topology (i.e. the topology inherited from the product space $\R^X$). If $E$ is a topological vector space, then we say that $C_p(X)$ contains \textit{a complemented copy of $E$} if there are closed linear subspaces $E_1$ and $E_2$ of $C_p(X)$ such that $C_p(X)$ is a direct algebraic sum of $E_1$ and $E_2$ (i.e. $C_p(X)=E_1+E_2$ and $E_1\cap E_2=\{0\}$), $E_1$ is isomorphic to $E$, and the natural projection from $C_p(X)$ onto $E_1$ is continuous. If $K$ is a compact space, then $C(K)$ denotes the Banach space of real-valued continuous functions on $K$ endowed with the supremum norm defined as $\|f\|_\infty=\sup\big\{|f(x)|\colon x\in K\big\}$ for every $f\in C(K)$. The symbols $\ell_1$, $\elli$, $c$ and $c_0$ denote the usual standard sequence Banach spaces. We also write $(c_0)_p$ for the space $\big\{x\in\R^\omega\colon \lim_{n\to\infty}x(n)=0\big\}$ but endowed with the product topology inherited from $\R^\omega$.

\medskip

If we say that \textit{$\mu$ is a measure on a topological space $X$}, then we mean that $\mu$ is a signed $\sigma$-additive measure defined on the Borel $\sigma$-algebra of $X$ and that $\mu$ is Radon, i.e. $\mu$ is (outer and inner) regular and locally finite. We define the norm $\|\mu\|$ of $\mu$ as
\[\|\mu\|=\sup\big\{|\mu(A)|+|\mu(B)|\colon\ A,B\sub X\text{ are Borel and disjoint}\big\}\]
If $X$ is compact, then $\|\mu\|<\infty$. $|\mu|$ denotes the variation of $\mu$---it follows that $|\mu|(X)=\|\mu\|$. On the other hand, if we say that \textit{$\mu$ is a measure on a Boolean algebra $\aA$}, then we assume that it is signed, finitely additive and that the norm $\|\mu\|$ of $\mu$ defined similarly as
\[\|\mu\|=\sup\big\{|\mu(A)|+|\mu(B)|\colon\ A,B\iA, A\wedge B=0_\aA\big\}\]
is finite. Note that every measure $\mu$ on a Boolean algebra $\aA$ (and hence on the Boolean algebra of clopen subsets of $St(\aA)$) has a unique extension to a measure $\wh{\mu}$ on $St(\aA)$ and that $\|\mu\|=\|\wh{\mu}\|$. We will usually identify $\mu$ and $\wh{\mu}$ and omit $\wh{\ }$.

A measure $\mu$ on a space $X$ is \textit{a probability measure} if $\mu(A)\ge0$ for every Borel $A$ and $\|\mu\|=1$. We say that $\mu$ \textit{vanishes at points} (or, is \textit{non-atomic}) if $\mu(\{x\})=0$ for every $x\in X$.

If $\mu$ is a measure on a space $X$, then by $\supp(\mu)$ we denote \textit{the support} of $\mu$, i.e. the smallest closed subset $L$ of $X$ such that for every open subset $U\sub X\sm L$ we have $|\mu|(U)=0$. We will say that $\mu$ is \textit{finitely (countably) supported} if $\supp(\mu)$ is a finite (countable) set. A sequence $\seqn{\mu_n}$ of measures on $X$ is \textit{finitely (countably) supported} if every $\mu_n$ is finitely (countably) supported. The space of all finitely supported measures on $X$ is denoted by $\Delta(X)$. $\ell_1(X)$ denotes on the other hand the space of all countably supported measures on $X$. Note that the norm $\|\cdot\|$ defined above makes it a Banach space isometrically isomorphic to the space $\ell_1(|X|)$. Obviously, $\Delta(X)$ is a linear subspace of $\ell_1(X)$. If $x\in X$, then by $\delta_x$ we mean \textit{the point measure} (or \textit{the Dirac measure}) concentrated at $x$ and defined as $\delta_x(A)=\chi_A(x)$. $\Delta(X)$ may be thus understood as a linear hull of a set $\big\{\delta_x\colon\ x\in X\}$ in the space $C_p(C_p(X))$. Also, each element $\mu$ of $\Delta(X)$ may be written as:
\[\mu=\sum_{x\in\supp(\mu)}\alpha_x\cdot\delta_x\]
for some non-zero $\alpha_x\in\R$ and every $x\in\supp(\mu)$. Similarly, the variation of $\mu$ may be written as $|\mu|=\sum_{x\in\supp(\mu)}\big|\alpha_x\big|\cdot\delta_x$ and thus the norm $\|\mu\|$ is equal to $\sum_{x\in\supp(\mu)}\big|\alpha_x\big|$.

\medskip

If $\mu$ is a measure on a space $X$, then $L_1(\mu)$ and $L_\infty(\mu)$ denote the spaces of all $\mu$-integrable and $\mu$-essentially bounded functions on $X$, respectively. Note that if $X$ is compact or $\mu$ is finitely supported, then $C(X)$ is a subspace of $L_1(\mu)$. If $f\in L_1(\mu)$, then we write simply $\mu(f)=\int_Xf{\der}\mu$.

If $\seqn{\mu_n}$ is a sequence of measures on a compact space $K$, then we say that $\seqn{\mu_n}$ is \textit{weakly* convergent} to a measure $\mu$ on $K$ if $\lim_{n\to\infty}\mu_n(f)=\mu(f)$ for every $f\in C(K)$, and that it is \textit{weakly convergent} to $\mu$ if $\lim_{n\to\infty}\mu_n(B)=\mu(B)$ for every Borel subset $B$ of $K$. We also say that $\seqn{\mu_n}$ is \textit{weakly* null} (\textit{weakly null}) if it is weakly* convergent (weakly convergent) to the zero measure $0$ on $K$. Note that, due to the Riesz representation theorem, these notions of weak* and weak convergences coincide with the weak* and weak convergences in the dual space $C(K)^*$ (see also \cite[Theorem 11, page 90]{Die84}). Recall also that $\ell_1(K)$ is a complemented linear subspace of $C(K)^*$ and that it has the Schur property, i.e. every weakly convergent sequence in $\ell_1(K)$ is also norm convergent.

Similarly, if $\seqn{\mu_n}$ is a finitely supported sequence of measures on a space $X$, then we say that $\seqn{\mu_n}$ is \textit{weakly* convergent} to a measure $\mu$ on $X$ if $\lim_{n\to\infty}\mu_n(f)=\mu(f)$ for every $f\in C(X)$, and that $\seqn{\mu_n}$ is \textit{weakly* null} if it is weakly* convergent to the zero measure $0$ on $X$.

\section*{Part I. JN-sequences on (non)-compact spaces}

\section{The \jn\ theorem for $C(K)$-spaces}

Josefson \cite{Jos75} and Nissenzweig \cite{Nis75} proved their theorem for general Banach spaces and both of the proofs are rather long, technical and intricate. However, when we restrict our attention only to the Banach spaces of continuous functions on compact spaces, then it appears that the theorem may be proved in a much easier way. Below we present one of such proofs suggested to the authors by G. Plebanek and relying on measure-theoretic tools (such as the Maharam theorem). Let us note here that another basic proof for the case of $C(K)$-spaces can be also easily extracted from the proof of the general \jn\ theorem due to Behrends \cite{Beh94,Beh95}, who proved the theorem using famous Rosenthal's $\ell_1$-lemma and Banach limits (however, since the space $\ell_1$ embeds into $C(K)^*$, we may omit the application of the $\ell_1$-lemma and directly go to Case 2 of Behrends' proof presented in \cite{Beh94}). A common point of the two proofs is that both consist of two cases from which the first one concerns sequences of finitely supported measures---a main subject of this paper.

%Below we present two of such proofs---one suggested to the authors by Plebanek and relying on measure-theoretic tools (Subsection \ref{section:plebanek_proof}), and the other one derived from the proof of the general Josefson--Nissenzweig theorem, due to Behrends \cite{Beh94,Beh95}, and using the existence of Banach limits on $\elli$ (Subsection \ref{section:behrends_proof}). A common point of the two proofs is that both consist of two cases from which the first one concerns sequences of finitely supported measures---a main subject studied in this paper.

\subsection{A measure-theoretic proof of the Josefson--Nissenzweig theorem for $C(K)$-spaces\label{section:plebanek_proof}}

We will prove that every infinite compact space admits a JN-sequence. Let thus $K$ be an infinite compact space. If $K$ is a scattered space, i.e. every subset of $K$ contains an isolated point in the inherited topology, then it is a simple folklore fact that $K$ contains a non-trivial sequence $\seqn{x_n}$ convergent to some point $x\in K$. A sequence $\seqn{\mu_n}$ of measures defined for each $n\io$ by the formula $\mu_n=\frac{1}{2}\big(\delta_{x_n}-\delta_x\big)$ is then a JN-sequence on $K$.

If $K$ is not scattered, then the proof requires more work. By \cite[Theorem 19.7.6]{Sem71}, there is a non-atomic probability measure $\mu$ on $K$. It follows from the celebrated Maharam theorem (\cite{Mah42}, see also \cite{Fre89}) that there exists a sequence $\seqn{B_n}$ of $\mu$-independent Borel subsets of $K$ such that $\mu\big(B_n\big)=1/2$ for every $n\io$. (The $\mu$-independence of $\seqn{B_n}$ means here that for every finite sequence $n_1,\ldots,n_k$ of distinct natural numbers and every sequence $\eps_1,\ldots,\eps_k\in\{-1,1\}$ we have:
\[\mu\Big(\bigcap_{i=1}^kB_{n_i}^{\eps_i}\Big)=\prod_{i=1}^k\mu\big(B_{n_i}^{\eps_i}\big)=1/2^k,\]
where $A^1=A$ and $A^{-1}=K\sm A$ for a subset $A$ of $K$.) For each $n\io$ define the measure $\mu_n$ as follows:
%\[\mu_n(A)=\frac{1}{2}\Big(\mu\big(B_n\cap A\big)-\mu\big(B_n^c\cap A\big)\Big),\]
\[\mu_n(A)=\mu\big(B_n\cap A\big)-\mu\big(B_n^c\cap A\big),\]
where $A$ is a Borel subset of $K$; then, $\big\|\mu_n\big\|=1$. The sequence $\seqn{\mu_n}$ is a desired JN-sequence on $K$. Indeed, note that $\mu_n(g)=\int_Kg\cdot\big(\chi_{B_n}-\chi_{B_n^c}\big){\der}\mu$ for every $n\io$ and $g\in L_1(\mu)$. By the $\mu$-independence of the sequence $\seqn{B_n}$ and the generalized Riemann--Lebesgue lemma (\cite[Page 3]{Tal84Pettis}), the bounded sequence $\seqn{\chi_{B_n}-\chi_{B_n^c}}$ of functions in $L_\infty(\mu)$ has the property that
\[\int_Kg\cdot\big(\chi_{B_n}-\chi_{B_n^c}\big){\der}\mu=0\]
for every $g\in L_1(\mu)$, which implies that $\lim_{n\to\infty}\mu_n(g)=0$ for every $g\in C(K)$, too. The proof of theorem is thus finished.

%\subsection{A functional-analytic proof of Theorem \ref{theorem:jn}\label{section:behrends_proof}}
%
%We start (the sketch of) the second proof with a short reminder on Banach limits. Recall that a functional $\lambda\colon\elli\to\R$ is \textit{a Banach limit} if for every $x=\seqn{x_n}\in\elli$ the following conditions hold:
%\begin{itemize}
%   \item if $x\in c$, then $\lambda(x)=\lim_{n\to\infty}x_n$,
%   \item if $x_n\ge0$ for every $n\io$, then $\lambda(x)\ge0$,
%   \item $\lambda\big(\seqn{x_n}\big)=\lambda\big(\seqn{x_{n+1}}\big)$.
%\end{itemize}
%(For basic information concerning Banach limits see e.g. \cite[Appendix 2]{Beh94}.)

\section{JN-sequences of measures\label{section:jn_sequences}}

This section is devoted to the study of basic analytic and topological properties of fsJN-sequences. %The main result of the section, Theorem \ref{theorem:disjoint_supps}, states that if a space has the fsJNP, then it admits an fsJN-sequence $\seqn{\mu_n}$ such that $\supp\big(\mu_n\big)\cap\supp\big(\mu_{n'}\big)=\emptyset$ for every $n\neq n'\io$.
The first result asserts that in our study of \textit{simple} JN-sequences on compact spaces we can confine our attention to finitely supported JN-sequences only.

\begin{proposition}\label{prop:fsJNP_equiv_csJNP}
The properties fsJNP and csJNP are equivalent for compact spaces.
\end{proposition}
\begin{proof}
Let $K$ be a compact space.  If $K$ has the fsJNP, then $K$ has
trivially also the csJNP, since $\Delta(K)\sub\ell_1(K)$. Let us
thus assume that $K$ has the csJNP and let $\seqn{\mu_n}$ be a
csJN-sequence. For each $n\io$ let $F_n$ be a finite subset of
$\supp\big(\mu_n\big)$ such that $\big\|\mu_n\rstr\big(K\sm
F_n\big)\big\|<1/n$, so $\big\|\mu_n\rstr F_n\big\|>1-1/n$. For
every $n\io$ define the measure $\nu_n$ on $K$ as follows:
\[\nu_n=\big(\mu_n\rstr F_n\big)\Big/\big\|\mu_n\rstr F_n\big\|,\]
then, $\nu_n\in\Delta(K)$ and $\big\|\nu_n\big\|=1$. For every $f\in C(K)$ we have:
\[\big|\nu_n(f)\big|=\big|\big(\mu_n\rstr F_n\big)(f)\big|\Big/\big\|\mu_n\rstr F_n\big\|\le\Big(\big|\mu_n(f)\big|+\big|\big(\mu_n\rstr\big(K\sm F_n\big)\big)(f)\big|\Big)\Big/\big\|\mu_n\rstr F_n\big\|<\]
\[\Big(\big|\mu_n(f)\big|+\|f\|_\infty/n\Big)\Big/\big(1-1/n\big),\]
so $\lim_{n\to\infty}\nu_n(f)=0$, since $\lim_{n\to\infty}\mu_n(f)=0$, which implies that $\seqn{\nu_n}$ is weakly* null. It follows that $\seqn{\nu_n}$ is an fsJN-sequence on $K$ and hence $K$ has the fsJNP.
\end{proof}

The following lemma shows that measures in an fsJN-sequence have eventually similar absolute values on their negative and positive parts, equal to $\approx\frac{1}{2}$.

\begin{lemma}\label{lemma:jnseq_pos_neg}
Let $\seqn{\mu_n}$ be an fsJN-sequence on a space $X$. For every $n\io$ let $P_n=\big\{x\in\supp\big(\mu_n\big)\colon\ \mu_n(x)>0\big\}$ and $N_n=\supp\big(\mu_n\big)\sm P_n$. Then,
\[\lim_{n\to\infty}\big\|\mu_n\rstr P_n\big\|=\lim_{n\to\infty}\big\|\mu_n\rstr N_n\big\|=1/2.\]
\end{lemma}
\begin{proof}
Assume there exists a subsequence $\seqk{\mu_{n_k}}$ such that the limit $\alpha=\lim_{k\to\infty}\big\|\mu_{n_k}\rstr P_{n_k}\big\|$ exists and $\alpha\neq 1/2$. % Note that for every $k\io$ we have:
%\[\tag{$*$}\big\|\mu_{n_k}\rstr N_{n_k}\big\|=1-\big\|\mu_{n_k}\rstr P_{n_k}\big\|.\]
Assume first that $\alpha>1/2$. Let $\eps=\big(\alpha-1/2\big)/2$. There exists $K\io$ such that for every $k>K$ we have:
\[\Big|\big\|\mu_{n_k}\rstr P_{n_k}\big\|-\alpha\Big|<\eps\]
and hence% by ($*$):
\[1-\big\|\mu_{n_k}\rstr N_{n_k}\big\|=\big\|\mu_{n_k}\rstr P_{n_k}\big\|>1/2+\eps.\]%>1/2-\eps>\big\|\mu_{n_k}\rstr N_{n_k}\big\|.\]
Then,
\[\big|\mu_{n_k}(X)\big|=\big\|\mu_{n_k}\rstr P_{n_k}\big\|-\big\|\mu_{n_k}\rstr N_{n_k}\big\|%=2\big\|\mu_{n_k}\rstr P_{n_k}\big\|-1
>1+2\eps-1=2\eps>0,\]
so $\liminf_{k\to\infty}\big|\mu_{n_k}(X)\big|>0$, a contradiction, since $\seqn{\mu_n}$ is weakly* null.

The proof for $\alpha<1/2$ is similar. Naturally,
\[\lim_{n\to\infty}\big\|\mu_{n_k}\rstr N_{n_k}\big\|=1-\lim_{n\to\infty}\big\|\mu_{n_k}\rstr P_{n_k}\big\|=1/2.\]
\end{proof}

For a given finitely supported sequence $\seqn{\mu_n}$ of measures on a space $X$, let us put:
\[S\big(\seqn{\mu_n}\big)=\bigcup_{n\io}\supp\big(\mu_n\big),\]
\[LS\big(\seqn{\mu_n}\big)=\Big\{x\in X\colon\ \limsup_{n\to\infty}\big|\mu_n(\{x\})\big|>0\Big\},\]
\[LI\big(\seqn{\mu_n}\big)=\Big\{x\in X\colon\ \liminf_{n\to\infty}\big|\mu_n(\{x\})\big|>0\Big\},\]
and
\[L\big(\seqn{\mu_n}\big)=\Big\{x\in X\colon\ \lim_{n\to\infty}\mu_n(\{x\})\text{ exists and is not }0\Big\}.\]
We will usually write shorter $S\big(\mu_n\big)$, $LS\big(\mu_n\big)$, $LI\big(\mu_n\big)$ and $L\big(\mu_n\big)$ instead of $S\big(\seqn{\mu_n}\big)$, $LS\big(\seqn{\mu_n}\big)$, $LI\big(\seqn{\mu_n}\big)$ and $L\big(\seqn{\mu_n}\big)$, or even simply $S$, $LS$, $LI$ and $L$ if the sequence $\seqn{\mu_n}$ is clear from the context. Of course, always $L\sub LI\sub LS\sub S$, but the reverse inclusions may not hold (cf. Proposition \ref{prop:square_fsjnseqs}).

\begin{lemma}\label{lemma:s_infinite}
If $\seqn{\mu_n}$ is an fsJN-sequence on a space $X$, then $S$ is infinite.
\end{lemma}
\begin{proof}
If $S$ is finite, then there exists $x_0\in S$ and $\eps>0$ such that $\limsup_{n\to\infty}\big|\mu_n\big(\big\{x_0\big\}\big)\big|>\eps$ (if not, then there is $N\io$ such that $\big|\mu_n(\{x\})\big|<1/|S|$ for every $x\in S$ and $n>N$, which implies that $\big\|\mu_n\big\|<1$ for every $n>N$). Let $f\in C(X)$ be such that $f(x_0)=1$ and $f(x)=0$ for every $x\in S\sm\big\{x_0\big\}$. It follows that $\limsup_{n\to\infty}\big|\mu_n(f)\big|>\eps$, which is a contradiction.
\end{proof}

Note that despite the fact that the set $S$ is a countable subset of $X$ its topology may be very hard to study --- see e.g. Levy \cite{Lev77}, where it was proved that there exist $2^\frakc$ many non-homeomorphic countable regular (hence normal) spaces without points of countable character.

\begin{remark}\label{remark:jn_pointwise_limits}
Let $\seqn{\mu_n}$ be a JN-sequence on a given space $X$. Then, since $S$ is countable, by induction we can find a subsequence $\seqk{\mu_{n_k}}$ such that $\lim_{k\to\infty}\big|\mu_{n_k}(\{x\})\big|$ exists for every $x\in X$.
\end{remark}

\begin{definition}
A sequence $\seqn{\mu_n}$ of finitely supported measures on a space $X$ is \textit{pointwise convergent} if the limit $\lim_{n\to\infty}\mu_n\big(\{x\}\big)$ exists for every $x\in X$.
\end{definition}

Note that the definition is equivalent to say that $\lim_{n\to\infty}\mu_n\big(\{x\}\big)=0$ for every $x\in X\sm L$. It follows that $L\big(\mu_n\big)=LI\big(\mu_n\big)=LS\big(\mu_n\big)\sub S\big(\mu_n\big)$ if $\seqn{\mu_n}$ is pointwise convergent. By the previous remark, every fsJN-sequence $\seqn{\mu_n}$ on a space $X$ contains a pointwise convergent fsJN-(sub)sequence $\seqk{\mu_{n_k}}$. Of course, every subsequence of a pointwise convergent sequence of measures is also pointwise convergent.

The proof of the following lemma is left to the reader.

\begin{lemma}\label{lemma:fsjn_subsequences_sets}
For every finitely supported sequence $\seqn{\mu_n}$ of measures on a space $X$ and its subsequence $\seqk{\mu_{n_k}}$ it holds:
\begin{enumerate}[(i)]
    \item $S\big(\seqk{\mu_{n_k}}\big)\sub S\big(\seqn{\mu_n}\big)$;
    \item $LS\big(\seqk{\mu_{n_k}}\big)\sub LS\big(\seqn{\mu_n}\big)$;
    \item $LI\big(\seqn{\mu_n}\big)\sub LI\big(\seqk{\mu_{n_k}}\big)$;
    \item $L\big(\seqn{\mu_n}\big)\sub L\big(\seqk{\mu_{n_k}}\big)$.
\end{enumerate}
If $\seqn{\mu_n}$ is pointwise convergent, then% $L\big(\seqn{\mu_n}\big)=L\big(\seqk{\mu_{n_k}}\big)$, so
\[L\big(\seqn{\mu_n}\big)=L\big(\seqk{\mu_{n_k}}\big)=LS\big(\seqk{\mu_{n_k}}\big)=LS\big(\seqn{\mu_n}\big).\]\noproof
\end{lemma}

The following proposition asserts that the unit square $[0,1]^2$ admits fsJN-sequences satisfying various proper inclusions between sets $L$, $LI$, $LS$ and $S$ as well as they have other quantitative properties. It also shows that even in the case of a metric space an fsJN-sequence may be quite intricate.

\begin{proposition}\label{prop:square_fsjnseqs}
Let $\alpha\in(0,1)$. The unit square $[0,1]^2$ admits fsJN-sequences $\seqn{\mu_n^1}$, $\seqn{\mu_n^2}$, $\seqn{\mu_n^3}$ and $\seqn{\mu_n^4}$ such that:
\begin{enumerate}
    \item %$\seqn{\mu_n^1}$ has the following properties:
%   \begin{enumerate}[(i)]
%       \item
$\emptyset\neq L\big(\mu_n^1\big)\subsetneq LI\big(\mu_n^1\big)\subsetneq LS\big(\mu_n^1\big)\subsetneq S\big(\mu_n^1\big)$;
%   \end{enumerate}
    \medskip
    \item %$\seqn{\mu_n^2}$ has the following properties:
    \begin{enumerate}[(i)]
        \item $LS\big(\mu_n^2\big)=\big([0,1]\cap\Q\big)\times\{0\}$, so $LS\big(\mu_n^2\big)$ is dense-in-itself;
        %\item $\emptyset=L\big(\mu_n^2\big)\subsetneq LI\big(\mu_n^2\big)=LS\big(\mu_n^2\big)\subsetneq S\big(\mu_n^2\big)$;
        \item $\emptyset=L\big(\mu_n^2\big)=LI\big(\mu_n^2\big)\subsetneq LS\big(\mu_n^2\big)\subsetneq S\big(\mu_n^2\big)$;
        \item $\mu_n^2(\{x\})\in\{0,1/2\}$ for every $x\in LS\big(\mu_n^2\big)$ and $n\io$;
        %\item for every $x\in LS\big(\mu_n^2\big)$ it holds that
        %\[\liminf_{n\to\infty}\mu_n^2(\{x\})=\limsup_{n\to\infty}\mu_n^2(\{x\})=1/2,\]
        %so for every finite $F\sub LS\big(\mu_n^2\big)$ we have
        %\[\sum_{x\in F}\liminf_{n\to\infty}\mu_n^2(\{x\})=\sum_{x\in F}\limsup_{n\to\infty}\mu_n^2(\{x\})=|F|/2,\]
        %and hence
        %\[\sum_{x\in LI(\mu_n^2)}\liminf_{n\to\infty}\mu_n^2(\{x\})=\sum_{x\in LS(\mu_n^2)}\limsup_{n\to\infty}\mu_n^2(\{x\})=\infty;\]
        \item for every $x\in LS\big(\mu_n^2\big)$ we have $\limsup_{n\to\infty}\mu_n^2(\{x\})=1/2$, so for every finite $F\sub LS\big(\mu_n^2\big)$ it holds:
        \[\sum_{x\in F}\limsup_{n\to\infty}\mu_n^2(\{x\})=|F|/2,\]
        and hence:
        \[\sum_{x\in LS(\mu_n^2)}\limsup_{n\to\infty}\mu_n^2(\{x\})=\infty;\]
    \end{enumerate}
    \medskip
    \item %$\seqn{\mu_n^3}$ has the following properties:
    \begin{enumerate}[(i)]
        \item $L\big(\mu_n^3\big)=\big([0,1]\cap\Q\big)\times\{0\}$, so $L\big(\mu_n^3\big)$ is dense-in-itself;
        \item $\emptyset\neq L\big(\mu_n^3\big)=LI\big(\mu_n^3\big)=LS\big(\mu_n^3\big)\subsetneq S\big(\mu_n^3\big)$;
        \item %it holds that
        \[\sum_{x\in L(\mu_n^3)}\lim_{n\to\infty}\mu_n^3(\{x\})=(1-\alpha)/2\le1/2\]
        and
        \[\lim_{n\to\infty}\big\|\mu_n\rstr L\big\|=(1-\alpha)/2\le1/2;\]
    \end{enumerate}
    \medskip
    \item %$\seqn{\mu_n^4}$ has the following properties:
    \begin{enumerate}[(i)]
        \item $L\big(\mu_n^4\big)=\big\{k/2^{n+1}\colon\ k,n\io,\ 0\le k<2^{n+1}\big\}\times\{0\}$;
        \item $\emptyset\neq L\big(\mu_n^4\big)=LI\big(\mu_n^4\big)=LS\big(\mu_n^4\big)=S\big(\mu_n^4\big)$;
        \item $\big\|\mu_n^4\rstr L\big\|=1$ for every $n\io$.
    \end{enumerate}
\end{enumerate}
\end{proposition}
\begin{proof}
Put $K=[0,1]^2$ and fix an enumeration $\big\{q_n\colon\ n\io\big\}$ of $[0,1]\cap\Q$.% such that $q_0=0$.
%\begin{enumerate}
%   \item

\medskip

(1) If $n\io$ is even, then let $\mu_n^1$ be defined as follows:
    \[\mu_n^1=\frac{1}{4}\big(\delta_{(0,0)}-\delta_{(0,1/(n+1))}\big)+\frac{1}{4}\big(\delta_{(1/2,0)}-\delta_{(1/2,1/(n+1))}\big),\]
    and if $n$ is odd, then define $\mu_n^1$ as follows:
    \[\mu_n^1=\frac{1}{4}\big(\delta_{(0,0)}-\delta_{(0,1/(n+1))}\big)+\frac{1}{8}\big(\delta_{(1/2,0)}-\delta_{(1/2,1/(n+1))}\big)+\frac{1}{8}\big(\delta_{(1,0)}-\delta_{(1,1/(n+1))}\big).\]
    It is immediate that $\seqn{\mu_n^1}$ is an fsJN-sequence on $K$ and:
    \[L\big(\mu_n^1\big)=\big\{(0,0)\big\},\]
    \[LI\big(\mu_n^1\big)=\big\{(0,0),\ (1/2,0)\big\},\]
    \[LS\big(\mu_n^1\big)=\big\{(0,0),\ (1/2,0),\ (1,0)\big\},\]
    \[S\big(\mu_n^1\big)=\big\{(0,0),\ (1/2,0),\ (1,0)\big\}\cup\big\{(x,1/(n+1)\colon\ x\in\{0,1/2,1\},\ n\io\big\},\]
    which yields (1).

    \medskip

%   \item
(2) Let $\big\{P_n\colon\ n\io\big\}$ be a partition of $\omega$ into infinite sets. For every $n\io$ and $k\in P_n$ write:
    \[\mu_k^2=\frac{1}{2}\big(\delta_{(q_n,0)}-\delta_{(q_n,1/k)}\big).\]
    Then, for each $k\io$ we have $\big\|\mu_k^2\big\|=1$ and it is immediate that for every $n\io$ the sequence $\seq{\mu_k^2}{k\in P_n}$ is weakly* null. We will now show that the whole sequence $\seqk{\mu_k^2}$ is weakly* null. Let $f\in C(K)$. We have:
    \[\mu_k^2(f)=\frac{1}{2}\big(f(q_n,0)-f(q_n,1/k)\big),\]
    where $n\io$ and $k\in P_n$. Fix $\eps>0$. Since $K$ is compact, $f$ is uniformly continuous, so there is $\delta>0$ such that if for $k,n\io$ we have $1/k<\delta$ and $k\in P_n$, then $\big|f(q_n,0)-f(q_n,1/k)\big|<\eps$. So pick $N\io$ such that $1/N<\delta$. For every $k>N$ and $n\io$ such that $k\in P_n$ we have:
    \[\big|\mu_k^2(f)\big|=\frac{1}{2}\big|f(q_n,0)-f(q_n,1/k)\big|<\eps.\]
    Thus, $\seqk{\mu_k^2}$ is weakly* null.

    That the conditions (i)--(iv) are satisfied follows directly from the definition of the sequence $\seqn{\mu_n^2}$.

    \medskip

%   \item
(3) Let us assume additionally that $0=q_n$ for some $n>2$. %.for every $n>1$ and $0\le i\le n$ we have $1/n<q_i$.
    For every $n\io$ define the measure $\mu_n^3$ as follows:
    \[\mu_n^3=(1-\alpha)\cdot\sum_{k=0}^n\big(\delta_{(q_k,0)}-\delta_{(q_k,1/(n+1))}\big)/2^{k+2}+\Big(\frac{\alpha}{2}+\frac{1-\alpha}{2^{n+2}}\Big)\cdot\big(\delta_{(0,1-1/(n+1))}-\delta_{(0,1-1/(n+2))}\big).\]
    It follows that $\big\|\mu_n^3\big\|=1$. %We need to show that $\seqn{\mu_n^3}$ is weakly* null---to achieve this, it is enough to prove that the sequence $\seqn{\nu_n}$ of measures defined by the formula:
%   \[\nu_n=\sum_{k=1}^n\big(\delta_{(q_k,0)}-\delta_{(q_k,1/(n+1))}\big)/2^{k+1}\]
%   is weakly* null. But this follows again from the fact that every $f\in C(K)$ is uniformly continuous---cf. the previous example.
    That $\seqn{\mu_n^3}$ is weakly* null follows again from the fact that every $f\in C(K)$ is uniformly continuous---cf. the previous example.

    %For every $n\io$ we have:
    %\[\tag{$*$}\mu_n^3\big(\big\{(q_0,0)\big\}\big)=\mu_n^3(\{(0,0)\})=\frac{\alpha}{2}+\frac{1-\alpha}{2^{n+1}},\]
    %so $\lim_{n\to\infty}\mu_n\big(\big\{(q_0,0)\big\}\big)=\alpha/2>0$, hence $\big(q_0,0\big)\in L\big(\mu_n^3\big)$. Similarly, for every $k>0$ and $n\ge k$ we have:
    For every $k\io$ and $n\ge k$ we have:
    \[\tag{$*$}\mu_n^3\big(\big\{(q_k,0)\big\}\big)=(1-\alpha)/2^{k+2},\]
    so $\big(q_k,0\big)\in L\big(\mu_n^3\big)$. If $x\in K$ is of the form $\big(q_k,1/(n+1)\big)$ or $\big(0,1-1/n\big)$ for some $k,n\io$, then $\mu_l^3(\{x\})=0$ for every $l>n+2$, so $x\not\in L\big(\mu_n^3\big)$. Thus, (i) is satisfied. (ii) follows immediately from (i) and the definition of $\seqn{\mu_n^3}$. (iii) follows from ($*$).

    \medskip

%   \item %Put $Q=[0,1]\cap\Q$. Fix a family $\big\{L_\sigma\colon\ \sigma\in\omega^{<\omega}\big\}$ of infinite pairwise disjoint subsets of $Q$ such that for every $\sigma\in\omega^{<\omega}$ the set $L_\sigma=\big\{p_n^\sigma\in Q\colon\ n\io\big\}$ has the following properties:
%   \begin{itemize}
%       \item $p_n^\sigma>p_{n+1}^\sigma$ for each $n\io$;
%       \item $p_{k+1}^\sigma+\big(p_k^\sigma-p_{k+1}^\sigma\big)/2^{n+1}>p_n^{\sigma\concat k}>p_{k+1}^\sigma$ for every $k,n\io$;
%       \item $\lim_{n\to\infty}p_n^{\sigma\concat k}=p_{k+1}^\sigma$ for every $k\io$;
%       \item $\lim_{n\to\infty}p_n^{\emptyset}=0$.
%   \end{itemize}
(4) Let $n\io$. Put $P_n=\big\{0,\ldots,2^n-1\big\}$ and for each $k\in P_n$ write $e_k^n=(2k)/2^{n+1}$ and $o_k^n=(2k+1)/2^{n+1}$. Note that $e_0^n=0$. Put: $E_n=\big\{e_k^n\colon k\in P_n\big\}$, $O_n=\big\{o_k^n\colon k\in P_n\big\}$ and $S_n=E_n\cup O_n$. The set $S_n$ will be the support of the measure $\mu_n^4$ we are going to construct.

    Note that for every $n\io$ we have $S_n=E_{n+1}$ and $\big|S_n\big|=2\big|P_n\big|=2\cdot2^n$, so $\big|S_{n+1}\big|=2\big|S_n\big|$. For every $n\io$ let $c_n=1/2^{n+1}$ and define the auxiliary measure $\nu_n$ as follows:
    \[\nu_n=\sum_{k\in P_n}\alpha_k^n\cdot\big(\delta_{(e_k^n,0)}-\delta_{(o_k^n,0)}\big),\]
    where the coefficients $\alpha_k^n$'s are defined in the following way. For $n=0$ we simply write $\alpha_0^0=1/4$ and for $n>0$ every $k\in P_n$ we define:
    \[\alpha_k^n=\begin{cases}
    \alpha_{k/2}^{n-1},&\text{ if }e_n^k\in E_{n-1},\\
    c_n/2^n,&\text{ otherwise.}
    \end{cases}\]
    Note that if $e_n^k\in E_{n-1}$, then $k$ is even, so the definition is correct. It also holds $\big|\supp\big(\nu_n\big)\big|=2^{n+1}$.

    It follows that $\big\|\nu_n\big\|=1-c_n$. Indeed, this is obviously true for $n=0$, so fix $n\ge0$ and assume that $\big\|\nu_n\big\|=1-c_n$. Since $E_n\sub S_n\sub S_{n+1}$ and $\big|S_{n+1}\big|=2\big|S_n\big|$, we have:
    \[\big\|\nu_{n+1}\big\|=\big\|\nu_n\big\|+2\cdot 2^n\cdot\frac{c_{n+1}}{2^{n+1}}=1-c_n+c_{n+1}=1-c_{n+1},\]
    as required.

    We will now show that $\seqn{\nu_n}$ is weakly* null. Let $f\in C(K)$ and $\eps>0$. Again, note that $f$ is uniformly continuous, so there is $\delta>0$ such that for every $n\io$ if $1/2^{n+1}<\delta$, then $\big|f\big(e_k^n,0\big)-f\big(o_k^n,0\big)\big|<\eps$. Let thus $N$ be such that $1/2^{n+1}<\delta$ for every $n>N$. We have:
    \[\big|\nu_n(f)\big|\le\sum_{k\in P_n}\alpha_k^n\cdot\big|f\big(e_k^n,0\big)-f\big(o_k^n,0\big)\big|<\eps\cdot\sum_{k\in P_n}\alpha_k^n<\eps\cdot\big(1-c_n\big)<\eps,\]
    which yields that $\lim_{n\to\infty}\nu_n(f)=0$.

    Finally, for every $n\io$ let
    \[\mu_n^4=c_n\cdot\delta_{(e_0^n,0)}+\nu_n,\]
    so $\mu_n^4\big(\big\{\big(e_0^n,0\big)\big\}\big)=c_n+\alpha_0^n$ and hence $\big\|\mu_n^4\big\|=1$ and $(0,0)\in L\big(\mu_n^4\big)$. Since $\lim_{n\to\infty}c_n=0$, the sequence $\seqn{\mu_n^4}$ is weakly* null.

    We will now prove (i) and (ii) together. First, notice that $\supp\big(\mu_n^4\big)=S_n\times\{0\}$ for every $n\io$, so
    \[S\big(\mu_n^4\big)=\bigcup_{n\io}S_n=\big\{k/2^{n+1}\colon\ k,n\io,\ 0\le k<2^{n+1}\big\}\times\{0\}.\]
    Next, if for $x\in(0,1]$ and $n\io$ it holds that $x\in S_n$, then $x\in E_{n+1}$, so $\mu_l^4\big(\{(x,0)\}\big)=\alpha_{n+1}^k$ for some $k\in P_{n+1}$ and every $l>n+1$. It follows that $(x,0)\in L\big(\mu_n^4\big)$. %Similarly, $\mu_n\big(\{(0,0)\}\big)=c_n+\alpha_0^0$ for every $n\io$, so $(0,0)\in L\big(\mu_n^4\big)$.
    (i) and (ii) are thus proved.

    (iii) follows from (ii).
%\end{enumerate}
\end{proof}

Let us note here that we presented the constructions of the sequences in Proposition \ref{prop:square_fsjnseqs} in the square $[0,1]^2$ only for simplicity---similar constructions may be carried out also in the unit interval $[0,1]$ or, in fact, any metric compact dense-in-itself space.

%As it is a bit tedious and long, we will postpone the proof of Proposition \ref{prop:square_fsjnseqs} to the end of the section.

\medskip

The next lemma shows that the value $1/2$ in the property (iii) of $\seqn{\mu_n^3}$ is not accidental. An intuitive meaning of the lemma is that if for some fixed points of the space $X$ the values of measures of the corresponding singletons grow too much, then they must be nullified by the values on some other points which lie closer and closer to these fixed ones (in the sense of the topology of $X$), cf. also Lemma \ref{lemma:fsjnseq_liminf_s_ls_infinite}. The property (iv) of $\seqn{\mu_n^2}$ implies that we cannot relax here limits to inferior limits or superior limits.

\begin{lemma}\label{lemma:fsjn_sum_lim_12}
For every fsJN-sequence $\seqn{\mu_n}$ on a space $X$ it holds:
\[\sum_{x\in L(\mu_n)}\lim_{n\to\infty}\big|\mu_n(\{x\})\big|\le 1/2.\]
%In particular, $\sum_{x\in L(\mu_n)}\big|\lim_{n\to\infty}\mu_n(\{x\})\big|\le 1/2$.
\end{lemma}
\begin{proof}
Let $\seqn{\mu_n}$ be an fsJN-sequence on a space $X$. For the sake of contradiction, assume that
\[\sum_{x\in L(\mu_n)}\lim_{n\to\infty}\big|\mu_n(\{x\})\big|>1/2,\]
so there is a finite set $F\sub L\big(\mu_n\big)$ such that
\[\sum_{x\in F}\lim_{n\to\infty}\big|\mu_n(\{x\})\big|>1/2.\]
Denote the above sum by $\alpha$, so $\alpha>1/2$. Let $\eps=\big(\alpha-1/2\big)/2$, so $\alpha=2\eps+1/2$. For every $x\in F$ there is $N_x\io$ such that for every $n>N_x$ we have:
\[\Big|\mu_n(\{x\})-\lim_{k\to\infty}\mu_k(\{x\})\Big|<\eps/|F|.\]
Let $N>\max_{x\in F}N_x$ be such that
\[\sgn\big(\mu_n(\{x\})\big)=\sgn\big(\lim_{k\to\infty}\mu_k(\{x\})\big)\]
for every $x\in F$ and $n>N$. %---then, for every $n>N$ and $x\in F$ it holds, too:
%\[\big|\mu_n(\{x\})-\lim_{k\to\infty}\mu_k(\{x\})\big|<\eps/|F|.\]
Let $\big\{U_x\colon x\in F\big\}$ be a collection of pairwise disjoint open subsets of $X$ such that $x\in U_x$ for every $x\in F$. % and $\ol{U_x}\cap\ol{U_y}=\emptyset$ for every $x\neq y\in F$.
There exists a continuous function $f\in C(X)$ such that: $-1\le f\le 1$, $f(x)=\sgn\big(\lim_{k\to\infty}\mu_k(\{x\})\big)$ for every $x\in F$  (so $\|f\|_\infty\le1$), and $f(y)=0$ for every $y\in X\sm\bigcup_{x\in F}U_x$. For every $n>N$ it holds:
\[\big(\mu_n\rstr F\big)(f)=\sum_{x\in F}\big|\mu_n(\{x\})\big|,\]
so
\[\big|\big(\mu_n\rstr F\big)(f)\big|=\Big|\sum_{x\in F}\big|\mu_n(\{x\})\big|\Big|=\]
\[\Big|\sum_{x\in F}\big|\mu_n(\{x\})\big|-\sum_{x\in F}\lim_{k\to\infty}\big|\mu_k(\{x\})\big|+\sum_{x\in F}\lim_{k\to\infty}\big|\mu_k(\{x\})\big|\Big|\ge\]
\[\Big|\sum_{x\in F}\lim_{k\to\infty}\big|\mu_k(\{x\})\big|\Big|-\Big|\sum_{x\in F}\big|\mu_n(\{x\})\big|-\sum_{x\in F}\lim_{k\to\infty}\big|\mu_k(\{x\})\big|\Big|=\]
\[\alpha-\Big|\sum_{x\in F}\big(\big|\mu_n(\{x\})\big|-\lim_{k\to\infty}\big|\mu_k(\{x\})\big|\big)\Big|\ge\]
\[\alpha-\sum_{x\in F}\Big|\big|\mu_n(\{x\})\big|-\lim_{k\to\infty}\big|\mu_k(\{x\})\big|\Big|>\]
\[\alpha-|F|\cdot\eps/|F|=\alpha-\eps=\eps+1/2.\]
It follows that for every $n>N$ we have:
\[\mu_n(f)=\big|\big(\mu_n\rstr F\big)(f)+\big(\mu_n\rstr(X\sm F)\big)(f)\big|\ge\]
\[\big|\big(\mu_n\rstr F\big)(f)\big|-\big|\big(\mu_n\rstr(X\sm F)\big)(f)\big|>\]
\[\eps+1/2-\|f\|_\infty\cdot\big\|\mu_n\rstr(X\sm F)\big\|>\eps+1/2-1\cdot(1/2-\eps)=2\eps>0,\]
so $\limsup_{n\to\infty}\big|\mu_n(f)\big|>2\eps>0$, which is a contradiction.
\end{proof}

\begin{corollary}
For every fsJN-sequence $\seqn{\mu_n}$ on a space $X$ it holds:
\[\lim_{x\in L(\mu_n)}\lim_{n\to\infty}\big|\mu_n(\{x\})\big|=0,\]
i.e. for every $\eps>0$ there is a finite subset $F\sub L\big(\mu_n\big)$ such that $\lim_{n\to\infty}\big|\mu_n(\{x\})\big|<\eps$ for every $x\in L\big(\mu_n\big)\sm F$.\noproof
\end{corollary}

%The next lemma will be useful in the following section.

\begin{lemma}\label{lemma:fsjnseq_liminf_s_ls_infinite}
For every pointwise convergent fsJN-sequence $\seqn{\mu_n}$ on a space $X$, if $\liminf_{n\to\infty}\big\|\mu_n\rstr L\big(\mu_n\big)\big\|<1$, then the set $S\big(\mu_n\big)\sm L\big(\mu_n\big)$ is infinite.
\end{lemma}
\begin{proof}
Let $\seqk{\mu_{n_k}}$ be such a subsequence that $\lim_{k\to\infty}\big\|\mu_{n_k}\big\|=\alpha$, where $\alpha<1$. There is $K\io$ such that for every $k>K$ we have:
\[\Big|\big\|\mu_{n_k}\rstr L\big\|-\alpha\Big|<(1-\alpha)/2,\]
so $\big\|\mu_{n_k}\rstr L\big\|-\alpha/2<1/2$. Since $\seqk{\mu_{n_k}}$ is pointwise convergent, $\lim_{k\to\infty}\mu_{n_k}(\{x\})=0$ for every $x\in S\sm L$, so if $S\sm L$ is finite, then there is $K'>K$ such that for every $k>K'$ we have $\big\|\mu_{n_k}\rstr(S\sm L)\big\|<(1-\alpha)/2$, so $\big\|\mu_{n_k}\rstr(S\sm L)\big\|+\alpha/2<1/2$, but then for every $k>K'$ we also have:
\[1=\big\|\mu_{n_k}\big\|=\Big(\big\|\mu_{n_k}\rstr(S\sm L)\big\|+\alpha/2\Big)+\Big(\big\|\mu_{n_k}\rstr L\big\|-\alpha/2\Big)<1/2+1/2=1,\]
a contradiction.
\end{proof}

Note that Proposition \ref{prop:square_fsjnseqs}.(4) provides an example of an fsJN-sequence for which the assumption stated in the above lemma does not hold. The following lemma asserts an interesting and useful property of the subspace $S\big(\mu_n\big)$.

\begin{lemma}\label{lemma:fsjn_f_bounded}
Let $\seqn{\mu_n}$ be an fsJN-sequence on a Tychonoff space $X$. Then, every function $f\in C(X)$ is bounded on the subspace $\ol{S\big(\mu_n\big)}^X$.
\end{lemma}
\begin{proof}
Suppose that there is a function $f\in C(X)$, $f\ge0$, which is unbounded on
$S\big(\mu_n\big)$. Passing to a subsequence of $\seqn{\mu_n}$, if
necessary, we may assume that there exist a strictly increasing
sequence $\seqn{k_n\io}$ and a sequence $\seqn{x_n\in\supp\big(\mu_n\big)}$
such that for every $n\io$ we have
\[f\Big[\bigcup_{l<n}\supp\big(\mu_l\big)\Big]\subset\big(0,k_n-1\big)\]
and $f\big(x_n\big)>k_n$. It follows that $f\big(x_n\big)<k_{n+1}-1$. For every $n\io$ put
\[\epsilon_n=k_{n+1}-1-f\big(x_n\big)\]
and let
$\rho_n\colon X\to\big[0,\epsilon_n\big)$ (note that $\epsilon_n>0$) be a continuous function such that:
\[\rho_n\rstr\Big(X\setminus f^{-1}\big[\big(k_n,k_{n+1}-1\big)\big]\Big)\equiv 0,\]
the restriction $\rho_n\rstr\Big(f^{-1}\big(f\big(x_n\big)\big)\cap\supp\big(\mu_n\big)\Big)$ is injective and non-zero, and
\[\rho_n\rstr\Big(f^{-1}\big[\big(k_n,k_{n+1}-1\big)\big]\setminus f^{-1}\big(f\big(x_n\big)\big)\Big)\cap\supp\big(\mu_n\big)\equiv 0.\]
The existence of such a function $\rho_n$ is a direct consequence of the complete regularity of $X$.
By replacing $f$ on $f^{-1}\big[\big(k_n,k_{n+1}-1\big)\big]$ by $f+\rho_n$ for each
$n\io$, we may additionally obtain a continuous function $h$ on $X$ such that there is no $x\neq x_n$ in
$\supp\big(\mu_n\big)$ for which we have $h(x)=h\big(x_n\big)$ (to see the continuity of $h$, note that the family $\big\{\rho_n^{-1}\big[\R\sm\{0\}\big]\colon\ n\io\big\}$ is locally finite).

By induction on $n\io$, construct a continuous function
$g_n\colon\big[0,k_{n+1}-1\big]\to\mathbb R$ such that
\[\Big|\sum_{x\in\supp(\mu_n)}g_n(h(x))\cdot\mu_n(x)\Big|>n,\]
and $g_n\rstr\big[0,k_n-1\big]=g_{n-1}$. Then, the union $g=\bigcup_{n\io}g_n$ is a continuous function $g\colon\R_+\to\R$ and has the property that
\[\Big|\sum_{x\in\supp(\mu_n)}g(h(x))\cdot\mu_n(x)\Big|>n\]
for all $n\io$, contradicting the fact that $\seqn{\mu_n}$ is an fsJN-sequence.
\end{proof}

\begin{corollary}
If a normal space $X$ admits an fsJN-sequence $\seqn{\mu_n}$, then the subspace $\ol{S\big(\mu_n\big)}^X$ is pseudocompact.
\end{corollary}
\begin{proof}
Put $S=\ol{S\big(\mu_n\big)}^X$. Let $f\in C(S)$. By the Tietze extension theorem there is $F\in C(X)$ extending $f$. By Lemma \ref{lemma:fsjn_f_bounded}, $f=F\rstr S$ is bounded.
\end{proof}

\subsection{Disjointly supported fsJN-sequences\label{section:fsjn_disjoint_supps}}

In this section we will show that if a compact space $K$ has the fsJNP, then $K$ admits an fsJN-sequence with disjoint supports (Theorem \ref{theorem:disjoint_supps}). Let us thus start with the following convenient definition.

\begin{definition}
A finitely supported sequence $\seqn{\mu_n}$ of measures on a space $X$ is \textit{disjointly supported} if $\supp\big(\mu_n\big)\cap\supp\big(\mu_{n'}\big)=\emptyset$ for every $n\neq n'\io$.
\end{definition}

The following two lemmas imply that if a space admits an fsJN-sequence with measures having supports of size $2$, then there exists also such a sequence with disjoint supports. In Theorem \ref{theorem:disjoint_supps} we will generalize this result, however the proof will be much more complicated. For more information on sizes of supports of fsJN-sequences, see Section \ref{section:sizes_of_supports}, especially Theorem \ref{theorem:sizes_of_supps}.

\begin{lemma}\label{lemma:fsjn_conv_seq}
Let $X$ be a space. Fix a sequence $\seqn{x_n}$ in $X$ and a point $x\in X$. For every $n\io$ put $\mu_n=\frac{1}{2}\big(\delta_{x_n}-\delta_x\big)$. Then, $\seqn{\mu_n}$ is an fsJN-sequence if and only if $x_n\to x$ in $X$.
\end{lemma}
\begin{proof}
Easy.%THE FOLLOWING IS FALSE: The lemma follows from the fact that the mapping $X\ni x\mapsto\delta_x\in\big(C(X)^*,weak^*\big)$ is a topological embedding.
\end{proof}

\begin{lemma}\label{lemma:disjoint_supps_size_2}
Let a space $X$ admit an fsJN-sequence $\seqn{\mu_n}$ defined for every $n\io$ as $\mu_n=\frac{1}{2}\big(\delta_{x_n}-\delta_{y_n}\big)$, where $x_n,y_n\in X$. Then, there exists a disjointly supported fsJN-sequence $\seqn{\nu_n}$ defined for every $n\io$ as $\nu_n=\frac{1}{2}\big(\delta_{u_n}-\delta_{w_n}\big)$, where $u_n,w_n\in X$.
\end{lemma}
\begin{proof}
If the space $X$ contains a non-trivial convergent sequence $\seqn{z_n}$, then it is easy to see that the measures defined as $\nu_n=\frac{1}{2}\big(\delta_{z_{2n}}-\delta_{z_{2n+1}}\big)$ form an fsJN-sequence satisfying the conclusion of the lemma.

If $X$ does not contain any non-trivial convergent sequences, then, by Lemma \ref{lemma:fsjn_conv_seq}, for every $A\in\cso$ we have $\bigcap_{n\in A}\supp\big(\mu_n\big)=\emptyset$, so there exists a subsequence $\seqk{\mu_{n_k}}$ such that $\supp\big(\mu_{n_k}\big)\cap\supp\big(\mu_{n_l}\big)=\emptyset$ for every $k\neq l\io$. To finish the proof put $\nu_k=\mu_{n_k}$ for every $k\io$.
\end{proof}

%The following definition and lemma are crucial for the proof of Theorem \ref{theorem:disjoint_supps}.
%
%\begin{definition}\label{def:fsbjn_sequence}
%A sequence $\seqn{\mu_n}$ of measures on a space $X$ is \textit{a finitely supported Bounded Josefson--Nissenzweig sequence} (or \textit{an fsBJN-sequence}) if $\supp\big(\mu_n\big)$ is finite and $\big\|\mu_n\big\|=1$ for every $n\io$, and $\lim_{n\to\infty}\mu_n(f)=0$ for every bounded $f\in C(X)$.
%\end{definition}
%
%Note that if $X$ is compact (or pseudocompact), then every fsBJN-sequence on $X$ is immediately an fsJN-sequence.

Before we present the proof of the main result of this section, Theorem \ref{theorem:disjoint_supps}, we need to prove several auxiliary lemmas, mostly concerning modifying given fsJN-sequences in order to obtain new fsJN-sequences having nicer (or more tamed) properties.

\begin{lemma}\label{lemma:normalization_fsJN}
Fix $\eps>0$. Let $\seqn{\mu_n}$ be a weakly* null sequence of measures on a space $X$ such that $\big\|\mu_n\big\|>\eps$ for every $n\io$. Then, $\seqn{\mu_n/\big\|\mu_n\big\|}$ is a JN-sequence on $X$.
\end{lemma}
\begin{proof}
Let $\nu_n=\mu_n/\big\|\mu_n\big\|$ for each $n\io$. Then, $\big\|\nu_n\big\|=1$. For every $f\in C(X)$ we have:
\[\big|\nu_n(f)\big|=\big|\mu_n(f)\big|\Big/\big\|\mu_n\big\|<\big|\mu_n(f)\big|/\eps\longrightarrow0\]
as $n\to\infty$, so $\seqn{\nu_n}$ is weakly* null. It follows that it is a JN-sequence.
\end{proof}

\begin{lemma}\label{lemma:norm_conv_fsjn_seqs}
Let $\seqn{\mu_n}$ and $\seqn{\nu_n}$ be two finitely supported sequences of measures on a space $X$ such that $\lim_{n\to\infty}\big\|\mu_n-\nu_n\big\|=0$. Assume that $\seqn{\mu_n}$ is an fsJN-sequence on $X$, $\big\|\nu_n\big\|=1 $ for every $n\io$, and that every function $f\in C(X)$ is bounded on $S\big(\nu_n\big)$. Then, $\seqn{\nu_n}$ is also an fsJN-sequence on $X$.
\end{lemma}
\begin{proof}
It is only necessary to prove that $\lim_{n\to\infty}\nu_n(f)=0$ for every $f\in C(X)$. Let thus $f\in C(X)$ and put $\alpha=\sup\big\{|f(x)|\colon\ x\in S\big(\mu_n\big)\big\}$ and $\beta=\sup\big\{|f(x)|\colon\ x\in S\big(\nu_n\big)\big\}$. By Lemma \ref{lemma:fsjn_f_bounded} the function $f$ is bounded on $S\big(\mu_n\big)$, so $\alpha<\infty$. Similarly, $\beta<\infty$ by the assumption. We then have:
\[\big|\nu_n(f)\big|\le\big|\mu_n(f)-\nu_n(f)\big|+\big|\mu_n(f)\big|\le(\alpha+\beta)\big\|\mu_n-\nu_n\big\|+\big|\mu_n(f)\big|,\]
so $\lim_{n\to\infty}\nu_n(f)=0$. It follows that $\seqn{\nu_n}$ is an fsJN-sequence on $X$.
\end{proof}

\begin{lemma}\label{lemma:pointwise_conv}
Let $\seqn{\mu_n}$ and $\seqn{\nu_n}$ be two sequences of measures on a space $X$ such that $\lim_{n\to\infty}\big\|\mu_n-\nu_n\big\|=0$. If $\seqn{\mu_n}$ is pointwise convergent, then $\seqn{\nu_n}$ is also pointwise convergent and $L\big(\mu_n\big)=L\big(\nu_n\big)$.
\end{lemma}
\begin{proof}
For every $x\in X$ we have $\big|\mu_n(\{x\})-\nu_n(\{x\})\big|\le\big\|\mu_n-\nu_n\big\|$, so if $\lim_{n\to\infty}\mu_n(\{x\})$ exists, then $\lim_{n\to\infty}\nu_n(\{x\})$ must exist, too. Similarly, $\lim_{n\to\infty}\mu_n(\{x\})\neq0$ if and only if $\lim_{n\to\infty}\nu_n(\{x\})\neq0$, so $L\big(\mu_n\big)=L\big(\nu_n\big)$.
\end{proof}

\begin{lemma}\label{lemma:disjoint_supps_outside_ls}
For every fsJN-sequence $\seqn{\mu_n}$ on a space $X$ there exists a pointwise convergent fsJN-sequence $\seqk{\nu_k}$ on $X$ and an increasing sequence $\seqk{n_k}$ of indices such that:
\begin{enumerate}
    \item $\supp\big(\nu_k\big)\sub\supp\big(\mu_{n_k}\big)$ for every $k\io$,
    \item for every $k\io$ there exists $\alpha_k\in\big(1,1+1/k\big)$ such that
\[\nu_k=\alpha_k\cdot\Big(\mu_{n_k}\rstr\supp\big(\nu_k\big)\Big)\quad\text{and}\quad\lim_{k\to\infty}\big\|\nu_k-\mu_{n_k}\big\|=0,\]
    \item $L\big(\seqk{\nu_k}\big)=LS\big(\seqk{\nu_k}\big)=LS\big(\seqk{\mu_{n_k}}\big)$,
    \item $\supp\big(\nu_l\big)\cap\supp\big(\nu_{l'}\big)\sub L\big(\seqk{\nu_k}\big)$ for every $l\neq l'\io$.
\end{enumerate}
\end{lemma}
\begin{proof}
By Remark \ref{remark:jn_pointwise_limits}, without loss of generality we may assume that $\seqn{\mu_n}$ is pointwise convergent. Put $Y=LS\big(\seqn{\mu_n}\big)$.

Let $n_0=0$. There exists $n_1>n_0$ such that for every $n\ge n_1$ we have:
\[\big\|\mu_n\rstr\big(\supp\big(\mu_{n_0}\big)\sm Y\big)\big\|<1/2.\]
Again, there exists $n_2>n_1$ such that for every $n\ge n_2$ we have:
\[\big\|\mu_n\rstr\Big(\big(\supp\big(\mu_{n_0}\big)\cup\supp\big(\mu_{n_1}\big)\big)\sm Y\big)\Big)\big\|<1/3.\]
Continuing in this manner, we will get a subsequence $\seqk{\mu_{n_k}}$ such that for every $k\io$ it holds:
\[\tag{$*$}\big\|\mu_{n_k}\rstr\big(\bigcup_{i=0}^{k-1}\supp\big(\mu_{n_i}\big)\sm Y\big)\big\|<1/(k+1).\]
For every $k\io$ define the measures $\nu_k'$ and $\nu_k$ as follows:
\[\nu_k'=\mu_{n_k}\rstr\supp\big(\mu_{n_k}\big)\sm\big(\bigcup_{i=0}^{k-1}\supp\big(\mu_{n_i}\big)\sm Y\big),\]
and
\[\nu_k=\alpha_k\cdot\nu_k',\]
where $\alpha_k=\big\|\nu_k'\big\|^{-1}$, so $\big\|\nu_k\big\|=1$. Note that:
\[\nu_k=\alpha_k\cdot\Big(\mu_{n_k}\rstr\supp\big(\nu_k\big)\Big).\]

(1) follows immediately. We have $k/(k+1)<\big\|\nu_k'\big\|<1$ for every $k\io$, so by ($*$) it holds:
\[\big\|\nu_k-\mu_{n_k}\big\|=\big\|\nu_k-\big(\mu_{n_k}\rstr\supp\big(\nu_k\big)\big)-\big(\mu_{n_k}\rstr\supp\big(\nu_k\big)^c\big)\big\|\le\]
\[\big\|\nu_k-\big(\mu_{n_k}\rstr\supp\big(\nu_k\big)\big)\big\|+\big\|\mu_{n_k}\rstr\supp\big(\nu_k\big)^c\big\|=\]
\[\big(\alpha_k-1)\big\|\mu_{n_k}\rstr\supp\big(\nu_k\big)\big\|+\big\|\mu_{n_k}\rstr\supp\big(\nu_k\big)^c\big\|<\]
\[\big\|\mu_{n_k}\rstr\supp\big(\nu_k\big)\big\|/k+1/(k+1)\le 1/k+1/(k+1),\]
which converges to $0$ as $k\to\infty$, hence (2) holds as well. Since $S\big(\nu_k\big)\sub S\big(\mu_{n_k}\big)$, by Lemmas \ref{lemma:fsjn_f_bounded} and \ref{lemma:norm_conv_fsjn_seqs}, $\seqn{\nu_k}$ is an fsJN-sequence on $X$. By Lemma \ref{lemma:pointwise_conv}, $\seqn{\nu_k}$ is pointwise convergent.

%Recall that $\seqn{\mu_n}$ is pointwise convergent, so the limit $\lim_{n\to\infty}\mu_n(\{x\})$ exists for every $x\in X$. If $x\in X$ is such that $\lim_{n\to\infty}\mu_n(\{x\})\ge0$, then for sufficiently large $k\io$ we have:
%\[\frac{k+1}{k}\cdot\mu_{n_k}(\{x\})\ge\nu_k(\{x\})=\alpha_k\cdot\mu_{n_k}(\{x\})\ge\mu_{n_k}(\{x\}),\]
%so $\lim_{k\to\infty}\nu_k(\{x\})$ exists and is equal to $\lim_{k\to\infty}\mu_{n_k}(\{x\})$. If $x\in X$ is such that $\lim_{n\to\infty}\mu_n(\{x\})<0$, then we proceed similarly. It follows that $\seqk{\nu_k}$ is pointwise convergent.

%We also claim that $\lim_{k\to\infty}\nu_k(f)=0$ for every $f\in C(X)$. Let thus $f\in C(X)$ and put $\sigma=\sup\big\{|f(x)|\colon\ x\in S\big(\mu_k\big)\big\}$---note that $\sigma$ is finite by Lemma \ref{lemma:fsjn_f_bounded}. For every $k\io$ we have (again we use ($*$) to obtain the last inequality):
%\[\big|\nu_k(f)\big|=\big|\mu_{n_k}\Big(f\rstr\supp\big(\mu_{n_k}\big)\sm\big(\bigcup_{i=0}^{k-1}\supp\big(\mu_{n_i}\big)\sm Y\big)\Big)\big|/\big\|\nu_k'\big\|=\]
%\[\big|\mu_{n_k}(f)-\mu_{n_k}\Big(f\rstr\big(\bigcup_{i=0}^{k-1}\supp\big(\mu_{n_i}\big)\sm Y\big)\Big)\big|/\big\|\nu_k'\big\|<\]
%\[\big|\mu_{n_k}(f)\big|\cdot(k+1)/k+\big|\mu_{n_k}\Big(f\rstr\big(\bigcup_{i=0}^{k-1}\supp\big(\mu_{n_i}\big)\sm Y\big)\Big)\big|\cdot(k+1)/k<\]
%\[\big|\mu_{n_k}(f)\big|\cdot(k+1)/k+\sigma/k,\]
%so, since $\seqk{\mu_{n_k}}$ is weakly* null, $\lim_{k\to\infty}\nu_k(f)=0$. Thus, $\seqk{\nu_k}$ is an fsJN-sequence on $X$.

We now show that (3) holds. Let $x\in LS\big(\nu_k\big)$, so $\limsup_{k\to\infty}\big|\nu_k(\{x\})\big|>0$. %Let $\seql{\nu_{k_l}}$ be such a subsequence that
%\[\lim_{l\to\infty}\big|\nu_{k_l}(\{x\})\big|=\limsup_{k\to\infty}\big|\nu_k(\{x\})\big|.\]
%Let $\alpha\in\R$ be equal to the above limits. We then have:
%\[\limsup_{l\to\infty}\big|\mu_{n_{k_l}}(\{x\})\big|=\limsup_{l\to\infty}\alpha_{k_l}^{-1}\big|\nu_{k_l}(\{x\})\big|\ge\limsup_{l\to\infty}\frac{k_l}{k_l+1}\big|\nu_{k_l}(\{x\})\big|\]
We have:
\[\limsup_{k\to\infty}\big|\mu_{n_k}(\{x\})\big|=\limsup_{k\to\infty}\alpha_k^{-1}\big|\nu_k(\{x\})\big|\ge\limsup_{k\to\infty}\frac{k}{k+1}\big|\nu_k(\{x\})\big|>0,\]
which proves that $x\in LS\big(\mu_{n_k}\big)$. Conversely, for every $x\in LS\big(\mu_{n_k}\big)$ (so $\limsup_{k\to\infty}\big|\mu_{n_k}(\{x\})\big|>0$) we have:
\[\limsup_{k\to\infty}\big|\nu_k(\{x\})\big|=\limsup_{k\to\infty}\alpha_k\big|\mu_{n_k}(\{x\})\big|\ge\limsup_{k\to\infty}\big|\mu_{n_k}(\{x\})\big|>0,\]
which yields that $x\in LS\big(\nu_k\big)$. Since $\seqk{\nu_k}$ is pointwise convergent, $L\big(\nu_k\big)=LS\big(\nu_k\big)$ and (3) is satisfied.

It is left to prove (4). Note that immediately by the definition of $\nu_k$'s we have $\supp\big(\nu_l\big)\cap\supp\big(\nu_{l'}\big)\sub Y$ for every $l\neq l'\io$. Since $\seqn{\mu_n}$ is pointwise convergent, Lemma \ref{lemma:fsjn_subsequences_sets} implies that $LS\big(\mu_{n_k}\big)=Y$. By (3), it follows that $\supp\big(\nu_l\big)\cap\supp\big(\nu_{l'}\big)\sub LS\big(\nu_k\big)=L\big(\nu_k\big)$ for every $l\neq l'$, so (4) holds.
\end{proof}

%Note that if $X$ is pseudocompact (in particular, if $X$ is compact), then $\seqk{\nu_k}$ in Lemma \ref{lemma:disjoint_supps_outside_ls} is also an fsJN-sequence.

If the set $LS(\mu_n)$ for a given fsJN-sequence $\seqn{\mu_n}$ on a space $X$ has an isolated point $x$, then it is easy to construct an fsJN-sequence $\seqn{\nu_n}$ on $X$ with $LS\big(\nu_n\big)=\{x\}$---intuitively speaking, such an fsJN-sequence is ``concentrated'' in a sense around the point $x$.

\begin{proposition}\label{prop:jn_isolated}
Let $\seqn{\mu_n}$ be an fsJN-sequence on a space $X$ such that the set $LS\big(\mu_n\big)$ has an isolated point $x$ in the relative topology. Then, there exists an increasing sequence $\seqk{n_k}$ and an fsJN-sequence $\seqk{\nu_k}$ such that $LS\big(\nu_k\big)=\{x\}$ and $\supp\big(\nu_k\big)\sub\supp\big(\mu_{n_k}\big)$ for every $k\io$.
\end{proposition}
\begin{proof}
Let $U$ be an open set in $X$ such that $U\cap LS\big(\mu_n\big)=\{x\}$ and let:
\[\alpha=\limsup_{n\to\infty}\big|\mu_n(\{x\})\big|.\]
There exists an increasing sequence $\seqk{n_k}$ such that $\lim_{k\to\infty}\big|\mu_{n_k}(\{x\})\big|=\alpha$ and $\big|\mu_{n_k}(\{x\})\big|>\alpha/2$ for every $k\io$. There exists a function $g\in C(X)$ such that $0\le g\le 1$, $g(x)=1$ and $g\rstr U^c\equiv 0$. For every $k\io$ we define the measure $\nu_k$ as follows:
\[\nu_k=g{\der}\mu_{n_k}\ \big/\ \big\|g{\der}\mu_{n_k}\big\|.\]
Note that $\big\|g{\der}\mu_{n_k}\big\|\neq0$, since
\[1\ge\big\|g{\der}\mu_{n_k}\big\|\ge|g(x)|\cdot\big|\mu_{n_k}(\{x\})\big|=\big|\mu_{n_k}(\{x\})\big|>\alpha/2>0,\]
so $\nu_k$ is well-defined and $\big\|\nu_k\big\|=1$. Obviously, $\supp\big(\nu_k\big)\sub\supp\big(\mu_{n_k}\big)\cap U$ for every $k\io$. Since
\[\big|\nu_k(f)\big|<\big|\mu_{n_k}(f\cdot g)\big|\cdot2/\alpha\]
for every $f\in C(X)$ and $k\io$, the sequence $\seqk{\nu_k}$ is weakly* convergent to $0$ and hence it is an fsJN-sequence. Finally, $\big|\nu_k(\{x\})\big|>\alpha/2$ for every $k\io$, so $\{x\}\sub LS\big(\nu_k\big)$. To prove the converse inclusion, let $y\in U\sm\{x\}$. If $y\not\in\supp\big(\mu_{n_k}\big)$ for some $k\io$, then $\nu_k(\{y\})=0$. If $y\in\supp\big(\mu_{n_k}\big)$ for some $k\io$, then
\[\big\|\nu_k(\{y\})\big|=\big|g(y)\cdot\mu_{n_k}(\{y\})\big|\big/\big\|g{\der}\mu_{n_k}\big\|<\big|g(y)\cdot\mu_{n_k}(\{y\})\big|\cdot2/\alpha,\]
which converges to $0$ as $k\to\infty$, since $y\not\in LS\big(\mu_{n_k}\big)$ (by Lemma \ref{lemma:fsjn_subsequences_sets}). It follows that $LS\big(\nu_k\big)=\{x\}.$
\end{proof}

\begin{lemma}\label{lemma:lim_l_1_supp_in_l}
Assume that a space $X$ admits an fsJN-sequence $\seqn{\mu_n}$ such that $\lim_{m\to\infty}\big\|\mu_m\rstr L\big(\mu_n\big)\big\|=1$. Then, $\seqn{\mu_n}$ is pointwise convergent and there exist a pointwise convergent fsJN-sequence $\seqk{\nu_k}$ on $X$ and an increasing sequence $\seqk{n_k}$ of indices such that:
\begin{enumerate}
    \item $\supp\big(\nu_k\big)\sub\supp\big(\mu_{n_k}\big)\cap L\big(\mu_n\big)$ for every $k\io$,
    \item for every $k\io$ there exists $\alpha_k\in\big[1,1+1/k\big)$ such that
\[\nu_k=\alpha_k\cdot\Big(\mu_{n_k}\rstr\supp\big(\nu_k\big)\Big)\quad\text{and}\quad\lim_{k\to\infty}\big\|\nu_k-\mu_{n_k}\big\|=0,\]
    \item $L\big(\seqk{\nu_k}\big)=L\big(\seqn{\mu_n}\big)$.
\end{enumerate}
In particular, $\supp\big(\nu_l\big)\sub L\big(\nu_k\big)$ for every $l\io$.
\end{lemma}
\begin{proof}
Let $L=L\big(\mu_n\big)$. For every $k\io$ let $n_k\io$ be such that $\big\|\mu_m\rstr L^c\big\|<1/(k+1)$ for every $m\ge n_k$ and put:
\[\nu_k=\Big(\mu_{n_k}\rstr L\Big)\big/\big\|\mu_{n_k}\rstr L\big\|.\]
(1) follows immediately. For every $k\io$ we have $\big\|\nu_k\big\|=1$ and $k/(k+1)<\big\|\mu_{n_k}\rstr L\big\|\le 1$, so $\nu_k=\alpha_k\cdot\Big(\mu_{n_k}\rstr\supp\big(\nu_k\big)\Big)$ for some $\alpha_k\in\big[1,1+1/k\big)$. It holds that:
\[\big\|\nu_k-\mu_{n_k}\big\|=\big\|\big(\nu_k\rstr L\big)-\big(\mu_{n_k}\rstr L\big)\big\|+\big\|\mu_{n_k}\rstr L^c\big\|=\big(1-\big\|\mu_{n_k}\rstr L\big\|\big)+\big\|\mu_{n_k}\rstr L^c\big\|=\]
\[=2\big\|\mu_{n_k}\rstr L^c\big\|<2/(k+1),\]
which converges to $0$ as $k\to\infty$. Thus, (2) is also proved.

To see that $\seqn{\mu_n}$ is pointwise convergent, note that for every $x\in L$ the limit $\lim_{n\to\infty}\mu_n(\{x\})$ exists by the definition and for every $x\in X\sm L$ we have:
\[\big|\mu_n(\{x\})\big|=1-\big\|\mu_n\rstr X\sm\{x\}\big\|\le1-\big\|\mu_n\rstr L\big\|,\]
which converges to $0$ as $n\to\infty$. The sequence $\seqk{\nu_k}$ is then pointwise convergent by Lemma \ref{lemma:pointwise_conv}. Similarly as in the proof of Lemma \ref{lemma:disjoint_supps_outside_ls}, we conclude that $\seqk{\nu_k}$ is an fsJN-sequence on $X$.

%Let $f\in C(X)$ and put $\sigma=\sup\big\{|f(x)|\colon\ x\in S\big(\mu_k\big)\big\}$---by Lemma \ref{lemma:fsjn_f_bounded}, $\sigma\in\R_+$. For every $k\io$ we have:
%\[\big|\nu_k(f)\big|=\big|\nu_k(f)-\mu_{n_k}\big|+\big|\mu_{n_k}(f)\big|\le\big\|\nu_k-\mu_{n_k}\big\|\cdot\sigma+\big|\mu_{n_k}(f)\big|,\]
%which, by (2), converges to $0$ as $k\to\infty$, so $\seqn{\nu_k}$ is an fsJN-sequence.

Since $\seqn{\mu_n}$ is pointwise convergent, $L\big(\mu_n\big)=L\big(\mu_{n_k}\big)$. By Lemma \ref{lemma:pointwise_conv}, $L\big(\nu_k\big)=L\big(\mu_{n_k}\big)$, which proves (3).
\end{proof}

%{\LARGE\textbf{(...)}}
We are now in the position to prove the main theorem.

\begin{theorem}\label{theorem:disjoint_supps}
Assume that a space $X$ has the fsJNP. Then, $X$ admits a disjointly supported fsJN-sequence.
\end{theorem}
\begin{proof}
Let $\seqn{\mu_n}$ be an fsJN-sequence on $X$. Apply Lemma \ref{lemma:disjoint_supps_outside_ls} to $\seqn{\mu_n}$ to obtain a pointwise convergent fsJN-sequence $\seqn{\nu_n}$ such that
\[\tag{$\#$}\supp\big(\nu_l\big)\cap\supp\big(\nu_{l'}\big)\sub L\big(\seqn{\nu_n}\big)\]
for every $l\neq l'\io$. Let $L=L\big(\nu_n\big)$. By going to a further subsequence, we may assume that the limit $\lim_{n\to\infty}\big\|\nu_n\rstr L\big\|$ exists and is equal to some $\alpha\in\R$. If $L=\emptyset$, then we are done by the properties of $\seqn{\nu_n}$. Let us thus assume that $|L|>0$ and enumerate $L=\big\{q_k\colon k<|L|\big\}$ (note that $|L|$ may be finite or $\omega$). By going again to a subsequence, we may assume that for every $k<|L|$ and $l\ge k$ we have:
\[\tag{$\dagger$}\Big|\nu_l\big(\big\{q_k\big\}\big)-\lim_{n\to\infty}\nu_n\big(\big\{q_k\big\}\big)\Big|<\frac{1}{l+1}\cdot\big|\lim_{n\to\infty}\nu_n\big(\big\{q_k\big\}\big)\big|.\]
It follows that $\supp\big(\nu_k\big)\cap L\sub\supp\big(\nu_l\big)\cap L$ for every pair $k,l\io$ such that $k\le l$. We now need to consider three cases depending on the value of $\alpha$.
\begin{enumerate}
    \item If $\alpha=0$, then $L=\emptyset$ and we are done.
    \item If $\alpha\in(0,1)$, then there is $N\io$ such that for every $n>N$ we have
    \[\big\|\nu_n\rstr L\big\|<\alpha+\frac{1-\alpha}{2}=\frac{1+\alpha}{2},\]
    so
    \[\big\|\nu_n\rstr L^c\big\|=1-\big\|\nu_n\rstr L\big\|>1-\frac{1+\alpha}{2}=\frac{1-\alpha}{2}>0.\]
    Since for every $n\io$ we have $\supp\big(\nu_{2n}\big)\cap\supp\big(\nu_{2n+1}\big)\sub L$, for every $n>N$ it holds:
    \[\tag{$*$}\big\|\nu_{2n}-\nu_{2n+1}\big\|\ge\big\|\nu_{2n}\rstr L^c\big\|+\big\|\nu_{2n+1}\rstr L^c\big\|>2\cdot\frac{1-\alpha}{2}=1-\alpha>0.\]
    Let us remove the first $N$ elements from the sequence $\seqn{\nu_n}$, so we may assume that ($*$) holds actually for every $n\io$. Let $\beta=1-\alpha$.
    \item If $\alpha=1$, then we can substitute the current fsJN-sequence $\seqn{\nu_n}$ by the pointwise convergent fsJN-sequence $\seqn{\nu_n}$ from Lemma \ref{lemma:lim_l_1_supp_in_l}, i.e. $\supp\big(\nu_l\big)\sub L\big(\nu_n\big)$ for every $l\io$. By going again to a subsequence we may assume that for every $k\io$ and $x\in\supp\big(\nu_{2k}\big)$ we have:
    \[\tag{$\dagger\dagger$}\big|\nu_{2k+1}(\{x\})-\lim_{n\to\infty}\nu_n(\{x\})\big|<\frac{1}{4\cdot\big|\supp\big(\nu_{2k}\big)\big|}.\]
    It follows by ($\dagger\dagger$) and Lemma \ref{lemma:fsjn_sum_lim_12} that for every $n\io$ it holds:
    \[\big\|\nu_{2k+1}\rstr\supp\big(\nu_{2k}\big)\big\|=\sum_{x\in\supp(\nu_{2k})}\big|\nu_{2k+1}(\{x\})\big|\le\sum_{x\in\supp(\nu_{2k})}\Big(\frac{1}{4\cdot\big|\supp\big(\nu_{2k}\big)\big|}+\big|\lim_{n\to\infty}\nu_n(\{x\})\big|\Big)=\]
    \[\sum_{x\in\supp(\nu_{2k})}\Big(\frac{1}{4\cdot\big|\supp\big(\nu_{2k}\big)\big|}+\lim_{n\to\infty}\big|\nu_n(\{x\})\big|\Big)\le\big|\supp\big(\nu_{2k}\big)\big|\cdot\frac{1}{4\cdot\big|\supp\big(\nu_{2k}\big)\big|}+1/2=3/4,\]
    so $\big\|\nu_{2k+1}\rstr\supp\big(\nu_{2k}\big)^c\big\|\ge1/4$ and hence:
    \[\tag{$**$}\big\|\nu_{2k}-\nu_{2k+1}\big\|\ge\big\|\nu_{2k+1}\rstr\supp\big(\nu_{2k}\big)^c\big\|\ge1/4>0.\]
    Let $\beta=1/4$.
\end{enumerate}

\medskip

Note that in both cases ($*$) and ($**$) the inequality $\big\|\nu_{2n}-\nu_{2n+1}\big\|\ge\beta>0$ holds for every $n\io$. For every $n\io$ we define the measure $\theta_n$ as follows:
\[\theta_n=\big(\nu_{2n}-\nu_{2n+1}\big)\ \big/\ \big\|\nu_{2n}-\nu_{2n+1}\big\|.\]
For every $f\in C(X)$ and $n\io$ we have:
\[\big|\theta_n(f)\big|\le\beta^{-1}\big(\big|\nu_{2n}(f)\big|+\big|\nu_{2n+1}(f)\big|\big),\]
so $\seqn{\theta_n}$ is an fsJN-sequence. We also claim that $LS\big(\theta_n\big)=\emptyset$. Let thus $x\in X$. If $x\not\in L$, then immediately by ($\#$) and the definition of $\theta_n$'s we have that $\lim_{n\to\infty}\theta_n(\{x\})=0$, so $x\not\in LS\big(\theta_n\big)$. On the other hand, if $x\in L$, then $x=q_k$ for some $k\io$, but then by ($\dagger$) for every $l\ge k$ we have:
\[\big|\theta_l\big(\big\{q_k\big\}\big)\big|=\big|\nu_{2l}\big(\big\{q_k\big\}\big)-\nu_{2l+1}\big(\big\{q_k\big\}\big)\big|\ \big/\ \big\|\nu_{2l}-\nu_{2l+1}\big\|\le\]
\[\beta^{-1}\cdot\Big|\ \Big(\nu_{2l}\big(\big\{q_k\big\}\big)-\lim_{n\to\infty}\nu_n\big(\big\{q_k\big\}\big)\Big)-\Big(\nu_{2l+1}\big(\big\{q_k\big\}\big)-\lim_{n\to\infty}\nu_n\big(\big\{q_k\big\}\big)\Big)\ \Big|\le\]
\[\beta^{-1}\cdot\Big(\ \Big|\nu_{2l}\big(\big\{q_k\big\}\big)-\lim_{n\to\infty}\nu_n\big(\big\{q_k\big\}\big)\Big|+\Big|\nu_{2l+1}\big(\big\{q_k\big\}\big)-\lim_{n\to\infty}\nu_n\big(\big\{q_k\big\}\big)\Big|\ \Big)<\]
\[\frac{2\beta^{-1}}{2l+1}\cdot\big|\lim_{n\to\infty}\nu_n\big(\big\{q_k\big\}\big)\big|,\]
which goes to $0$ as $l\to\infty$, so again $\lim_{l\to\infty}\theta_l(\{x\})=0$ and hence $x\not\in LS\big(\theta_n\big)$. This proves that $LS\big(\theta_n\big)=\emptyset$ indeed.

Finally, by Lemma \ref{lemma:disjoint_supps_outside_ls}, there exists an fsJN-sequence $\seqn{\rho_n}$ on $X$ such that for every $m\neq m'$ we have:
\[\supp\big(\rho_m\big)\cap\supp\big(\rho_{m'}\big)\sub LS\big(\seqn{\rho_n}\big)\sub LS\big(\seqn{\theta_n}\big)=\emptyset,\]
and the proof is finished.
\end{proof}

In Lemma \ref{lemma:disjoint_supps_size_2} we stated that if the supports of a given fsJN-sequence on a space $X$ have cardinality $2$, then there is a disjointly supported fsJN-sequence on $X$ having the same property. It is an easy topological fact that if $\seqn{F_n}$ is a countable collection of pairwise disjoint finite subsets of a space $X$ for which there exists $M\io$ such that $\big|F_n\big|\le M$, then there exist a subsequence $\seqk{F_{n_k}}$ and a sequence $\seqk{U_k}$ of pairwise disjoint open subsets of $X$ such that $F_{n_k}\sub U_k$ for every $k\io$. It follows that the union $\bigcup_{k\io}F_{n_k}$ is a discrete subspace of $X$. Consequently, if there is a disjointly supported fsJN-sequence $\seqn{\mu_n}$ on $X$ such that $\big|\supp\big(\mu_n\big)\big|=2$ for every $n\io$, then there is a subsequence $\seqk{\mu_n}$ such that the union $\bigcup_{k\io}\supp\big(\mu_{n_k}\big)$ is discrete (cf. Theorem \ref{theorem:sizes_of_supps}). We do not know whether the same property holds for every fsJN-sequence.

\begin{question}\label{ques:fsjnp_discrete_supps}
Does every space $X$ with the fsJNP admit an fsJN-sequence $\seqn{\mu_n}$ such that the union $\bigcup_n\supp\big(\mu_n\big)$ is a discrete subspace of $X$?
\end{question}

The positive answer to Question \ref{ques:fsjnp_discrete_supps} would bring much simplification to the study of the finitely supported Josefson--Nissenzweig property.

%\begin{corollary}
%Assume that $K$ is a compact (or pseudocompact) space with the fsJNP. Then, $K$ admits a disjointly supported fsJN-sequence.\noproof
%\end{corollary}

%\subsection{Proof of Proposition \ref{prop:square_fsjnseqs}}

\section{Sizes of supports in fsJN-sequences \label{section:sizes_of_supports}}

In this section we will study possible cardinalities of supports of measures from fsJN-sequences. We have two cases here: either (1) a space $X$ admits an fsJN-sequence $\seqn{\mu_n}$ for which there exists $M\io$ such that $\big|\supp\big(\mu_n\big)\big|\le M$ for every $n\io$, or (2) every fsJN-sequence $\seqn{\mu_n}$ on $X$ has the property that $\lim_{n\io}\big|\supp\big(\mu_n\big)\big|=\infty$. As an example of the former case we may name any space $X$ having a non-trivial convergent sequence. An appropriate example for the latter case is more difficult to find---however, it appears that the space $K$ considered in Banakh, K\k{a}kol and \'Sliwa \cite[Section 4]{BKS19} (\textit{Plebanek's example}) has the required property. In Subsection \ref{section:sizes_two_examples} we prove this statement as well as we present another example (due to Schachermayer) which is in many places very similar to Plebanek's one but satisfies the case (1).

In Subsection \ref{section:study_of_sizes_of_supports} we will provide several general statements concerning cardinalities of supports. In particular, we prove in Theorem \ref{theorem:sizes_of_supps} that if a compact space $K$ satisfies the case (1), then there exists an fsJN-sequence $\seqn{\mu_n}$ such that $\big|\supp\big(\mu_n\big)\big|=2$ for every $n\io$.

\subsection{Two examples}\label{section:sizes_two_examples}

\begin{example}\label{example:plebanek}
In Banakh, K\k{a}kol and \'Sliwa \cite[Section 4]{BKS19}, the authors provided the following example of a Boolean algebra $\dD$ due to Plebanek:
\[\dD=\big\{A\in\wo\colon\ \lim_{n\to\infty}\frac{|A\cap\{0,\ldots,n-1\}|}{n}\in\{0,1\}\big\}.\]
Let $\dD$ be called \textit{the density Boolean algebra}. Since for each $n\io$ the set $\{n\}$ belongs to $\dD$ and is an atom therein, we may consider $St(\dD)$ as a compactification of $\omega$. Let us additionally define the ideal $\zZ$ and the ultrafilter $p$ in $\dD$ as follows:
\[\zZ=\big\{A\in\wo\colon\ \lim_{n\to\infty}\frac{|A\cap\{0,\ldots,n-1\}|}{n}=0\big\}\]
and
\[p=\dD\sm\zZ.\]
\end{example}

\begin{proposition}\label{prop:ex_plebanek}
The density Boolean algebra $\dD$ has the following properties:
\begin{enumerate}
    \item $St(\dD)$ does not have any non-trivial convergent sequences;
    \item if $X\sub St(\dD)$ is infinite, then there exists an infinite subset $Y\sub X$ such that $\ol{Y}^{St(\dD)}$ is homeomorphic to $\bo$;
    \item $St(\dD)$ has the fsJNP;
    \item every fsJN-sequence $\seqn{\mu}$ on $St(\dD)$ has the property that $\lim_{n\to\infty}\big|\supp\big(\mu_n\big)\big|=\infty$.
\end{enumerate}
\end{proposition}
\begin{proof}
For (1)--(3), see \cite[Section 4, Fact 1--3, page 3026]{BKS19}. We
now prove (4), so for the sake of contradiction let us assume that
there exists an fsJN-sequence $\seqn{\mu_n}$ on $St(\dD)$ and  an
integer $M>1$ such that $\big|\supp\big(\mu_n\big)\big|=M$ for every
$n\io$. By Theorem \ref{theorem:sizes_of_supps}, we may assume that
$\mu_n=\frac{1}{2}\big(\delta_{x_n}-\delta_{y_n}\big)$. By Lemma
\ref{lemma:disjoint_supps_size_2}, we may also assume that
$\big\{x_n,y_n\big\}\cap\big\{x_{n'},y_{n'}\big\}=\emptyset$ for
every $n\neq n'\io$ and that $p\not\in\big\{x_n,y_n\big\}$ for every
$n\io$. We need to consider several cases:
\begin{enumerate}[(i)]
    \item There is $Q\in\cso$ such that $\big\{x_n,y_n\big\}\sub\omega$ for every
     $n\in Q$. We then go to a subsequence $\seqk{n_k\in Q}$ such that $A=\bigcup_{k\io}\big\{x_{n_k},y_{n_k}\big\}\in\zZ$. Since $[A]_\dD$ is homeomorphic to $\bo$, it follows that $\seqk{\mu_{n_k}\rstr[A]_\dD}$ gives rise to an fsJN-sequence in $\bo$, which is impossible.
    \item There is $Q\in\cso$ such that   $\big\{x_n,y_n\big\}\cap\omega=\emptyset$ for
     every $n\in Q$. We find $A_n\in\zZ$ such that $\big\{x_n,y_n\big\}\sub\big[A_n\big]_\dD$
     for every $n\in Q$. By \cite[Section 4, Fact 1, page 3026]{BKS19}, there is infinite
      $B\in\zZ$ such that $A_n\sm B$ is finite for every $n\in Q$. Since
       $\big\{x_n,y_n\big\}\cap\omega=\emptyset$ for every $n\in Q$, it follows that
        $A_n\sm B\not\in x_n$ and $A_n\sm B\not\in y_n$, and hence
        $\big\{x_n,y_n\big\}\sub[B]_\dD$. Again, since $[B]_\dD$ is homeomorphic to $\bo$, we
        obtain an fsJN-sequence on $\bo$, which is a contradiction.
    \item There is $Q\in\cso$ such that $\big|\big\{x_n,y_n\big\}\cap\omega\big|=1$ for every $n\in Q$. Without loss of generality, we may assume that $x_n\io$ for every $n\in Q$. First, let us find $R\in[Q]^\omega$ such that $\big\{x_n\colon\ n\in R\big\}\in\zZ$. Then, similarly as in (ii), let us find $B\in\zZ$ such that $\big\{y_n\colon\ n\in R\big\}\sub[B]_\dD$. Since $\zZ$ is an ideal, $C=\big\{x_n\colon\ n\in R\big\}\cup B\in\zZ$. It follows that $[C]_\dD$ is homeomorphic to $\bo$ and $\seq{\mu_n\rstr[C]_\dD}{n\in R}$ is an fsJN-sequence on $[C]_\dD$, a contradiction.
\end{enumerate}
\end{proof}

\begin{example}\label{example:schachermayer}
In \cite[Example 4.10]{Sch82} Schachermayer provided a simple example of a Boolean algebra $\sS$ without the Grothendieck property (see Section \ref{section:grothendieck}) and such that $St(\sS)$ does not have any non-trivial convergent sequences. The algebra $\sS$ is defined as a subalgebra of $\wo$ in the following way:
\[\sS=\big\{A\in\wo\colon\ \big(\forall^\infty k\io\big)\ \big(2k\in A\equiv 2k+1\in A\big)\big\}.\]
Since $\{n\}\in\sS$ for every $n\io$, we may---as previously---identify isolated points of the Stone space $St(\sS)$ with the elements of $\omega$, so $St(\sS)$ is a compactification of $\omega$. $St(\sS)$ has the fsJNP---indeed, we can define an fsJN-sequence as follows:
\[\mu_n=\frac{1}{2}\big(\delta_{2n}-\delta_{2n+1}\big),\ n\io.\]
Note that for every $A\in\sS$ we have $\mu_n(A)=0$ for sufficiently large $n\io$.
\end{example}

The next proposition is similar to \ref{prop:ex_plebanek}.

\begin{proposition}\label{prop:ex_schachermayer}
Schachermayer's Boolean algebra $\sS$ has the following properties:
\begin{enumerate}
    \item $St(\sS)$ does not have any non-trivial convergent sequences;
    \item if $X\sub St(\sS)$ is infinite, then there exists an infinite subset
    $Y\sub X$ such that $\ol{Y}^{St(\dD)}$ is homeomorphic to $\bo$;%if $A\in\sS$, then either $A$ is finite or there exists a closed set $L\sub[A]_\sS$ homeomorphic to $\bo$;
    \item there exists an fsJN-sequence $\seqn{\mu_n}$ such that $\big|\supp\big(\mu_n\big)\big|=2$ for every $n\io$.
\end{enumerate}
\end{proposition}
\begin{proof}
For (1) and (3), see Example \ref{example:schachermayer}. We now show (2), so let $X\sub St(\sS)$ be infinite. We have two cases:
\begin{enumerate}[(i)]
    \item $X\cap\omega$ is infinite, so without loss of generality we may assume that $X\sub\omega$. Enumerate $X$ as a strictly increasing sequence $\seqn{a_n\in X}$ and go to a subsequence $\seqk{a_{n_k}}$ such that $a_{n_{k+1}}-a_{n_k}>1$ for every $k\io$. For every $k\io$ let $A_k$ be a subset of $\omega$ of size $2$ such that:
\begin{itemize}
    \item $a_{n_k}\in A_k$, and
    \item for every $l\io$ we have: $2l\in A_k$ if and only if $2l+1\in A_k$.
\end{itemize}
Then, $A_k\in\sS$. Let $U,V\in\cso$ be two disjoint sets such that $U\cup V=\omega$. It follows that $\bigcup_{k\in U}A_k\in\sS$, $\bigcup_{k\in V}A_k\in\sS$, and $\bigcup_{k\in U}A_k\cap\bigcup_{k\in V}A_k=\emptyset$, and hence finally
\[\ol{\big\{a_{n_k}\colon\ k\in U\big\}}^{St(\sS)}\sub\Big[\bigcup_{k\in U}A_k\Big]_\sS\quad\text{and}\quad\ol{\big\{a_{n_k}\colon\ k\in V\big\}}^{St(\sS)}\sub\Big[\bigcup_{k\in V}A_k\Big]_\sS,\]
so
\[\ol{\big\{a_{n_k}\colon\ k\in U\big\}}^{St(\sS)}\cap\ol{\big\{a_{n_k}\colon\ k\in V\big\}}^{St(\sS)}=\emptyset.\]
Since $U$ and $V$ were arbitrary, we get that $\ol{\big\{a_{n_k}\colon\ k\io\big\}}^{St(\sS)}$ is homeomorphic to $\bo$.
    \item $X\cap\omega$ is finite, so without loss of generality we may assume that $X\cap\omega=\emptyset$. Let $Y=\big\{x_n\in X\colon\ n\io\big\}$ be a discrete subset of $X$ and find a sequence $\seqn{A_n}$ of pairwise disjoint elements of $\sS$ such that $x_n\in\big[A_n\big]_\sS$ for every $n\io$. For each $n\io$ let $B_n\in\sS$ be such that:
\begin{itemize}
    \item $B_n\le A_n$,
    \item $A_n\sm B_n$ is finite,
    \item for every $l\io$ we have: $2l\in B_n$ if and only if $2l+1\in B_n$.
\end{itemize}
Since for each $n\io$ we have $x_n\not\in\omega$, $x_n\in B_n$. Now, notice that for every $W\in\cso$ we have $\bigcup_{n\in W}B_n\in\sS$ and proceed as in the previous case to prove that $\ol{\big\{x_n\colon\ n\io\big\}}^{St(\sS)}$ is homeomorphic to $\bo$.
\end{enumerate}
\end{proof}

\begin{remark}
Let us note that we can provide a completely different proof of Proposition \ref{prop:ex_schachermayer}.(2) in the case when $X=[A]_\sS$ for some infinite $A\in\sS$. Indeed, let $B\sub A$ be an infinite subset such that for every $k\io$ we have: $2k\in B\equiv 2k+1\in B$. Put:
\[\bB=\big\{C\iA\colon\ C\le B\big\};\]
then, $\bB$ is a Boolean algebra with obvious operations and the unit element $B$. $\bB$ is also isomorphic with $\sS$. By Koszmider and Shelah \cite[Proposition 2.5]{KS12}, $\bB$ has the so-called Weak Subsequential Separation Property, and hence, by \cite[Theorem 1.4]{KS12}, it contains an independent family $\fF$ of size $\frakc$. Let $\cC$ be the subalgebra of $\bB$ generated by $\fF$. Since $|\fF|=\frakc$, there is a homomorphism $\varphi$ from $\cC$ onto $\wo$. Since $\wo$ is complete, by the Sikorski extension theorem, there is an extension $\Phi$ of $\varphi$ onto $\bB$. By the Stone duality, it follows that $\bo$ is a subspace of $St(\bB)\approx[B]_\sS\sub[A]_\sS$.
\end{remark}

\subsection{Study of sizes of supports\label{section:study_of_sizes_of_supports}}

We will now restrict our study to those compact space which admits fsJN-sequences with bounded sizes of supports, i.e. such fsJN-sequences $\seqn{\mu_n}$ that there exists $M\io$ such that $\big|\supp\big(\mu_n\big)\big|\le M$ for every $n\io$.

\begin{proposition}\label{prop:supports_never_singletons}
Let $\seqn{\mu_n}$ be an fsJN-sequence on a compact space $K$. Then, there is $N\io$ such that for every $n>N$ the support $\supp\big(\mu_n\big)$ is not a singleton.
\end{proposition}
\begin{proof}
Assume for the sake of contradiction that there exists a subsequence $\seqk{\mu_{n_k}}$ such that $\big|\supp\big(\mu_{n_k}\big)\big|=1$ for every $k\io$. Then, for each $k\io$ there exist $x_k\in K$ and $\alpha_k\in\{-1,1\}$ such that $\mu_{n_k}=\alpha_k\delta_{x_k}$. It follows that $\big|\mu_{n_k}(K)\big|=1$ for every $k\io$, which contradicts the fact that $\big|\mu_{n_k}(K)\big|\to 0$ as $k\to\infty$.
\end{proof}

\begin{proposition}\label{prop:constant_coeffs}
Let a compact space $K$ admit an fsJN-sequence $\seqn{\mu_n}$ such that there exists $M\ge2$, $M\io$, for which we have $\big|\supp\big(\mu_n\big)\big|=M$ for every $n\io$. For each $n\io$ write $\supp\big(\mu_n\big)=\big\{x_1^n,\ldots,x_M^n\big\}$. Then, there exist $\alpha_1,\ldots,\alpha_M\in\R$ and an increasing sequence $\seqk{n_k}$  such that the measures $\nu_k=\sum_{i=1}^M\alpha_i\delta_{x_i^{n_k}}$, $k\io$, form an fsJN-sequence such that $\big\|\nu_k-\mu_{n_k}\big\|\to0$ as $k\to\infty$.
\end{proposition}
\begin{proof}
For each $n\io$ let $\mu_n=\sum_{i=1}^M\alpha_i^n\delta_{x_i^n}$. Since $\alpha_1^n\in[-1,1]$ for each $n\io$, there is $A_1\in\cso$ and $\alpha_1\in\R$ such that
\[\lim_{\substack{n\to\infty\\n\in A_1}}\alpha_1^n=\alpha_1.\]
Similarly, since $\alpha_2^n\in[-1,1]$ for each $n\in A_1$, there is $A_2\in\big[A_1\big]^\omega$ and $\alpha_2$ such that
\[\lim_{\substack{n\to\infty\\n\in A_2}}\alpha_2^n=\alpha_2.\]
We continue in this way, until we get a sequence $A_1\supset\ldots\supset A_M$ of infinite subsets of $\omega$ and a sequence of real numbers $\alpha_1,\ldots,\alpha_M\in\R$ such that
\[\lim_{\substack{n\to\infty\\n\in A_i}}\alpha_M^n=\alpha_i\]
for each $i=1,\ldots,M$. We claim that $\sum_{i=1}^M\big|\alpha_i\big|=1$. To see this, assume that $\sum_{i=1}^M\big|\alpha_i\big|=\alpha\neq 1$. Assume first that $\alpha<1$. Let $\eps=(1-\alpha)/M$. There is $N\io$ such that for every $n\in A_M\sm\{0,\ldots,N\}$ and $i=1,\ldots,M$ we have: $\big|\alpha_i^n-\alpha_i\big|<\eps$. But then for those $n$ it holds:
\[1=\sum_{i=1}^M\big|\alpha_i^n\big|\le\sum_{i=1}^M\Big(\big|\alpha_i^n-\alpha_i\big|+\big|\alpha_i\big|\Big)=\sum_{i=1}^M\big|\alpha_i^n-\alpha_i\big|+\sum_{i=1}^M\big|\alpha_i\big|<M\cdot\eps+\alpha=1,\]
a contradiction. Similarly, if $\alpha>1$, then let $\eps=(\alpha-1)/M$, and again let $N\io$ be such such that for every $n\in A_M\sm\{0,\ldots,N\}$ and $i=1,\ldots,M$ we have: $\big|\alpha_i^n-\alpha_i\big|<\eps$. But then for every such $n$ we have:
\[\alpha-1=M\cdot\eps>\sum_{i=1}^M\big|\alpha_i^n-\alpha_i\big|\ge\sum_{i=1}^M\Big(\big|\alpha_i\big|-\big|\alpha_i^n\big|\Big)=\sum_{i=1}^M\big|\alpha_i\big|-\sum_{i=1}^M\big|\alpha_i^n\big|=\alpha-1,\]
again a contradiction, so $\alpha=1$.

Enumerate $A_M=\seqk{n_k}$ and define for every $k\io$ the measure $\nu_k$ as follows:
\[\nu_k=\sum_{i=1}^M\alpha_i\delta_{x_i^{n_k}};\]
then, $\big\|\nu_k\big\|=1$, by the previous argument. %We now prove that $\seqk{\nu_k}$ is weakly* null. Let $f\in C(K)$, $f\neq0$, and $\eps>0$. For each $k\io$ we have:
%\[\big|\nu_k(f)\big|\le\big|\nu_k(f)-\mu_{n_k}(f)\big|+\big|\mu_{n_k}(f)\big|=\]
%\[\big|\sum_{i=1}^M\big(\alpha_i-\alpha_i^{n_k}\big)f\big(x_i^{n_k}\big)\big|+\big|\sum_{i=1}^M\alpha_i^{n_k}f\big(x_i^{n_k}\big)\big|\le\]
%\[\|f\|_\infty\cdot\sum_{i=1}^M\big|\alpha_i-\alpha_i^{n_k}\big|+\big|\mu_{n_k}(f)\big|.\]
%It follows that if $N\io$ is such that for every $k>N$ we have $\big|\mu_{n_k}(f)\big|<\eps/2$ and $\big|\alpha_i-\alpha_i^{n_k}\big|<\eps/\big(2M\|f\|_\infty\big)$ for every $i=1,\ldots,M$, then $\big|\nu_k(f)\big|<\eps$. This proves that $\seqk{\nu_k}$ is weakly* null.
To finish the proof notice that
\[\big\|\nu_k-\mu_{n_k}\big\|=\sum_{i=1}^M\big|\alpha_i-\alpha_i^{n_k}\big|\longrightarrow0\]
as $k\to\infty$ and appeal to Lemma \ref{lemma:norm_conv_fsjn_seqs} to conclude that $\seqk{\nu_k}$ is an fsJN-sequence on $K$.
\end{proof}

By Lemma \ref{lemma:jnseq_pos_neg}, we immediately get the following corollary.

\begin{corollary}\label{cor:constant_coeffs_neg_pos}
If $\seqn{\mu_n}$ is an fsJN-sequence on a compact space $K$ and
there exist numbers $\alpha_1,\ldots,\alpha_M\in\R$ such that every
$\mu_n$ can be written in the form
$\mu_n=\sum_{i=1}^M\alpha_i\delta_{x_i^n}$, for some
$x_1^n,\ldots,x_M^n\in K$, then:
\[\sum\big\{\alpha_i\colon\ \alpha_i>0, 1\le i\le M\big\}=1/2= -
\sum\big\{\alpha_i\colon\ \alpha_i<0, 1\le i\le M\big\}.\]
\noproof
\end{corollary}

Note that Proposition \ref{prop:constant_coeffs} does not say that $\alpha_i\neq 0$ for all $i=1,\ldots,M$, but of course we may remove from the definition of $\nu_k$ all such points $x_i^{n_k}$ for which we have $\alpha_i=0$ and obtain a sequence $\seqn{\nu_n}$ such that $\big|\supp\big(\nu_n\big)\big|<M$ for every $n\io$. The next lemma is thus a variant of Proposition \ref{prop:constant_coeffs} (with an alternative proof).

\begin{lemma}\label{lemma:supports_decrease_sizes}
Let a compact space $K$ admit an fsJN-sequence $\seqn{\mu_n}$  such
that there exists $M>2$, $M\io$, for which we have
$\big|\supp\big(\mu_n\big)\big|=M$ for every $n\io$. If there exists
a sequence $\seqn{x_n}$ such that $x_n\in\supp\big(\mu_n\big)$ for every $n\io$ and
$\lim_{n\to\infty}\mu_n\big(\big\{x_n\big\}\big)=0$, then $K$ admits
an fsJN-sequence $\seqn{\nu_n}$ such that
$\big|\supp\big(\nu_n\big)\big|=M-1$ for every $n\io$.
\end{lemma}
\begin{proof}
Let $\seqk{n_k}$ be such an increasing sequence that $\big|\mu_{n_k}\big(\big\{x_{n_k}\big\}\big)\big|<1/(k+1)$, so $\big\|\mu_{n_k}\rstr\big(K\sm\big\{x_{n_k}\big)\big\}\big\|>1-1/(k+1)$. For every $k\io$ put:
\[\nu_k=\Big(\mu_{n_k}\rstr\big(K\sm\big\{x_{n_k}\big\}\big)\Big)\Big/\big\|\mu_{n_k}\rstr\big(K\sm\big\{x_{n_k}\big\}\big)\big\|.\]
Obviously, $\big|\supp\big(\nu_k\big)\big|=M-1$ and $\big\|\nu_k\big\|=1$ for every $k\io$. We need to show that $\seqk{\nu_k}$ is weakly* null. Let $f\in C(K)$. For every $k\io$ we have:
\[\big|\nu_k(f)\big|=\big|\int_{K\sm\{x_{n_k}\}}f{\der}\mu_{n_k}\big|\Big/\big\|\mu_{n_k}\rstr\big(K\sm\big\{x_{n_k}\big\}\big)\big\|\le\]
\[\Big(\big|\mu_{n_k}(f)\big|+\big|f\big(x_{n_k}\big)\cdot\mu_{n_k}\big(\big\{x_{n_k}\big\}\big)\big|\Big)\Big/\big\|\mu_{n_k}\rstr\big(K\sm\big\{x_{n_k}\big\}\big)\big\|<\]
\[\Big(\big|\mu_{n_k}(f)\big|+\|f\|_\infty\cdot1/(k+1)\Big)\Big/\big(1-1/(k+1)\big),\]
so, since $\lim_{k\to\infty}\mu_{n_k}(f)=0$, $\lim_{k\to\infty}\nu_k(f)=0$, too. This proves that $\seqk{\nu_k}$ is weakly* null.
\end{proof}

Note that a compact space $K$ admits an fsJN-sequence $\seqn{\mu_n}$ of the form $\mu_n=\frac{1}{2}\big(\delta_{x_n}-\delta_{y_n}\big)$, where $x_n,y_n\in K$, if and only if there exist two disjoint sequences $\seqn{x_n\in K}$ and $\seqn{y_n\in K}$ such that for every $f\in C(K)$ and $\eps>0$ there exists $N\io$ such that for every $n>N$ we have $\big|f\big(x_n\big)-f\big(y_n\big)\big|<\eps$. If $K$ is totally disconnected, this observation boils down to the following one: $K$ admits an fsJN-sequence $\seqn{\mu_n}$ of the form $\mu_n=\frac{1}{2}\big(\delta_{x_n}-\delta_{y_n}\big)$, where $x_n,y_n\in K$, if and only if there exist two disjoint sequences $\seqn{x_n\in K}$ and $\seqn{y_n\in K}$ such that for every clopen set $U$ there is $N\io$ such that for every $n>N$ either $x_n,y_n\in U$ or $x_n,y_n\in U^c$. %(cf. Remark \ref{remark:tot_disc_easier_prof_sizes_2}).
These two  observations, after appropriate generalizations, are
crucial for proving Theorem \ref{theorem:sizes_of_supps}---see the
next two lemmas, where $\kK(K)$ denotes the space of all non-empty
closed subsets of a compact space $K$ endowed with the Vietoris
topology.

\begin{lemma}\label{lemma:supports_minimal_sizes_limit_points}
Let a compact space $K$ have the fsJNP and assume that $M\io$ is the
minimal natural number for which there exists an fsJN-sequence
$\seqn{\mu_n}$ such that $\big|\supp\big(\mu_n\big)\big|=M$ for
every $n\io$. For every $n\io$ put $F_n=\supp\big(\mu_n\big)$. Then,
the set $\fF=\big\{F_n\colon\ n\io\big\}$ has the following two
properties as a subset of the space $\kK(K)$:
\begin{enumerate}
    \item every limit point of $\fF$ is a singleton;
    \item $\fF$ is not closed.
\end{enumerate}
\end{lemma}
\begin{proof}
(1) By Lemma \ref{lemma:supports_decrease_sizes} and the minimality of $M$, we may assume that there exists $\eps>0$ such that for every $n\io$ and $x\in F_n$ we have $\big|\mu_n(\{x\})\big|>\eps$. By Proposition \ref{prop:supports_never_singletons}, $M>1$. Let $F\in\kK(K)$ be a limit point of $\fF$. We claim that $|F|=1$. To see this, let us suppose that $|F|>1$, so there exist distinct $x_0,x_1\in F$. %Since $[K]^{\le M}$ is a closed subspace of $\kK(K)$, $|F|\le M$.
%Using the minimality of $M$, we may assume that neither $x_0$ nor $x_1$ is isolated in $K$ TO BE CHECKED!!!.
Let $U_0$, and $U_1$ be two open subsets of $K$ such that $x_0\in U_0$, $x_1\in U_1$ and $\ol{U_0}\cap\ol{U_1}=\emptyset$. Put:
\[I=\big\{n\io\colon\ F_n\cap U_0\neq\emptyset,\ F_n\cap U_1\neq\emptyset\big\}.\]
Since $F$ is a
limit point of $\big\{F_n\colon\ n\io\big\}$, $I$ is infinite. Let
$g\in C(K)$ be a continuous function such that $0\le g\le 1$,
$g\rstr\ol{U_0}\equiv 1$ and $g\rstr\ol{U_1}\equiv 0$. For every
$n\in I$ define the measure $\theta_n$ as follows:
\[\theta_n=g{\der}\mu_n\Big/\big\|g{\der}\mu_n\big\|.\]
Then, $\seq{\theta_n}{n\in I}$ is an fsJN-sequence. Indeed, for each $n\in I$ we have $\big\|\theta_n\big\|=1$ and since $F_n\cap U_0=\supp\big(\mu_n\big)\cap U_0\neq\emptyset$, it follows that
\[\big\|g{\der}\mu_n\big\|\ge\big\|\big(g{\der}\mu_n\big)\rstr U_0\big\|=\big\|\mu_n\rstr U_0\big\|>\eps,\]
so if $f\in C(K)$, then for every $n\in I$ we have
$\theta_n(f)=\mu_n(f\cdot g)/\big\|g{\der}\mu_n\big\|$ and
\[\big|\theta_n(f)\big|=\big|\mu_n(f\cdot g)\big|\Big/\big\|g{\der}\mu_n\big\|<\big|\mu_n(f\cdot g)\big|/\eps.\]
Since
\[\lim_{\substack{n\to\infty\\n\in I}}\mu_n(f\cdot g)=0,\]
it follows that
\[\lim_{\substack{n\to\infty\\n\in I}}\theta_n(f)=0.\]
This proves that $\seq{\theta_n}{n\in I}$ is weakly* null and hence an fsJN-sequence.
Since $g\rstr\ol{U_1}\equiv 0$ and for each $n\in I$ it holds that $\supp\big(\theta_n\big)\sub\supp\big(\mu_n\big)$, it follows that $\supp\big(\theta_n\big)\subsetneq\supp\big(\mu_n\big)$, so $\big|\supp\big(\theta_n\big)\big|<M$, which is a contradiction with the assumption that $M$ is minimal. This proves that $F$ is a singleton.

\medskip

(2) By (1), each limit point of $\fF$ is a singleton, so since, by Proposition \ref{prop:supports_never_singletons}, none of the elements of $\fF$ is a singleton, $\fF$ cannot be closed.
\end{proof}

\begin{lemma}\label{lemma:supports_bounded_sizes_2}
Assume that a compact space $K$ admits an fsJN-sequence $\seqn{\mu_n}$ such that there exists $M\io$ for which we have $\big|\supp\big(\mu_n\big)\big|=M$ for every $n\io$. Then, there exists an fsJN-sequence $\seqn{\nu_n}$ such that $\nu_n=\frac{1}{2}\big(\delta_{x_n}-\delta_{y_n}\big)$ for every $n\io$, where $x_n,y_n\in K$.
\end{lemma}
\begin{proof}
Let $M$ be minimal such that there exists an fsJN-sequence $\seqn{\mu_n}$ on $K$ for which $\big|\supp\big(\mu_n\big)\big|=M$ for every $n\io$. By Proposition \ref{prop:supports_never_singletons}, $M>1$. We shall show that $M=2$.

By Lemma \ref{lemma:supports_decrease_sizes} and the minimality of $M$, we may assume that there is $\eps>0$ such that for every $n\io$ and $x\in\supp\big(\mu_n\big)$ it holds $\big|\mu_n(\{x\})\big|>\eps$. For every $n\io$ put $F_n=\supp\big(\mu_n\big)$; then, $|F_n|=M$. Let $\fF=\big\{F_n\colon\ n\io\big\}$; by Lemma \ref{lemma:supports_minimal_sizes_limit_points} every limit point of $\fF$ in the Vietoris topology of $\kK(K)$ is a singleton.

For every $n\io$ choose $x_n\neq y_n\in F_n$ and define the measure $\nu_n$ as $\nu_n=\frac{1}{2}\big(\delta_{x_n}-\delta_{y_n}\big)$. We claim that the sequence $\seqn{\nu_n}$ is weakly* null and hence an fsJN-sequence. To see this, assume that there exists $f\in C(K)$ and $\eta>0$ such that the set
\[J=\big\{n\io\colon\ \frac{1}{2}\big|f\big(x_n\big)-f\big(y_n\big)\big|>\eta\big\}\]
is infinite. Let $z\in K$ be such that $\{z\}$ is a limit point of the set $\big\{F_n\colon\ n\in J\big\}$ in $\kK(K)$. Let $U$ be a neighborhood of $z$ such that for every $x,y\in U$ we have $\big|f(x)-f(y)\big|<2\eta$. Since $\{z\}$ is a limit point of $\big\{F_n\colon n\in J\big\}$, there is $n\in J$ such that $F_n\sub U$, and hence $x_n,y_n\in U$, which is a contradiction, as $\big|f\big(x_n\big)-f\big(y_n\big)\big|>2\eta$.
\end{proof}

\begin{remark}\label{remark:tot_disc_easier_prof_sizes_2}
Let us note that if $K$ is totally disconnected, then we can prove
Lemma \ref{lemma:supports_bounded_sizes_2} without appealing to
Lemma \ref{lemma:supports_minimal_sizes_limit_points}. Indeed, let
$\seqn{\mu_n}$ and $M$ be as in Lemma
\ref{lemma:supports_bounded_sizes_2}. By Lemma
\ref{prop:constant_coeffs}, we may assume that there exist non-zero
$\alpha_1,\ldots,\alpha_M\in[-1,1]$ such that for every $n\io$ the
measure $\mu_n$ is of the form
$\mu_n=\sum_{i=1}^M\alpha_i\delta_{x_i^n}$ for some
$x_1^n,\ldots,x_M^n\in K$. Note that for every clopen set $U\sub K$
the sequences $\seqn{\mu_n\rstr U}$ and $\seqn{\mu_n\rstr U^c}$ are
weakly* null, so it follows that for sufficiently large $n\io$
either $x_1^n,\ldots,x_M^n\in U$, or $x_1^n,\ldots,x_M^n\in
U^c$---otherwise, we would get a contradiction with the minimality
of $M$. Now, the formula
$\nu_n=\frac{1}{2}\big(\delta_{x_1^n}-\delta_{x_2^n}\big)$ defines a
fsJN-sequence on $K$,  with the property that
$\big|\supp\big(\nu_n\big)\big|=2$ for every $n\io$. Since $M$ is
minimal, it follows that $M=2$.
\end{remark}

We now immediately obtain the main theorem of this section.

\begin{theorem}\label{theorem:sizes_of_supps}
Let $K$ be a compact space with the fsJNP. Then, either there is an
fsJN-sequence $\seqn{\mu_n}$ of measures on $K$ such that for every
$n\io$ the measure $\mu_n$ is of the form
$\mu_n=\frac{1}{2}\big(\delta_{x_n}-\delta_{y_n}\big)$, where
$x_n,y_n\in K$, or every fsJN-sequence $\seqn{\mu_n}$ satisfies the
equality $\lim_{n\to\infty}\big|\supp\big(\mu_n\big)\big|=\infty$.
\end{theorem}
\begin{proof}
Assume that there exists an fsJN-sequence $\seqn{\mu_n}$ and an integer $M>1$ such that $\big|\supp\big(\mu_n\big)\big|=M$ for every $n\io$, and apply Lemma \ref{lemma:supports_bounded_sizes_2}.
\end{proof}

An immediate corollary to Theorem \ref{theorem:sizes_of_supps} is the following general form of the \jn\ theorem for $C(K)$-spaces.

\begin{corollary}
Let $K$ be an infinite compact space. Then, either there is a JN-sequence $\seqn{\mu_n}$ on $K$ such that $\big|\supp\big(\mu_n\big)\big|=2$ for every $n\io$, or $\lim_{n\to\infty}\big|\supp\big(\mu_n\big)\big|=\infty$ for every JN-sequences on $K$.\noproof
\end{corollary}

Note that if a compact space $K$ admits an fsJN-sequence $\seqn{\mu_n}$ such that $\big|\supp\big(\mu_n\big)\big|=2$ for every $n\io$, then by Lemma \ref{lemma:disjoint_supps_size_2} there exists a disjointly supported fsJN-sequence on $K$ having the same property.

%In closing, let us point that it seems unclear which topological property of compact spaces is necessary for a given compact space to admit an fsJN-sequence $\seqn{\mu_n}$ such that $\big|\supp\big(\mu_n\big)\big|=2$ for every $n\io$. Thus, we pose the following problem.
%
%\begin{problem}
%Characterize those compact spaces $K$ for which there exists an fsJN-sequence $\seqn{\mu_n}$ of measures such that $\big|\supp\big(\mu_n\big)\big|=2$ for every $n\io$.
%\end{problem}

\section*{Part II. The Grothendieck property of $C(K)$-spaces}

\section{The $\ell_1$-\gr property and the fsJNP\label{section:grothendieck}}

Recall that a compact space $K$ has
\textit{the \gr property} if and only if every weakly* convergent
sequence of measures on $K$ is weakly convergent, or, in other words, the Banach space $C(K)$ is a Grothendieck space. We also say that a
Boolean algebra $\aA$ has \textit{the \gr property} if its Stone space
$St(\aA)$ has the \gr property.

It appears that the property is closely related to the finitely supported Josefson--Nissenzweig property of compact spaces. Namely, its variant, the $\ell_1$-\gr property, is equivalent to the negation of the fsJNP---in this section we shall show different approaches to this fact, starting with the issue of complementability of the Banach space $c_0$ in $C(K)$. Recall that Schachermayer \cite[Proposition 5.3]{Sch82} and Cembranos \cite[Corollary 2]{Cem84} proved that a compact space $K$ has the Grothendieck property if and only if $C(K)$ does not contain any complemented copy of $c_0$.

\begin{proposition}\label{prop:complemented_c0p_complemented_c0}
Let $K$ be a compact space such that in $C_p(K)$ there  is a
complemented closed subspace $E$ isomorphic to $(c_0)_p$. Then, $E$ with the
norm topology of $C(K)$ is complemented in $C(K)$ and isomorphic to
the Banach space $c_0$.
\end{proposition}
\begin{proof}
Let $F$ be a closed subspace of $C_p(K)$ such that $C_p(K)=E\oplus
F$. Then,  since the norm topology of $C(K)$ is finer than the
product topology of $C_p(K)$, the spaces $(E,\|\cdot\|)$ and
$(F,\|\cdot\|)$ (i.e. endowed with the inherited norm topology of
$C(K)$) are still closed in $C(K)$ and $C(K)=E\oplus F$. It is
enough now to show that $(E,\|\cdot\|)$ is isomorphic to the Banach
space $c_0$. Since $(E,\tau_p)$ (i.e. with the inherited product
topology of $C_p(K)$) is isomorphic to $(c_0)_p$, there is a
topology $\tau$ on $E$ stronger than $\tau_p$ and such that $(E,\tau)$ is isomorphic to $c_0$.
The identity operator $T\colon(E,\|\cdot\|)\to(E,\tau)$ has the
closed graph, so it is continuous, and hence $\tau$ is a Banach
space topology on $E$ smaller than the norm topology of $E$. On the
other hand, the identity operator $S\colon(E,\tau)\to(E,\|\cdot\|)$
has the closed graph, too, so it is also continuous, and hence the
topology $\tau$ on $E$ is greater than the norm topology of $E$. It
follows that the both topologies are equal, and hence
$(E,\|\cdot\|)$ is isomorphic to the Banach space $c_0$.
%Niech L bedzie dopelnialna w Cp(K) izomorficzna z c_{0} z
%topologia produktowa i niech T bedzie dopelnialna w Cp(X) do L .  Mamy
%wiec dwie podprzestrzenie L oraz T w C(K) ktore sa wzajemnie
%dopelnialne  w Cp(K) a wiec rowniez  w C(K).  Topologia Banacha
%\tau_{\infty} |L jest mocniejsza niz \tau_{p}|L, gdzie \tau_{p} jest
%topologia przestrzeni  C_{p}(K) oraz \tau_{\infty} jest topologia
%Banacha C(K). Z drugie strony  na L mamy mocniejsza topologie  \tau
%(mocniejsza niz  ta  z c_{0} z topologia produktowa)  taka ze (L,\tau)
%jest izomorficzna z Banacha c_{0}. Tak wiec mamy dwie topologie na L
%(tę  Banacha \tau_{\infty}|L oraz  \tau (oczywiscie tez mocniejsza niz
%\tau_{p}|L) . Stosujemy tw. o domknietym wykresie bo identycznosc  I
%:(L, \tau_{infty}|L) ---> (L,\tau) ma domkniety wykres, i dostajemy
%ciaglosc I.  Tak wiec obie topologie (powyzsze) sa porownywalne... a
%wiec znowu stosujemy tw. o domknietym wykresie co daje identycznosc obu
%topologii na L. Tak wiec \tau_{\infty}|L=\tau  daje informacje ze
%(L,\tau) (ktora jest c_{0}) jest dopelnialna w C(K).
\end{proof}

\begin{corollary}\label{cor:complemented_c0p_complemented_c0}
Let $K$ be an infinite compact space. If $C_p(K)$ contains a complemented copy of $(c_0)_p$, then $C(K)$ contains a complemented copy of $c_0$.\noproof
\end{corollary}

By the result of Banakh, K\k{a}kol and \'Sliwa \cite[Theorem 1]{BKS19}, mentioned in Introduction, stating that the fsJNP of a space $X$ is equivalent to the complementability of $(c_0)_p$ in $C_p(X)$, it follows that the \gr property of a compact space $K$ implies the lack of the fsJNP of $K$. Below, we provide an alternative proof of this fact (see Corollary \ref{cor:gr_no_fsJNP}) and strengthen it in Theorem \ref{theorem:ell1_grothendieck_equiv_no_fsjnp}.

\begin{definition}\label{def:ell1_gr}
A compact space $K$ has \textit{the $\ell_1$-\gr property}  (resp.
\textit{the $\Delta$-\gr property}) if and only if every weakly*
convergent sequence of measures $\seqn{\mu_n\in\ell_1(K)}$ (resp.
$\seqn{\mu_n\in\Delta(K)}$) is weakly convergent.

A Boolean algebra $\aA$ has \textit{the $\ell_1$-\gr property} (resp. \textit{the $\Delta$-\gr property}) if its Stone space $St(\aA)$ has the property.
\end{definition}

\begin{proposition}\label{prop:Delta_gr_equiv_ell1_gr}
The $\Delta$-\gr property and $\ell_1$-\gr property are equivalent.
\end{proposition}
\begin{proof}
Let $K$ be a compact space. As $\Delta(K)\sub\ell_1(K)$, the $\ell_1$-\gr property implies immediately the $\Delta$-property. Assume now that $K$ has the $\Delta(K)$-\gr property and let $\seqn{\mu_n\in\ell_1(K)}$ be weakly* convergent. For each $n\io$ find a finite set $F_n\sub K$ such that $\big\|\mu_n\rstr\big(K\sm F_n\big)\big\|<1/n$. For every $x^{**}\in C(K)^{**}$ we have:
\[\big|x^{**}\big(\mu_n\big)\big|\le\big|x^{**}\big(\mu_n\rstr F_n\big)\big|+\big|x^{**}\big(\mu_n\rstr\big(K\sm F_n\big)\big)\big|\le\big|x^{**}\big(\mu_n\rstr F_n\big)\big|+\big\|x^{**}\big\|\cdot1/n,\]
so $\lim_{n\to\infty}\big|x^{**}\big(\mu_n\big)\big|=0$, since $\lim_{n\to\infty}\big|x^{**}\big(\mu_n\rstr F_n\big)\big|=0$. This proves that $K$ has the $\ell_1$-\gr property.
\end{proof}

\begin{proposition}\label{prop:grothendieck_null_ctbl_part}
Let $K$ be a compact space with the Grothendieck property. Assume that $\seqn{\mu_n}$ is a JN-sequence on $K$. For each $n\io$ write $\mu_n=\nu_n+\theta_n$, where $\nu_n\in\ell_1(K)$ and $\theta_n$ is non-atomic. Then, $\big\|\nu_n\big\|\to0$, or equivalently $\big\|\theta_n\big\|\to1$, as $n\to\infty$.
\end{proposition}
\begin{proof}
By the \gr property, $\seqn{\mu_n}$ is weakly null, so $\lim_{n\to\infty}\mu_n(B)=0$ for every Borel set $B\sub K$. Let $L=\bigcup_{n\io}\supp\big(\nu_n\big)$. As $L$ is countable, $\lim_{n\to\infty}\nu_n(B)=\lim_{n\to\infty}\mu_n(B)=0$ for every $B\in\wp(L)$, so $\seqn{\nu_n}$ is weakly null. It follows that the sequence $\seqn{\nu_n}$ is also weakly null as a sequence of elements of the space $\ell_1(L)$. Since the space $\ell_1(L)$ has the Schur property, it follows that $\lim_{n\to\infty}\big\|\nu_n\big\|_{C(K)^*}=\lim_{n\to\infty}\big\|\nu_n\big\|_{\ell_1(L)}=0$.
\end{proof}

Since in every fsJN-sequence every measure has norm $1$ and belongs to $\ell_1(K)$, we immediately get the following corollary.

\begin{corollary}\label{cor:gr_no_fsJNP}
If a compact space $K$ has the Grothendieck property, then it does not have the fsJNP.\noproof
\end{corollary}

Corollary \ref{cor:gr_no_fsJNP} can be generalized to the following characterization of the finitely supported Josefson--Nissenzweig property.

\begin{theorem}\label{theorem:ell1_grothendieck_equiv_no_fsjnp}
Let $K$ be a compact space. Then, the following are equivalent:
\begin{enumerate}
    \item $K$ has the $\ell_1$-\gr property,
    \item $K$ has the $\Delta$-\gr property,
    \item $K$ does not have the fsJNP,
    \item $K$ does not have the csJNP.
\end{enumerate}
\end{theorem}
\begin{proof}
The equivalences (1)$\Leftrightarrow$(2) and (3)$\Leftrightarrow$(4) were proved in Propositions \ref{prop:Delta_gr_equiv_ell1_gr} and \ref{prop:fsJNP_equiv_csJNP}, respectively. We now show the equivalence (2)$\Leftrightarrow$(3). To show the implication (2)$\Rightarrow$(3) we may easily adopt the proof of Proposition \ref{prop:grothendieck_null_ctbl_part}, but we proceed differently. Namely, for the sake of contradiction, let us assume that $K$ has the fsJNP. By Theorem \ref{theorem:disjoint_supps}, there exists a disjointly supported fsJN-sequence $\seqn{\mu_n}$. Let for each $n\io$ the set $P_n$ be as in Lemma \ref{lemma:jnseq_pos_neg}. Put $P=\bigcup_{n\io}P_n$. Since $K$ has the $\Delta$-\gr property, the sequence $\seqn{\mu_n}$ is weakly null and hence it is convergent to $0$ on every Borel subset of $K$, in particular on $P$. But this contradicts Lemma \ref{lemma:jnseq_pos_neg}.

To show (3)$\Rightarrow$(2), let us assume that there exists weakly* null sequence $\seqn{\mu_n\in\Delta(K)}$ which is not weakly null. Since the weak topology is weaker than the norm topology, it follows that there exists a subsequence $\seqk{\mu_{n_k}}$ and $\eps>0$ such that $\big\|\mu_{n_k}\big\|>\eps$ for every $k\io$. But then, by Lemma \ref{lemma:normalization_fsJN}, the sequence $\seqk{\mu_{n_k}/\big\|\mu_{n_k}\big\|}$ is an fsJN-sequence on $K$, a contradiction.
\end{proof}

\section{The $\ell_1$-Grothendieck property vs. \grp\label{section:ell1_gr_no_gr}}

In his unpublished note \cite{Ple05} Plebanek constructed in ZFC a compact space $K$ such that its every separable closed subspace $L$ has the \gr property, but $K$ itself does not have the property (cf. Bielas \cite{Bie11}). It follows that $K$ is not separable, but it has the $\ell_1$-\gr property.

Following the ideas of \cite{Ple05} and Plebanek's suggestions provided in the private communication, we will construct in
this section a separable compact space---in fact, a continuous image
of $\bo$---without the \gr property, but with the $\ell_1$-\gr
property.

\begin{lemma}\label{lemma:probability_grothendieck}
Let $K$ be a totally disconnected compact space  and $\mu$ a
probability measure on $K$. Let $\seqn{A_n}$ be a sequence of clopen
mutually disjoint subsets of $K$ such that $\mu\big(A_n\big)>0$ for
every $n\io$. Define the set $F$ as follows: $x\in F$ if and only if
for every clopen neighborhood $U$ of $x$ the following inequality is
satisfied:
\[\limsup_{n\to\infty}\frac{\mu\big(A_n\cap U\big)}{\mu\big(A_n\big)}>0.\]
Then, $F$ is closed and non-empty, and the quotient space $K/F$ does
not have the Grothendieck property.
\end{lemma}
\begin{proof}
We first show that $F\neq\emptyset$. Assume for the sake of contradiction that for every $x\in K$ there exists its clopen neighborhood $U_x$ such that $\lim_n\mu\big(A_n\cap U_x\big)/\mu\big(A_n\big)=0$. By compactness of $K$, there exists a finite cover $U_{x_1},\ldots,U_{x_k}$ of $K$. We then have:
\[1=\lim_{n\to\infty}\frac{\mu\big(A_n\cap K\big)}{\mu\big(A_n\big)}\le\sum_{i=1}^k\lim_{n\to\infty}\frac{\mu\big(A_n\cap U_{x_i}\big)}{\mu\big(A_n\big)}=0,\]
a contradiction.

Let us now prove that $K/F$ does not have the Grothendieck property.
Let $\varphi\colon K\to K/F$ be the quotient map.  Denote
$p=\varphi[F]$. For every $n\io$ define a measure $\mu_n$ on $K/F$
as follows:
\[\mu_n(A)=\frac{\mu\big(A_n\cap\varphi^{-1}[A]\big)}{\mu\big(A_n\big)},\]
where $A$ is a clopen subset of $K/F$. Then, $\mu_n$ converges weakly* to $\delta_p$ on $K/F$. Indeed, if $A$ is a clopen in $K/F$ not containing $p$, then $\varphi^{-1}[A]\cap F=\emptyset$ and hence, by compactness of $\varphi^{-1}[A]$, we have $\limsup_n\mu\big(A_n\cap\varphi^{-1}[A]\big)/\mu\big(A_n\big)=0$, and so $\lim_n\mu_n(A)=0$. On the other hand, if $p\in A$, then:
\[\lim_{n\to\infty}\mu_n(A)=\lim_{n\to\infty}\Big(\mu_n(K/F)-\mu_n(A^c)\Big)=1-\lim_{n\to\infty}\mu_n(A^c)=1-0=1.\]
Had $K/F$ the Grothendieck property, $\mu_n$ would converge weakly
to $\delta_p$ and hence $\mu_n(\{p\})$ would converge to $1$, which
is not the case, since $A_n\cap F=\emptyset$ for every $n\io$ as
elements of $\seqn{A_n}$ are mutually disjoint.
\end{proof}

\begin{lemma}\label{lemma:z_pseudo_intersection}
Let $K$ be an extremely disconnected compact space. Let $\mu$,
$\seqn{A_n}$ and $F$  be like in Lemma
\ref{lemma:probability_grothendieck}. Let $\zZ$ denote the family of
all clopen subsets $C$ of $K$ such that
$\lim_{n\to\infty}\frac{\mu\big(A_n\cap
C\big)}{\mu\big(A_n\big)}=0,$ i.e.,
$$ \zZ=\big\{C\subset K\: :\: K\mbox{ is clopen and }A\cap F=\emptyset\big\}. $$
 Then $\zZ$  has the following pseudo-intersection-like
property: for every sequence $\seqn{C_n}$ of elements in $\zZ$ there exists
$C\in\zZ$ such that \[\forall n\in\omega\: \exists m\in\omega\:
(C_n\setminus C\subset\bigcup_{j\leq m}A_j).\]
\end{lemma}
\begin{proof}
We shall show only the latter property, since the remaining two are obvious. The proof is now similar to the standard one showing that the density ideal on $\omega$ is a P-ideal. Namely, inductively find a strictly increasing sequence $\seqk{n_k}$ of indices such that
\[\frac{\mu\Big(A_n\cap\bigcup_{i=0}^{k}C_i\Big)}{\mu\big(A_n\big)}<\frac{1}{k+1}\]
for every $n>n_k$ and $k\io$. Put:
\[C=\ol{\bigcup_{k\io}\Big(C_k\sm\bigcup_{j=0}^{n_k}A_j\Big)}.\]
Since $K$ is extremely disconnected, $C$ is a clopen set. It
follows easily that for every $k\io$ we have:
\[C_k\sm C\sub\bigcup_{j=0}^{n_k}A_j.\]
We shall now show that $C\in\zZ$. Fix $n\io$, $n>n_0$, and let $k\io$ be such that $n_k<n\le n_{k+1}$. We have:
\[\frac{\mu\big(A_n\cap C\big)}{\mu\big(A_n\big)}\le\frac{\mu\Big(A_n\cap\bigcup_{i=0}^{k}C_i\Big)}{\mu\big(A_n\big)}+\frac{\mu\big(A_n\cap D\big)}{\mu\big(A_n\big)}<\frac{1}{k+1}+\frac{\mu\big(A_n\cap D\big)}{\mu\big(A_n\big)},\]
where:
\[D=C\sm\Big(\bigcup_{i\io}\big(C_i\sm\bigcup_{j=0}^{n_i}A_j\big)\Big),\]
so it is just the set added to
$\bigcup_{i\io}\big(C_i\sm\bigcup_{j=0}^{n_i}A_j\big)$ after taking
the closure. Since $k\to\infty$ if $n\to\infty$, it is enough to
show that $A_n\cap D=\emptyset$,  whence $\mu\big(A_n\cap
D\big)/\mu\big(A_n\big)=0$. Assume to the contrary that there exists
$x\in A_n\cap D$. Since $x\in D$, every neighborhood of $x$
intersects  $C_k\setminus\bigcup_{j=0}^{n_k}A_j$ for infinitely many
$k$. In particular, $A_n$ must intersect
$C_k\setminus\bigcup_{j=0}^{n_k}A_j$ for some $k$ with $n_k>n$,
which is impossible.
\end{proof}

\begin{theorem}\label{theorem:extr_disc_grothendieck}
For every extremely disconnected compact space $K$ there exists a compact space $L$ and a continuous surjection $\varphi\colon K\to L$ such that $L$ does not have the Grothendieck property but it has the $\ell_1$-Grothendieck property.
\end{theorem}
\begin{proof}
Let $\mu$, $\seqn{A_n}$, $F$ and $\zZ$ be like in Lemmas
\ref{lemma:probability_grothendieck} and
\ref{lemma:z_pseudo_intersection}. Note that since $K$ is not scattered we may assume that $\mu$ vanishes on points (see \cite[Theorem 19.7.6]{Sem71}). Put
$L=K/F$ and let $\varphi$ be the quotient map. It follows from Lemma
\ref{lemma:probability_grothendieck} that $L$ does not have
Grothendieck property. For the sake of contradiction assume that $L$
does not have the $\ell_1$-Grothendieck property either, so there is
a disjointly supported JN-sequence $\seqn{\mu_n}$ on $L$. We may
assume that $\varphi[F]\cap\supp\big(\mu_n\big)=\emptyset$ for every
$n\io$, and hence  for every $n\io$ we can find $C_n\in\zZ$ such
that $\supp\big(\mu_n\big)\sub\varphi\big[C_n\big]$. Let $C\in\zZ$
be like in Lemma \ref{lemma:z_pseudo_intersection} for the sequence
$\seqn{C_n}$, i.e., $C_n\setminus C\subset\bigcup_{j=0}^{m_n}A_j$
for some increasing number sequence $\seqn{m_n}$. By
Proposition~\ref{prop:grothendieck_null_ctbl_part},
 the Grothendieck property of $K$ (and hence of
$\varphi[C]$) yields
\[\lim_{n\to\infty}\big\|\mu_n\rstr\varphi[C]\big\|=0,\]
which together with $\|\mu_n\rstr\varphi[C_n]\big\|=1$ and
$C_n\setminus C\subset\bigcup_{j=0}^{m_n}A_j$ gives
\[\lim_{n\to\infty}\big\|\mu_n\rstr\bigcup_{j=0}^{m_n}\varphi\big[A_j\big]\big\|=1.\]
On the other hand, since  for every $Q\in\fso$ we have:
\[\lim_{n\to\infty}\big\|\mu_n\rstr\bigcup_{j\in Q}\varphi\big[A_j\big]\big\|=0,\]
it follows that there exists a subsequence $\seqk{\mu_{n_k}}$ and a
sequence $\seqk{Q_k\in\fso}$ of pairwise disjoint sets such that for
every $k\io$ we have:
\[\big\|\mu_{n_k}\rstr\bigcup_{j\in Q_k}\varphi\big[A_j\big]\big\|>1/2.\]
Let $B_k$ be a clopen subset of $\bigcup_{j\in Q_k}A_j$ containing
$\supp\mu_{n_k}\cap\bigcup_{j\in Q_k}A_j$ and such that
$\frac{\mu(B_k\cap A_j)}{\mu(A_j)}<1/2^k$ for all $j\in Q_k$ (this
is the only place where we use that $\mu$ vanishes on points). Set
$D=\overline{\bigcup_{k\in\omega}B_k}$ and note that $D$ is a clopen
subset of $K$ such that $D\cap A_j=B_k\cap A_j$ for all $k\in\omega$
and $j\in Q_k$, and $D\cap A_j=\emptyset $ for
$j\in\omega\setminus\bigcup_{k\in\omega}Q_k$. Thus
\[\lim_{j\to\infty}\frac{\mu(A_j\cap D)}{\mu(A_j)}=0,\]
which means $D\in\zZ$, i.e., $D\cap F=\emptyset$. It follows from
the above that
\[\big\|\mu_{n_k}\rstr\varphi[D]\big\|=\big\|\mu_{n_k}\rstr\bigcup_{j\in Q_k}\varphi\big[A_j\big]\big\|>1/2\]
for every $k\io$, which is a contradiction, since $K$ (and hence
$\varphi[D]$) has the Grothendieck property.
\end{proof}

Considering $K=\bo$ we obtain the following important corollary.

\begin{corollary}
There exists a separable compact space $L$ such that it does not have the Grothendieck property but it has the $\ell_1$-Grothendieck property. \noproof
\end{corollary}

By the Stone duality we also obtain the following corollary saying that every complete Boolean algebra may be slimmed down in such a way that it loses the Grothendieck property but preserves the $\ell_1$-Grothendieck property.

\begin{corollary}\label{cor:grothendieck_subalgebras}
%For every $\sigma$-complete Boolean algebra $\aA$ there exist subalgebras $\cC\sub\bB\sub\aA$ such that 1) $\aA$ has the Grothendieck property, 2) $\bB$ does not have the Grothendieck property but has the $\ell_1$-Grothendieck property, and 3) $\cC$ is countable so it does not have the $\ell_1$-Grothendieck property.
For every $\sigma$-complete Boolean algebra $\aA$ there exists a subalgebra $\bB\sub\aA$ such that $\bB$ does not have the Grothendieck property but it has the $\ell_1$-Grothendieck property. \noproof
\end{corollary}

Analyzing the proof of Theorem \ref{theorem:extr_disc_grothendieck}, it seems that the theorem might also hold for those totally disconnected compact spaces which are the Stone spaces of Boolean algebras with e.g. Haydon's Subsequential Completeness Property (see Haydon \cite{Hay81}). We do not know however how far the $\sigma$-completeness of $\aA$ can be weakened in Corollary \ref{cor:grothendieck_subalgebras}, which motivates the following question.

\begin{question}
Let $\aA$ be a Boolean algebra with the Grothendieck property. Does there exist a Boolean subalgebra $\bB$ of $\aA$ which fails to have the Grothendieck property but which nonetheless has the $\ell_1$-Grothendieck property?
\end{question}

\section{The \nik property\label{section:nikodym_property}}

A property closely related to the Grothendieck property is the Nikodym property defined for Boolean algebras or, equivalently, for totally disconnected compact spaces as follows.

\begin{definition}
A Boolean algebra $\aA$ has \textit{the \nik property} if every pointwise convergent sequence of measures on $\aA$ is also weakly* convergent.
\end{definition}

Equivalently, a Boolean algebra $\aA$ has the \nik property if every pointwise bounded sequence of measures on $\aA$ is also uniformly bounded, see Sobota and Zdomskyy \cite[Proposition 2.4]{SZ19}. (Recall that a sequence $\seqn{\mu_n}$ of measures on a Boolean algebra $\aA$ is \textit{pointwise convergent} to a measure $\mu$ on $\aA$ if $\lim_{n\to\infty}\mu_n(A)=\mu(A)$ for every $A\iA$, \textit{pointwise bounded} if $\sup_{n\io}\big|\mu_n(A)\big|<\infty$ for every $A\iA$, and \textit{uniformly bounded} if $\sup_{n\io}\big\|\mu_n\big\|<\infty$.)

Nikodym \cite{Nik33} (see also Dieudonn\'e \cite{Die51}, Darst \cite{Dar67} or Rosenthal \cite{Ros70}) proved that every $\sigma$-complete Boolean algebra has the \nik property. Later on, many weakenings of $\sigma$-completeness of Boolean algebras also implying the property have been found, see e.g. Seever \cite{See68}, Molt\'o \cite{Mol81}, Freniche \cite{Fre84_vhs}, Aizpuru \cite{Aiz92} etc. The property has been also studied in the context of topological vector spaces, see e.g. Valdivia \cite{Val79,Val13} or K\k{a}kol and L\'opez-Pellicer \cite{KLP17}.

It seems hard to distinguish the \nik property  from the \gr property. So far, only a few examples of Boolean algebras having only one of the properties have been found. Schachermayer \cite[Propositions 3.2 and 3.3]{Sch82} proved that the Jordan algebra $\jJ$ of Jordan measurable subsets of the interval $[0,1]$ (i.e. such subsets $A$ of $[0,1]$ that have boundary of Lebesgue measure $0$: $\lambda(\partial A)=0$) has the \nik property, but lacks the \gr property (see also Graves and Wheeler \cite{GW83} for generalizations). Recently, an example of a minimally generated Boolean algebra with similar properties has been also obtained under the set-theoretic assumption of the Diamond Principle $\diamondsuit$ by Sobota and Zdomskyy \cite{SZ19_min_gen}. On the other hand, Talagrand \cite{Tal84}, assuming the Continuum Hypothesis (CH, in short), constructed a Boolean algebra with the \gr property, but without the \nik property (no ZFC example is known). It seems thus natural to ask about the relation of the \nik property and the fsJNP, however it appears that both properties are independent of each other. Indeed, we have the following examples:
\begin{enumerate}
    \item $\sigma$-complete Boolean algebras have the \nik property, but they lack the fsJNP (since they have the \gr property);
    \item Talagrand's example is (consistently) an example of a Boolean algebra without the fsJNP and without the \nik property;
    \item Schachermayer proved that $\jJ$ has the \nik property, but lacks the \gr property---in fact, in \cite[Proposition 3.2]{Sch82} he proves that $\jJ$ does not have the $\ell_1$-\gr property, so $\jJ$ has the fsJNP;
    \item Schachermayer's algebra $\sS$ (see Example \ref{example:schachermayer}) has the fsJNP, but it does not have the \nik property, as well as its Stone space $St(\sS)$ does not contain non-trivial convergent sequences;
    \item Boolean algebras whose Stone spaces have non-trivial convergent sequences are examples of Boolean algebras with the fsJNP, but without the \nik property.
\end{enumerate}
The following table summarizes the above points.
\begin{table}[h]\begin{tabular}{|c|c|c|c|}
\hline
\textbf{Example} & \textbf{fsJNP} & \textbf{the \nik property} & \textbf{conv. sequences}\\ \hline
\textit{$\sigma$-complete Boolean algebras} & no & yes & no\\ \hline
\textit{Talagrand's example (under CH)} & no & no & no\\ \hline
\textit{Jordan algebra $\jJ$} & yes & yes & no\\ \hline
\textit{Schachermayer's example $\sS$} & yes & no & no\\ \hline
\textit{$St(\aA)$ contains convergent sequences} & yes & no & yes\\ \hline
\end{tabular}\end{table}

However, as the next two propositions show, in some cases the lack of the Nikodym property implies the fsJNP and \textit{vice versa}. The proofs are straightforward (in the second statement of Proposition \ref{prop:2fsjnp_no_nik} and Corollary \ref{cor:nik_no_gr} we need to appeal to Theorem \ref{theorem:sizes_of_supps}).

\begin{proposition}\label{prop:no_nik_fsjnp}
Let $\aA$ be a Boolean algebra such that there exists a sequence $\seqn{\mu_n}$ of finitely supported measures on $St(\aA)$ which is pointwise convergent on $\aA$ but not uniformly bounded. Then, the sequence $\seqn{\mu_n/\big\|\mu_n\big\|}$ is an fsJN-sequence on $St(\aA)$, so $St(\aA)$ has the fsJNP.

In other words, if there is a sequence of finitely supported measures on a Boolean algebra $\aA$ witnessing the lack of the \nik property, then $\aA$ does not have the $\ell_1$-\gr property, too. \noproof
\end{proposition}

\begin{proposition}\label{prop:2fsjnp_no_nik}
Let $\aA$ be a Boolean algebra such that $St(\aA)$ admits an fsJN-sequence $\seqn{\mu_n}$ of the form $\mu_n=\frac{1}{2}\big(\delta_{x_n}-\delta_{y_n}\big)$. Then, the sequence $\seqn{n\cdot\mu_n}$ is pointwise convergent but not uniformly bounded.

In other words, if there is an fsJN-sequence on $St(\aA)$ with
bounded  sizes of supports, then $\aA$ has neither the \gr property, nor the \nik property. \noproof
\end{proposition}

\begin{corollary}\label{cor:nik_no_gr}
If a Boolean algebra $\aA$ has the Nikodym property but not the $\ell_1$-Grothendieck property, then for every fsJN-sequence $\seqn{\mu_n}$ on $St(\aA)$ we have $\lim_{n\to\infty}\big|\supp\big(\mu_n\big)\big|=\infty$. \noproof
\end{corollary}

In Section \ref{section:sizes_two_examples}, we mentioned that the density Boolean algebra $\dD$ and Schachermayer's algebra $\sS$ have the following properties:
\begin{itemize}
    \item their Stone spaces do not have any non-trivial convergent sequences;
    \item every infinite subset of their Stone spaces contains a subset $Y$ such that $\ol{Y}$ is homeomorphic to $\bo$;
    \item they have the fsJNP;
    \item every fsJN-sequence on $St(\dD)$ has supports with cardinalities convergent to $\infty$, while $St(\sS)$ has an fsJN-sequence with supports of size $2$.
\end{itemize}
As said above, the Jordan algebra $\jJ$ has the Nikodym property, so its Stone space lacks any non-trivial convergent sequences, and it has the fsJNP. Schachermayer \cite[Proposition 3.11]{Sch82} provided also a proof that if $X$ is an infinite subset of $St(\jJ)$, then there exists a subset $Y$ of $X$ such that $\ol{Y}^{St(\jJ)}$ is homeomorphic to $\bo$. %He used the notion of so-called localization points: a point $t\in[0,1]$ is \textit{a localization point} of an ultrafilter $\uU$ in $\jJ$ if for every open neighborhood $A\in\jJ$ of $t$ it holds also $A\in\uU$. It follows by the compactness of $[0,1]$ that every ultrafilter in $\jJ$ has a unique localization point. Schachermayer then asserted in the proof of \cite[Proposition 3.11]{Sch82} that if we have a non-trivial sequence $\seqn{x_n\in St(\jJ)}$, then there exist sequences $\seqk{t_k\in[0,1]}$ and $\seqk{n_k\io}$ such that each $t_k$ is the localization point of $x_{n_k}$ and $\seqk{t_k}$ converges to some $t\in[0,1]$. However, it is possible that the obtained sequence $\seqk{t_k}$ is trivial and in this case it would be impossible to construct an antichain $\seqk{\tilde{A}_k}$ of clopen subsets of $St(\jJ)$ as it is described in Schachermayer's proof and on which the rest of the proof is based. That such a situation may occur follows actually from the conclusion of the proposition---there are only $\frakc$ points in $[0,1]$ and $2^\frakc$ points in $\bo\sub St(\jJ)$, so there are $2^\frakc$ ultrafilters in $\jJ$ with the same localization point.
Regarding sizes of supports of fsJN-sequences on $St(\jJ)$, by Corollary \ref{cor:nik_no_gr}, $\jJ$ must necessarily have the same property as $\dD$.

\begin{proposition}
Let $\seqn{\mu_n}$ be an fsJN-sequence on $St(\jJ)$. Then, $\lim_{n\to\infty}\big|\supp\big(\mu_n)\big|=\infty$. \noproof
\end{proposition}
%\begin{proof}
%Assume for the sake of contradiction that there exists an fsJN-sequence $\seqn{\mu_n}$ on $St(\jJ)$ and $M\io$ such that $\big|\supp\big(\mu_n\big)\big|=M$. By Theorem \ref{theorem:sizes_of_supps}, we may assume that $\mu_n=\frac{1}{2}\big(\delta_{x_n}-\delta_{y_n}\big)$ for every $n\io$, where $x_n,y_n\in St(\jJ)$, and, moreover, by Lemma \ref{lemma:disjoint_supps_size_2}, that $\big\{x_n,y_n\big\}\cap\big\{x_{n'},y_{n'}\big\}=\emptyset$ for every $n\neq n'\io$. TBC!!!!!!!!!!!!!!!!!!
%\end{proof}

\section*{Part III. Examples of classes of spaces with the fsJNP}

\section{Systems of simple extensions and the fsJNP}

In this section we will show, combining several already known results, that the limit of every inverse system of simple extensions of compact spaces has the fsJNP (Theorem \ref{theorem:inverse_fsjnp}). This yields a corollary that many consistent examples of Efimov spaces from the literature (e.g. \cite{Fed76}, \cite{DPM09}, \cite{DS13}), constructed under such axioms as the Continuum Hypothesis or Martin's axiom, have the fsJNP as well. In Subsection \ref{section:tau_simple_ext} we will generalize this result.%---however, this time, methods will be more complicated and fundamental.

\subsection{Systems of simple extensions\label{section:simple_ext}}

%A measure $\mu$ on a compact space $K$ is \textit{a probability measure} if $\mu$ is non-negative and $\|\mu\|=\mu(K)=1$.

We start this subsection with recalling what an inverse system of simple extensions is. For general information on limits of inverse systems, see Engelking \cite[Chapters 2.5 and 3.2]{Eng89}.

\begin{definition}\label{def:inv_sys_simple_ext}
An inverse system $\big\langle K_\alpha,\pi_\alpha^\beta\colon\alpha\le\beta\le\delta\big\rangle$ of totally disconnected compact spaces is \textit{a system of simple extensions} if
\begin{itemize}
    \item it is \textit{continuous}, i.e. for every limit ordinal $\gamma\le\delta$ the space $K_\gamma$ is the limit of the inverse system $\big\langle K_\alpha,\pi_\alpha^\beta\colon\alpha\le\beta\le\gamma\big\rangle$,
    \item $K_0=2^\omega$ and each $K_\alpha$ is \textit{perfect}, i.e. has no isolated points,
    \item for every $\alpha<\delta$ the space $K_{\alpha+1}$ is \textit{a simple extension of }$K_\alpha$, i.e. there is $x_\alpha\in K_\alpha$ such that $\Big|\big(\pi_\alpha^{\alpha+1}\big)^{-1}\big(x_\alpha\big)\Big|=2$ and for every $y\in K_\alpha\sm\big\{x_\alpha\big\}$ it holds $\Big|\big(\pi_\alpha^{\alpha+1}\big)^{-1}(y)\Big|=1$.
\end{itemize}
\end{definition}

To prove the main result of this section, Theorem \ref{theorem:inverse_fsjnp}, we also need to provide several definitions concerning complexity of probability measures on compact spaces.

\begin{definition}
\textit{The Maharam type} of a probability measure $\mu$ on a compact space $K$ is the minimal cardinality of a family $\cC$ of Borel subsets of $K$ such that for every Borel subset $B$ of $K$ and $\eps>0$ there exists $C\in\cC$ such that $\mu(B\triangle C)<\eps$.
\end{definition}

\noindent Equivalently, the Maharam type of a probability measure $\mu$ is the density of the Banach space $L_1(\mu)$ of all $\mu$-integrable functions. For more information on the topic, see Maharam \cite{Mah42}, Fremlin \cite{Fre89}, or Plebanek and Sobota \cite{PS14}.

A notion closely related to the countable Maharam type is the uniform regularity, introduced by Babiker \cite{Bab77} and later studied by Pol \cite{Pol82} and Mercourakis \cite{Mer96}; see also Krupski and Plebanek \cite{KP11}.

\begin{definition}\label{def:unif_reg_measure}
A probability measure $\mu$ on a compact space $K$ is \textit{uniformly regular} if there exists a countable family $\cC$ of zero subsets of $K$ such that for every open subset $U$ of $K$ and every $\eps>0$ there exists $F\in\cC$ such that $F\sub U$ and $\mu(U\sm F)<\eps$.
\end{definition}

\noindent Uniformly regular measures are also  called
\textit{strongly countably determined} (cf. Pol \cite{Pol82}). Note
that every uniformly regular probability measure $\mu$ has
necessarily separable support and countable Maharam type.

It is an easy fact that every zero set in a normal space is a closed $\G_\delta$-set,
 thus the definition of uniformly regular measures may be stated in
 terms of closed $\G_\delta$-sets. Recall that a subset $Y$ of a  space $K$
 is \textit{a $\G_\delta$-subset} if there exists a countable collection $\uU$ of open
  subsets of $K$ such that $Y=\bigcap\uU$. An element $x\in X$ is called \textit{a $\G_\delta$-point}
  if $\{x\}$ is a   $\G_\delta$-subset of $X$.

\begin{proposition}\label{prop:unif_reg_g_delta}
Let $\mu$ be a uniformly regular measure on a compact space $K$ and $x\in K$ be such that $\mu(\{x\})>0$. Then, $x$ is a $\G_\delta$-point.
\end{proposition}
\begin{proof}
Let $\cC$ be a countable collection of zero sets (closed $\G_\delta$'s) witnessing that $\mu$ is uniformly regular. Put $\cC'=\big\{F\in\cC\colon\ x\in F\big\}$. It follows that $\{x\}=\bigcap\cC'$. To see this, assume that there is $y\in\bigcap\cC'$ such that $x\neq y$. Put $\eps=\mu(\{x\})$; so $\eps>0$. Using the regularity of $\mu$, it is easy to see that there is an open neighborhood $U$ of $x$ not containing $y$ and such that $\mu(U\sm\{x\})<\eps/3$. Note that $\eps\le\mu(U)<4\eps/3$. However, there is no $F\in\cC$ such that $F\sub U$ and $\mu(U\sm F)<\eps/3$, since otherwise $x\in F$ and hence $y\in F\in\cC'$ and $y\in U$, which is a contradiction. Since the intersection of a countable collection of $\G_\delta$-sets is $\G_\delta$, $x$ is a $\G_\delta$-point.
\end{proof}

\begin{remark}\label{remark:atom_unif_reg_conv_seq}
Note that from Proposition \ref{prop:unif_reg_g_delta} it immediately follows that if a uniformly regular measure on a compact space $K$ has an atom (i.e. it does not vanish on points), then $K$ contains a non-trivial convergent sequence.
\end{remark}

\begin{definition}
A probability measure $\mu$ on a compact space $K$  \textit{admits a
uniformly distributed sequence $\seqn{x_n\in K}$} if
$\frac{1}{n}\sum_{i=0}^{n-1}\delta_{x_i}$ converges weakly* to
$\mu$. We then say that $\seqn{x_n}$ is \textit{$\mu$-uniformly
distributed}. %\textcolor{red}{Do uniformly distributed sequences
%have to be injective? I can imagine that  $\delta_x$ for certain
%$x\in K$ is uniformly regular, but the only uniformly distributed
%sequence for it would be $<x,x,...>$? I am asking because in the
%proof of Prop.~\ref{prop:uds_fsjnp} it is used that the sequence is
%injective, otherwise the norm of $\nu_n$ might tend to 0.}
%{\color{blue}I will think about it. I guess that such $x$ should be a $\G_\delta$-point, so there is a non-trivial sequence converging to $x$, but I have to think.}
\end{definition}

Uniformly distributed sequences constitute a useful tool for
investigating various properties of probability measures as they
allow to treat those measure in a way similar to the classical
Jordan measure on the real line, see e.g. the monograph of Kuipers
and Niederreiter \cite{KN74}, Losert \cite{Los78,Los79}, or
Mercourakis \cite{Mer96}.

%The next fact, together with Proposition \ref{prop:uds_fsjnp}, will be crucial for proving Theorem \ref{theorem:inverse_fsjnp}.

Recall that a sequence $\seqn{x_n}$ in a space $X$ is \textit{injective} if $x_n\neq x_{n'}$ for every $n\neq n'\io$. The following proposition will be crucial for the proof of the main theorem of this section.

\begin{proposition}\label{prop:unif_reg_meas_unif_distr_injective_seq}
If $\mu$ is a non-atomic uniformly regular measure on a compact space $K$, then $\mu$ admits a uniformly distributed injective sequence.
\end{proposition}
\begin{proof}
Let $\mu^\infty$ be the product measure on the countable product space $K^\omega$ induced by $\mu$. \cite[Corollary 2.8]{Mer96} implies that $\mu^\infty(S)=1$, where $S$ denotes the subspace of $K^\omega$ consisting of all $\mu$-uniformly distributed sequences in $K$. Since by the non-atomicity of $\mu$ the subspace $T$ of $K^\omega$ consisting of all non-injective sequences in $K$ satisfies $\mu(T)=0$, it follows that there exists a $\mu$-uniformly distributed sequence in $S$ which is injective.
\end{proof}

\begin{proposition}\label{prop:uds_fsjnp}
Let $\mu$ be a probability measure on a compact space $K$ and $\seqn{x_n}$ be a $\mu$-uniformly distributed injective sequence. Then, $K$ has the fsJNP.
\end{proposition}
\begin{proof}
Let $\omega=\bigcup_{n\io}P_n$ be a partition of $\omega$ into finite sets such that
\[\max P_n<\min P_{n+1}\quad\text{and}\quad\big|P_n\big|\big/\max P_n\ge1/2\]
for every $n\io$ (see the classical proof that $\sum_{n\io}1/n=\infty$). For every $n\io$ let us write:
\[\nu_n=\frac{1}{\max P_{n+1}}\sum_{k\le\max P_{n+1}}\delta_{x_k}-\frac{1}{\max P_n}\sum_{k\le\max P_n}\delta_{x_k}.\]
Then, by the injectivity of $\seqn{x_n}$,
\[\big\|\nu_n\big\|\ge\frac{1}{\max P_{n+1}}\cdot\big|P_{n+1}\big|\ge1/2.\]
Since either sum in the definition of $\nu_n$'s is weakly* convergent to $\mu$, $\seqn{\nu_n}$ converges weakly* to $0$. Normalizing $\mu_n=\nu_n/\big\|\nu_n\big\|$, $\seqn{\mu_n}$ is an fsJN-sequence on $K$.
\end{proof}

The following theorem was proved by Borodulin-Nadzieja \cite{PBN07} in the language of minimally generated Boolean algebras, the dual notion to the limits of inverse systems of simple extensions, cf. Koppelberg \cite{Kop89}.

\begin{theorem}[Borodulin-Nadzieja]\label{theorem:pbn}
The following assertions hold for every totally disconnected compact space $K$:
\begin{enumerate}
    \item \cite[Theorem 4.6]{PBN07} $K$ carries either a uniformly regular measure or a measure of uncountable Maharam type;
    \item \cite[Theorem 4.9]{PBN07} If $K$ is the limit of an inverse system of simple extensions, then every measure on $K$ has countable Maharam type. In particular, there exists a uniformly regular measure on $K$.\noproof
\end{enumerate}
\end{theorem}

\noindent Let us recall here that D\v{z}amonja and Plebanek \cite[Lemma 4.1]{DP07} proved that if an inverse system of simple extensions has length at most $\omega_1$, then every measure on its limit is uniformly regular.

We are in the position to prove the main theorem of this section.

\begin{theorem}\label{theorem:inverse_fsjnp}
If $K$ is the limit of an inverse system of simple extensions, then $K$ has the fsJNP.
\end{theorem}
\begin{proof}
By Theorem \ref{theorem:pbn}.(2) there exists a uniformly regular
measure $\mu$ on $K$. If  $\mu$ has an atom, then  $K$ contains a
non-trivial convergent sequence by Proposition
\ref{prop:unif_reg_g_delta} and Remark
\ref{remark:atom_unif_reg_conv_seq}, and hence $K$ has trivially the
fsJNP. If on the other hand $\mu$ is non-atomic, then
%Theorem \ref{theorem:mer} implies that $\mu$ admits a uniformly distributed sequence and Lemma \ref{lemma:unif_distr_injective_seq} implies the existence of an injective one.
Proposition \ref{prop:unif_reg_meas_unif_distr_injective_seq} implies that $\mu$ admits a uniformly distributed injective sequence. Now, Proposition \ref{prop:uds_fsjnp} yields an fsJN-sequence on $K$.
\end{proof}

The following corollary generalizes the well-known fact that no minimally generated Boolean algebra has the Grothendieck property.

\begin{corollary}\label{cor:min_gen_no_ell_1_gr}
If $\aA$ is a minimally generated Boolean algebra, then $\aA$ does not have the $\ell_1$-\gr property.\noproof
\end{corollary}

As a corollary to Theorem \ref{theorem:inverse_fsjnp} we obtain also that many Efimov spaces constructed in the literature have the fsJNP.

\begin{corollary}\label{cor:efimov_fsjnp}
If $K$ is an Efimov space obtained as the limit of an inverse system of simple extensions, then $K$ has the fsJNP. In particular, the examples of Efimov spaces by Fedorchuk (under $\diamondsuit$; see \cite{Fed76}), Dow and Pichardo-Mendoza (under CH; see \cite{DPM09}), or Dow and Shelah (under Martin's axiom; see \cite{DS13}) have the fsJNP.\noproof
\end{corollary}

Let us note that consistently there exist Efimov spaces with the \gr property and hence without the fsJNP, see e.g. Talagrand \cite{Tal80}, Brech \cite{Bre06}, or Sobota and Zdomskyy \cite{SZ19}.

In the next section we will prove that some other classes of Efimov spaces do have the fsJNP, too.

\subsection{$\tau$-simple extensions\label{section:tau_simple_ext}}

The aim of this section is to generalize Theorem \ref{theorem:inverse_fsjnp} to a broader class of inverse systems of compact spaces (however of length at most $\frakc$). We start with the following simple observations.

\begin{lemma}\label{lemma:simple_ext_boundaries}
Let $\big\langle K_\alpha,\pi_\alpha^\beta\colon\alpha\le\beta\le\delta\big\rangle$ be an inverse system of simple extensions. For every $\alpha<\delta$ and every subset $X\sub K_{\alpha+1}$ we have $\partial\pi_\alpha^{\alpha+1}[X]\sm\pi_\alpha^{\alpha+1}[\partial X]\sub\big\{x_\alpha\big\}$.
\end{lemma}
\begin{proof}
Fix $\alpha<\delta$ and a subset $X\sub K_{\alpha+1}$. For the sake of contradiction, assume there is $x\in\partial\pi_\alpha^{\alpha+1}[X]\sm\pi_\alpha^{\alpha+1}[\partial X]$ such that $x\neq x_\alpha$. Let $V$ be a clopen subset of $K_\alpha$ such that $x_\alpha\in V$ but $x\not\in V$. Put $X'=X\sm\big(\pi_\alpha^{\alpha+1}\big)^{-1}[V]$. Since $\big(\pi_\alpha^{\alpha+1}\big)^{-1}[V]$ is closed, $x\in\partial\pi_\alpha^{\alpha+1}[X']\sm\pi_\alpha^{\alpha+1}[\partial X']$. But, as $K_{\alpha+1}\sm\pi_\alpha^{\alpha+1}[V]$ is homeomorphic to $K_\alpha\sm V$, we have $\partial\pi_\alpha^{\alpha+1}[X']=\pi_\alpha^{\alpha+1}[\partial X']$, a contradiction.
\end{proof}

\begin{lemma}\label{lemma:simple}
Let $\big\langle K_\alpha,\pi_\alpha^\beta\colon\alpha\le\beta\le\delta\big\rangle$ be an inverse system of simple extensions. Then, $\pi^\beta_\alpha$ is irreducible for any $\alpha<\beta\le\delta$.
\end{lemma}
\begin{proof}
By the continuity of this inverse system, it is enough to prove that $\pi^{\alpha+1}_\alpha$ is irreducible for every $\alpha<\delta$. But this is fairly simple, and is actually proved in the first paragraph of the proof of Proposition \ref{prop:delavega_omega_simple}---one just needs to consider the case when $\big|G_\alpha\big|=1$.
\end{proof}

Let us also note that if $\big\langle
K_\alpha,\pi_\alpha^\beta\colon\alpha\le\beta\le\delta\big\rangle$
is an inverse system of simple extensions, then
$w\big(K_\alpha\big)=w\big(K_{\alpha+1}\big)$ for every
$\alpha<\delta$. Motivated by these three
observations, we introduce the following generalization of systems
of simple extensions.

\begin{definition}\label{def:tau_simple_extensions}
Let $\tau\le\frakc$ be a cardinal number. An inverse system $\big\langle K_\alpha,\pi_\alpha^\beta\colon\alpha\le\beta\le\delta\big\rangle$ of totally disconnected compact spaces is \textit{a system of $\tau$-simple extensions} if
\begin{itemize}
    \item it is continuous,
    \item $K_0=2^\omega$ and each $K_\alpha$ is perfect, i.e. has no isolated
    points,
    \item for every $\alpha<\delta$ the space $K_{\alpha+1}$ is \textit{a $\tau$-simple extension of }$K_\alpha$, i.e. $\Big|\partial\pi_\alpha^{\alpha+1}[U]\sm\pi_\alpha^{\alpha+1}[\partial U]\Big|\le\tau$ for every closed subset $U\sub K_{\alpha+1}$,
    \item the map $\pi_\alpha^\beta$ is irreducible for every $\alpha<\beta\le\delta$,
    \item $w\big(K_\alpha\big)=w\big(K_{\alpha+1}\big)$ for every $\alpha<\delta$.
\end{itemize}
\end{definition}

Theorem \ref{theorem:tau_simple_extensions_fsjnp} states that systems of $\tau$-simple extensions of length at most $\frakc$ have the fsJNP. In order to show this, we need first to prove several technical results.

\begin{lemma}\label{lemma:image_boundary}
Let $K$ and $L$ be two compact spaces and $f\colon K\to L$ a continuous surjection. Assume that for a clopen subset $U\sub K$ the interior $\big(f[U]\cap f[K\sm U]\big)^\circ=\emptyset$. Then, $f[U]\cap f[K\sm U]=\partial f[U]\cup\partial f[K\sm U]$.
\end{lemma}
\begin{proof}
We have:
\[\partial f[U]=\ol{f[U]}\cap\ol{L\sm f[U]}=f[U]\cap\ol{L\sm f[U]}\sub f[U]\cap\ol{f[K\sm U]}=f[U]\cap f[K\sm U],\]
where the only inclusion
follows from the surjectivity of $f$ and the last
equality from the closedness of $f$. We show
similarly that $\partial f[K\sm U]\sub f[U]\cap f[K\sm U]$, whence
we get:
\[f[U]\cap f[K\sm U]\sub\partial f[U]\cup\partial f[K\sm U].\]

Let now $x\in f[U]\cap f[K\sm U]$. By the assumption, for every open neighborhood $V$ of $x$ we have $V\not\subseteq f[U]\cap f[K\sm U]$, so either $V\sm f[U]\neq\emptyset$ or $V\sm f[K\sm U]\neq\emptyset$. If for every $V$ we have $V\sm f[U]\neq\emptyset$, then $x\in\partial f[U]$. %Similarly, if for every $V$ we have $V\sm f[K\sm U]\neq\emptyset$, then $x\in\partial f[K\sm U]$.
So let us assume that there exists an open neighborhood $V$ of $x$ such that $V\sm f[U]=\emptyset$, equivalently $V\sub\big(f[U]\big)^\circ$. It follows that $x\in\partial f[K\sm U]$, since otherwise there is  an open neighborhood $W$ of $x$ such that $W\sm f[K\sm U]=\emptyset$, so $W\sub\big(f[K\sm U]\big)^\circ$, and hence:
\[x\in V\cap W\sub f[U]^\circ\cap f[K\sm U]^\circ=\big(f[U]\cap f[K\sm U]\big)^\circ,\]
a contradiction. We get thus:
\[\partial f[U]\cup\partial f[K\sm U]\sub f[U]\cap f[K\sm U].\]
\end{proof}

Proposition \ref{prop:transport_fsjn_seq}, shows how fsJN-sequences may be recovered from the Cantor space via continuous surjections. Note that if for every $n\io$, $i\in 2$ and $s\in 2^n$ we put $x_s^i=s\concat(i)$, where $(i)$ denotes the constant sequence of length $\omega$ all of whose members equal $i$, then the measures defined as
\[\mu_n=\frac{1}{2^{n+1}}\sum_{s\in 2^n}\big(\delta_{x_s^1}-\delta_{x_s^0}\big)\]
form an fsJN-sequence on the Cantor space $\Cantor$. Recall that $\lambda$ denotes the standard product measure on $\Cantor$.

\begin{proposition}\label{prop:transport_fsjn_seq}
Let $Y$ be a totally disconnected compact space and  $f\colon Y\to\Cantor$ a continuous surjection such that $\lambda\big(f[U]\cap f[Y\setminus U]\big)=0$ for every clopen $U\sub Y$. For every $n\io$, $i\in 2$ and $s\in 2^n$ fix $y_s^i\in f^{-1}\big(x^i_s\big)$ and define the measure on $Y$ as follows:
\[\nu_n=\frac{1}{2^{n+1}}\sum_{s\in 2^n}\big(\delta_{y^1_s}-\delta_{y^0_s}\big).\]
Then, the sequence $\seqn{\nu_n}$ is an fsJN-sequence on $Y$.
\end{proposition}
\begin{proof}
We need only to show that $\seqn{\nu_n(U)}$ converges to $0$ for every clopen $U\sub Y$. Fix $\eps>0$. Since $\lambda\big(f[U]\cap f[Y\setminus U]\big)=0$, it follows that $f[U]\cap f[Y\setminus U]$ has empty interior in $\Cantor$. By Lemma \ref{lemma:image_boundary}, there is a clopen set $B\sub\Cantor$ such that $\lambda(B)<\eps$ and
\[\partial f[U]\cup\partial f[Y\sm U]=f[U]\cap f[Y\setminus U]\sub B.\]
It follows that $f[U]\setminus B$ and $f[Y\setminus U]\setminus B$ are also clopen sets. Since
\[f^{-1}\big[f[U]\setminus B\big]\sub U\]
and
\[f^{-1}\big[f[Y\setminus U]\setminus B\big]\sub Y\setminus U,\]
we have that $y^i_s\in U$ if $x^i_s\in f[U]\setminus B$, and $y^i_s\in Y\setminus U$ if $x^i_s\in f[Y\setminus U]\setminus B$. Let $n_0\io$ be such that $x^0_s\in f[U]\setminus B$ if and only if $x^1_s\in f[U]\setminus B$, for all $n\ge n_0$ and $s\in 2^n$. Then,
\[\big|\nu_n(U)\big|=\Big|\frac{1}{2^{n+1}}\sum_{s\in 2^n}\big(\delta_{y^1_s}(U)-\delta_{y^0_s}(U)\big)\Big|=\]
\[\frac{1}{2^{n+1}}\Big|\sum_{s\in 2^n}\big(\delta_{y^1_s}(U)-\delta_{y^0_s}(U)\big)-\sum_{s\in 2^n}\big(\delta_{x^1_s}(f[U]\setminus B)-\delta_{x^0_s}(f[U]\setminus B)\big)\Big|=\]
\[\frac{1}{2^{n+1}}\Big|\sum_{s\in 2^n}\big(\delta_{y^1_s}(U)-\delta_{x^1_s}(f[U]\setminus B)\big)-\sum_{s\in 2^n}\big(\delta_{y^0_s}(U)-\delta_{x^0_s}(f[U]\setminus B)\big)\Big|\le\]
\[\frac{1}{2^{n+1}}\Big(\big|\sum_{f(y^1_s)\in B}\delta_{y^1_s}(U)\big|+\big|\sum_{f(y^0_s)\in B}\delta_{y^0_s}(U)\big|\Big)=\frac{1}{2^{n+1}}\Big(\big|\sum_{x^1_s\in B}\delta_{y^1_s}(U)\big|+\big|\sum_{x^0_s\in B}\delta_{y^0_s}(U)\big|\Big)\le\]
\[\frac{1}{2^{n+1}}\big(|\{s\in 2^n:x^1_s\in B\}|+|\{s\in 2^n:x^0_s\in B\}|\big)\]
for all $n\geq n_0$.
Let $n_1\io$ be such that for every $n\ge n_1$ there exists $S_n\sub 2^n$ for which $B=\bigcup_{s\in S_n}[s]$.
Then $\lambda(B)=\big|S_n\big|/2^n<\eps$ for all $n\ge n_1$. Then, for all $n\ge\max\{n_0,n_1\}$ we have:
\[\big|\nu_n(U)\big|\le\frac{1}{2^{n+1}}\Big(\big|\big\{s\in 2^n\colon\ x^1_s\in B\big\}\big|+\big|\big\{s\in 2^n\colon\ x^0_s\in B\big\}\big|\Big)\le\frac{1}{2^{n+1}}\cdot 2\cdot\big|S_n\big|=\frac{\big|S_n\big|}{2^n}<\eps,\]
which completes the proof.
\end{proof}

\begin{lemma}\label{lemma:size_boundaries_induction}
Fix three cardinal numbers $\delta,\kappa,\tau<\frakc$, where $\kappa$ is infinite. Assume that $\big\langle K_\alpha,\pi_\alpha^\beta\colon\alpha\le\beta\le\delta\big\rangle$ is an inverse system of $\tau$-simple extensions. Then, for any $\alpha<\beta\le\delta$
and every closed set $U\sub X_\beta$ such that $|\partial U|\le\kappa$, we have $\big|\partial\pi_\alpha^\beta[U]\big|\le|\beta|\cdot\tau\cdot\kappa$.
\end{lemma}
\begin{proof}
Let us first observe that the case $\beta=\alpha+1$, where $\alpha<\delta$, follows immediately from Definition \ref{def:tau_simple_extensions}, thus we need only to prove the case where $\alpha+1<\beta$. Fix $\alpha<\delta$. The proof is by induction on $\beta>\alpha$. Let us thus fix also $\beta\le\delta$ and assume that the thesis holds for every $\alpha<\xi<\beta$. We have two cases:
\begin{enumerate}
    \item $\beta=\xi+1$ for some $\alpha<\xi<\delta$. Then, by the beginning remark, $\big|\partial\pi_\xi^\beta[U]\big|\le|\beta|\cdot\tau\cdot\kappa$. By the inductive assumption used for an ordinal number $\xi$, a closed set $\pi_\xi^\beta[U]$, and the cardinal $|\beta|\cdot\tau\cdot\kappa$, we conclude that:
\[\big|\partial\pi_\alpha^\beta[U]\big|=\big|\partial\pi^\xi_\alpha\big[\pi_\xi^\beta[U]\big]\big|\le|\xi|\cdot\tau\cdot|\beta|\cdot\tau\cdot\kappa=|\beta|\cdot\tau\cdot\kappa.\]
    \item $\beta$ is limit. First note that $w\big(K_\beta\big)\le|\beta|+\omega\le|\beta|\cdot\kappa$, because the inverse system is based on $\tau$-simple extensions. It follows that $\partial U=\bigcap_{\iota<|\beta|\cdot\kappa}A_\iota$ for some family $\big\{A_\iota\colon\ \iota<|\beta|\cdot\kappa\big\}$ of clopen subsets of $K_\beta$. Then,
\[\pi^\beta_\alpha[U\setminus\partial U]=\bigcup_{\iota<|\beta|\cdot\kappa}\pi^\beta_\alpha\big[U\setminus A_\iota\big].\]
We now claim that
\[\tag{$*$}\partial\pi^\beta_\alpha[U]\sub\bigcup_{\iota<|\beta|\cdot\kappa}\partial\pi^\beta_\alpha\big[U\setminus A_\iota\big]\cup\pi^\beta_\alpha[\partial U].\]
To see this, fix $x\in\partial\pi^\beta_\alpha[U]\setminus\pi^\beta_\alpha[\partial U]$ and note that
$\big(\pi^\beta_\alpha\big)^{-1}(x)\cap U\sub U^\circ$, and hence there exists $\iota<|\beta|\cdot\kappa$ such
that $\big(\pi^\beta_\alpha\big)^{-1}(x)\cap U\sub U\setminus A_\iota$. It follows that $x\in\pi^\beta_\alpha\big[U\setminus A_\iota\big]$, and hence $x\in \partial\pi^\beta_\alpha\big[U\setminus A_\iota\big]$, because otherwise $x\in\big(\pi^\beta_\alpha\big[U\setminus A_\iota\big]\big)^\circ\sub\big(\pi^\beta_\alpha[U]\big)^\circ$, thus contradicting $x\in\partial\pi^\beta_\alpha[U]$.

Since for every $\iota<|\beta|\cdot\kappa$ the set $U\setminus A_\iota$ is clopen in $K_\beta$, for every $\iota<|\beta|\cdot\kappa$ there are $\xi_\iota\in\beta\setminus\alpha$ and clopen $B_\iota\sub K_{\xi_\iota}$ such that $U\setminus A_\iota=\big(\pi^\beta_{\xi_\iota}\big)^{-1}\big[B_\iota\big]$, and hence $\pi^\beta_\alpha\big[U\setminus A_\iota\big]=\pi^{\xi_\iota}_\alpha\big[B_\iota\big]$ for all $\iota<|\beta|\cdot\kappa$. It follows from our inductive assumption that
\[\partial\pi^\beta_\alpha\big[U\setminus A_\iota\big]=\partial\pi^{\xi_\iota}_\alpha\big[B_\iota\big]\le\big|\xi_i\big|\cdot\tau\cdot\kappa\le|\beta|\cdot\tau\cdot\kappa,\]
and hence we conclude from ($*$) that
\[\big|\partial\pi^\beta_\alpha[U]\big|\le|\beta|\cdot\kappa\cdot|\beta|\cdot\tau\cdot\kappa+\kappa\le|\beta|\cdot\tau\cdot\kappa,\]
which completes our proof.
\end{enumerate}
\end{proof}

We are in the position to prove the main theorem of this section.

\begin{theorem}\label{theorem:tau_simple_extensions_fsjnp}
Let $\tau<\frakc$ be a cardinal number. Assume that $\big\langle K_\alpha,\pi_\alpha^\beta\colon\alpha\le\beta\le\delta\big\rangle$ is an inverse system of $\tau$-simple extensions with $\delta\le\frakc$. Then, $K_\delta$ has the fsJNP.
\end{theorem}
\begin{proof}
Since $\pi^\beta_\alpha$ are irreducible for any $\alpha<\beta\leq\delta$, for every clopen $U\sub K_\beta$ we have $\big(\pi^\beta_\alpha[U]\cap\pi^\beta_\alpha\big[K_\beta\setminus U\big]\big)^\circ=\emptyset$ and hence, by Lemma \ref{lemma:image_boundary},
\[\pi^\beta_\alpha[U]\cap\pi^\beta_\alpha\big[K_\beta\setminus U\big]=\partial\pi^\beta_\alpha[U]\cup\partial\pi^\beta_\alpha\big[K_\beta\setminus U\big].\]

Let $U\sub K_\delta$ be clopen. If $\delta<\frakc$, then it follows from the above equality and Lemma \ref{lemma:size_boundaries_induction} (with $\kappa=\omega$---note that $|\partial U|=0$) that
\[\pi^\delta_0[U]\cap \pi^\delta_0\big[K_\delta\setminus U\big]\le|\delta|\cdot\tau\cdot\omega<\frakc.\]
If $\delta=\frakc$, then $U=\big(\pi^\delta_\beta\big)^{-1}[W]$ for some $\beta<\delta$ and clopen $W\sub K_\beta$, and hence, again by Lemma \ref{lemma:size_boundaries_induction},
\[\big|\pi^\delta_0[U]\cap\pi^\delta_0\big[K_\delta\setminus U\big]\big|=\big|\pi^\beta_0[W]\cap\pi^\beta_0\big[K_\delta\setminus W\big]\big|\le|\beta|\cdot\tau\cdot\omega<\frakc.\]
Thus, $\big|\pi^\delta_0[U]\cap\pi^\delta_0\big[K_\delta\setminus U\big]\big|<\frakc$ in any case. Since
$\pi^\delta_0[U]\cap \pi^\delta_0[X_\delta\setminus U]$ is a closed subset of $\Cantor$ of size $<\frakc$, we conclude that
it is countable, and hence it must have Lebesgue measure $0$. It remains to apply Proposition \ref{prop:transport_fsjn_seq} for $Y=K_\delta$.
\end{proof}

Rephrasing the theorem, we get that the limits of inverse systems of $\tau$-simple extensions of length at most $\frakc$ do not have the $\ell_1$-\gr property, which generalizes Corollary \ref{cor:min_gen_no_ell_1_gr}.

The assumption in Theorem \ref{theorem:tau_simple_extensions_fsjnp} that the ordinal number $\delta$ is not greater than $\frakc$ seems to be essential as it allows us to appeal to Proposition \ref{prop:transport_fsjn_seq} in order to ``transport'' the fsJN-sequence $\seqn{\mu_n}$ from the Cantor space onto $K_\delta$. We do not know whether the conclusion of the theorem holds true without this assumption.

\begin{question}\label{ques:tau_simple_extensions_fsjnp_delta}
Let $\tau<\frakc$ be a cardinal number. Assume that $\big\langle K_\alpha,\pi_\alpha^\beta\colon\alpha\le\beta\le\delta\big\rangle$ is an inverse system of $\tau$-simple extensions with $\delta>\frakc$. Does $K_\delta$ necessarily have the fsJNP?
\end{question}

As the application of Theorem
\ref{theorem:tau_simple_extensions_fsjnp}, we will show that some
special Efimov spaces have the fsJNP, too. Namely, in
\cite{dlV04,dlV05}, under $\diamondsuit$, de la Vega  introduced
continuous inverse systems $\big\langle
K_\alpha,\pi_\alpha^\beta\colon\alpha\le\beta\le\omega_1\big\rangle$
such that $K_{\omega_1}$ is a hereditarily separable Efimov space
with various homogeneity properties, defined as follows: for every
$\alpha<\omega_1$ the space $K_\alpha$ is homeomorphic to the space
$\Cantor$ and there exist:
\begin{itemize}
    \item closed subsets $A_\alpha^0,A_\alpha^1\sub K_\alpha$ and a point $p_\alpha\in K_\alpha$ such that $K_\alpha=A_\alpha^0\cup A_\alpha^1$ and
$A_\alpha^0\cap A_\alpha^1=\big\{p_\alpha\big\}$, and
    \item a countable group $G_\alpha$ acting on $K_\alpha$ freely, i.e. $gx\neq x$
for any $x\in K_\alpha$ and $g\in G_\alpha\setminus\big\{e_\alpha\big\}$, where $e_\alpha\in G_\alpha$ is the group identity, such that
\[K_{\alpha+1}=\big\{(x,\phi)\in K_\alpha\times 2^{G_\alpha}\colon\ x\in gA_\alpha^{\phi(g)}\text{ for every }g\in G_\alpha\big\}\]
and $\pi^{\alpha+1}_\alpha\big((x,\phi)\big)=x$.
\end{itemize}
Let us call such inverse systems \textit{de la Vega systems}. Below, we show that they are based on $\omega$-simple extensions and hence their limits $K_{\omega_1}$ have the fsJNP.

\begin{proposition}\label{prop:delavega_omega_simple}
Every de la Vega system is based on $\omega$-simple extensions.
\end{proposition}
\begin{proof}
Let $\big\langle
K_\alpha,\pi_\alpha^\beta\colon\alpha\le\beta\le\omega_1\big\rangle$
be a de la Vega system. For each $\alpha<\omega_1$ fix
$A_\alpha^0,A_\alpha^1,p_\alpha$ and $G_\alpha$ as in the definition
of the system. We need to show that each $\pi_\alpha^\beta$ is
irreducible and that for each $\alpha<\omega_1$ and closed $U\sub
K_{\alpha+1}$ the set
$\partial\pi_\alpha^{\alpha+1}[U]\sm\pi_\alpha^{\alpha+1}[\partial
U]$ is countable.

\begin{enumerate}
    \item For every $\alpha<\beta\le\omega_1$ the function $\pi_\alpha^\beta$ is irreducible.

Since the system is continuous it is enough to show that $\pi_\alpha^{\alpha+1}$ is irreducible. To prove this we will use the following simple characterization of irreducible mappings: a function $f\colon K\to L$ between two totally disconnected compact spaces is irreducible if and only if for every clopen $U\sub K$ there is clopen $B\sub L$ such that $f^{-1}[B]\sub U$.

Thus fix $\alpha<\omega_1$ and a clopen $U\subset K_{\alpha+1}$. Without loss of generality let us assume $K_\alpha=\Cantor$. Shrinking $U$ if necessary we may assume that $\emptyset\neq U=([s]\times [t])\cap K_{\alpha+1} $ for some $s=\langle s(0),\ldots,s(n)\rangle\in 2^{n+1}$ and
$t=\langle t(g_0),\ldots,t(g_n)\rangle\in 2^{\{g_i\colon\ i\leq n\}}$, where $[s]=\big\{x\in\Cantor\colon\ x\rstr (n+1)=s\big\}$
and $[t]\subset 2^{G_\alpha}$ is defined analogously. Put:
\[W=\bigcap_{i\le n}g_i A_\alpha^{t(g_i)}\cap [s].\]
Since $\pi_\alpha^{\alpha+1}[U]=W$, $W$ is non-empty, and
$W\setminus\big\{g_ip\colon\ i\le n\big\}$ is open in $K_\alpha$.
Fix any clopen set $B\subset W\setminus\{g_ip:i\leq n\}$ and a pair
$(x,\phi)\in B\times 2^{G_\alpha}$ such that $(x,\phi)\in
K_{\alpha+1}$, i.e., $(x,\phi)\in (\pi^{\alpha+1}_\alpha)^{-1}(x)$.
Since $x\not\in\big\{g_ip:i\leq n\big\}$, for every $i\le n$ there
is a unique $j_i\in 2$ such that $x\in g_i
A_\alpha^{j_i}$, and hence $\phi\big(g_i\big)=t(g_i)=j_i$ for all
$i\le n$. It follows that $(x,\phi)\in [s]\times [t]\cap
K_{\alpha+1}\subset U$, so,
summarizing,$\big(\pi^{\alpha+1}_\alpha\big)^{-1}[B]\subset U$.

    \item for every $\alpha<\omega_1$ and closed $U\sub K_{\alpha+1}$, $\Big|\partial\pi_\alpha^{\alpha+1}[U]\sm\pi_\alpha^{\alpha+1}[\partial U]\Big|\le\omega$.

We shall show that
\[\partial\pi_\alpha^{\alpha+1}[U]\sm\pi_\alpha^{\alpha+1}[\partial U]\sub\big\{gp\colon\ g\in G_\alpha\big\}.\]
This has been almost done in the previous paragraph. Indeed,
suppose that
\[x\in\pi^{\alpha+1}_\alpha[U]\setminus\big(\pi^{\alpha+1}_\alpha[\partial U]\cup\{gp\colon\ g\in G_\alpha\big\}\big).\]
Then, $x\in\pi^{\alpha+1}_\alpha[U^\circ]\setminus\big\{gp\colon\ g\in G_\alpha\big\}$, and, by the same argument as
in (1), we get a clopen set $B\subset K_\alpha$ containing $x$ and such that $\big(\pi^{\alpha+1}_\alpha\big)^{-1}(x)\sub\big(\pi^{\alpha+1}_\alpha\big)^{-1}[B]\sub U^\circ$, and therefore $x\in B\sub\big(\pi^{\alpha+1}_\alpha\big)[U]^\circ$. But this implies that $x\not\in\partial\pi^{\alpha+1}_\alpha[U]$, which completes the proof.
\end{enumerate}
\end{proof}

Recall that a space $X$ is called \textit{rigid} if it has no non-trivial autohomeomorphisms, i.e. every homeomorphism $f\colon X\to X$ is the identity. Combining Theorem \ref{theorem:tau_simple_extensions_fsjnp} with Proposition \ref{prop:delavega_omega_simple} and de la Vega's \cite[Theorems 5.1 and 5.2]{dlV04}, we get the following corollary, important in the view of K\k{a}kol and \'Sliwa \cite[Example 15]{KS18}.

\begin{corollary}\label{cor:delavega_fsjnp}
Assume $\diamondsuit$.
\begin{enumerate}
    \item There exists a hereditarily separable totally disconnected rigid Efimov space satisfying the fsJNP.
    \item There exists a hereditarily separable totally disconnected Efimov space $K$ satisfying the fsJNP and such that any two non-empty clopen subsets of $F$ are homeomorphic.
\end{enumerate}
\end{corollary}

Furthermore, using the generalizations of de la Vega \cite{dlV05} obtained in
 Back\'e \cite{Back18}, we get the next corollary.

\begin{corollary}\label{cor:backe_fsjnp}
Under $\diamondsuit$ there exists a hereditarily separable totally disconnected Efimov space $K$ satisfying the fsJNP and such that there are no disjoint infinite closed homeomorphic subspaces of $K$.
\end{corollary}

%\textcolor{red}{Do we know that there are systems of $\omega$-simple
%extensions, so that the limit space cannot be obtained with the help
%of simple extensions? Also, assuming $2^\omega>\omega_1$, or even
%MA, do we know that there are  systems of $\omega_1$-simple
%extensions, so that the limit space cannot be obtained with the help
%of simple extensions? If so not know the answer, I suggest to ask
%the corresponding question}
%
%{\color{blue} I guess we don't know anything. Perhaps we should ask someone (de la Vega? Hart? Dow?)...}

Let us not here that we do not know however whether the classes of compact spaces obtained by simple extensions and $\tau$-simple extensions for $\tau\in[\omega,\frakc]$ are essentially different. Thus, we ask the following crucial questions.

\begin{question}\label{ques:tau_simple_ext_simple_ext}
\begin{enumerate}
    \item Does there exist a compact space which is the limit of an inverse system based on $\omega$-simple extensions but not the limit of any inverse system based on simple extensions?
    \item Assume that $\frakc>\omega_1$. Does there exist a compact space which is the limit of an inverse system based on $\omega_1$-simple extensions but not the limit of any inverse system based on simple extensions?
\end{enumerate}
\end{question}

The following problem is a special case of Question \ref{ques:tau_simple_ext_simple_ext}.

\begin{question}\label{ques:dlV_simple}
Does there (consistently) exist a de la Vega system whose limit cannot be represented as the limit of an inverse system based on simple extensions?
\end{question}

\section{l-Equivalence}

It is a well-known fact that given two spaces $X$
and $Y$ if  an operator $L\colon C_p(X)\to C_p(Y)$ is a linear
homeomorphism, then the operator $L^*\colon\Delta(Y)\to\Delta(X)$ given
by the formula $L^*(\mu)=\mu\circ L$ is also a linear homeomorphism
(see Tkachuk \cite[Problem 237]{TkaVol4}). Since for every measure
$\mu\in\Delta(X)$, written
$\mu=\sum_{i=1}^n\alpha_i\cdot\delta_{x_i}$, we have
$L^*(\mu)=\mu\circ L=\sum_{i=1}^n\alpha_i\cdot\big(\delta_{x_i}\circ
L\big)$, we obtain the following proposition.

\begin{proposition}\label{prop:l_equiv_fsjnp}
Let $X$ and $Y$ be two spaces. Assume that $X$ has the fsJNP. If $C_p(X)$ and $C_p(Y)$ are linearly homeomorphic, then $C_p(Y)$ has the fsJNP, too. \noproof
\end{proposition}

Let $K$ be a compact space. Recall that \textit{the Alexandrov Duplicate} $AD(K)$ of $K$ is the space defined as follows: $AD(K)=K\times\{0,1\}$ where for each $x\in K$ the point $(x,1)$ is isolated and the basic open neighborhoods of $(x,0)$ are the sets of the form $(U\times\{0\})\cup\big((U\sm\{x\})\times\{1\}\big)$ for every open basic neighborhood $U$ of $x$ in $K$. It follows that $AD(K)$ is compact (cf. \cite[Problem 364]{TkaVol1}) and $C_p(AD(K))$ is linearly homoemorphic to $C_p\big(K\cup\alpha(|K|)\big)$, where $\alpha(|K|)$ denotes the one-point compactification of the cardinal number $|K|$ (see \cite[Problem 267]{TkaVol4}). Since $\alpha(|K|)$ contains a non-trivial convergent sequence, $K\cup\alpha(|K|)$ has the fsJNP and hence $AD(K)$ does, too.

\begin{proposition}
Let $K$ be a compact space. Then, its Alexandrov Duplicate $AD(K)$ has the fsJNP. \noproof
\end{proposition}

It follows that $AD(\bo)$ has the fsJNP, although $\bo$ does not have. This is interesting in the context of Sections \ref{section:sizes_two_examples} and \ref{section:nikodym_property}, where examples of compact spaces with the fsJNP and containing many copies of $\bo$ were given---$AD(\bo)$ is another such example but of completely different kind.

\section{Products and the JNP\label{section:products}}

Khurana \cite[Theorem 2]{Khu78}, Cembranos \cite[Corollaries 2--3]{Cem84} and Freniche \cite[Corollary 2.6]{Fre84} proved that given two infinite compact spaces $K$ and $L$ the space $C(K\times L)$ does not have the Grothendieck property. In this section we will strengthen their result by proving that the product $K\times L$ does not even have the $\ell_1$-\gr property (Theorem \ref{theorem:products_fsjnp} and Corollary \ref{cor:product_ell1_grothendieck})---our proof is totally constructive and does not require any techniques from Banach space theory. As a corollary, we obtain that the space $C_p(K\times L)$ has always the quotient isomorphic to $(c_0)_p$.

%Schachermayer \cite[Proposition 5.3]{Sch82} and Cembranos \cite[Corollary 2]{Cem84} proved that a $C(K)$-space is Grothendieck if and only if it does not contain any complemented copy of $c_0$. This implies Cembranos' result \cite[Corollary 3]{Cem84} stating that the space $C(K\times L)$ admits a complemented copy of $c_0$ for any two infinite compact spaces $K$ and $L$. Corollary \ref{cor:product_ell1_grothendieck}, together with the characterization of Banakh, K\c akol and \'Sliwa \cite[Theorem 1]{BKS19} and Proposition \ref{prop:complement_c0p_complemented_c0}, implies the following stronger result. (THAT WILL GO THE INTRODUCTION RATHER!!!!!)

%\begin{corollary}\label{cor:product_c0}
%For every two infinite compact spaces $K$ and $L$, the space $C_p(K\times L)$ admits $(c_0)_p$ as a quotient.
%\end{corollary}

\medskip

Let us start with some additional notation. For every $n\io_+$
put $\Omega_n=\{-1,1\}^n$ and $\Sigma_n=n\times\{n\}$ (so
$\big|\Omega_n\big|=2^n$ and $\big|\Sigma_n\big|=n$). To simplify
the notation, we will usually write $i\in\Sigma_n$ instead of
$(i,n)\in\Sigma_n$---this should cause no confusion. Put also
$\Omega=\bigcup_{n\io_+}\Omega_n$ and
$\Sigma=\bigcup_{n\io_+}\Sigma_n$, and endow these two sets with the
discrete topology. This way, we can think of the product space
$\Omega\times\Sigma$ as a countable union of pairwise disjoint
discrete rectangles $\Omega_k\times\Sigma_m$ of size $m2^k$---the
rectangles $\Omega_n\times\Sigma_n$, lying along the diagonal, will
bear a special meaning, namely, they will be the supports of
measures from an fsJN-sequence $\seq{\mu_n}{n\io_+}$ on the space
$\beta\Omega\times\beta\Sigma$ defined as follows ($n\io_+$):
\[\mu_n=\sum_{\substack{s\in\Omega_n\\i\in\Sigma_n}}\frac{s(i)}{n2^n}\delta_{(s,i)}.\]
Then, $\supp\big(\mu_n\big)=\Omega_n\times\Sigma_n$, so $\big|\supp\big(\mu_n\big)\big|=n2^n$,
$\big\|\mu_n\big\|=1$, and
\[\pi_i\big[\supp\big(\mu_n\big)\big]\cap\pi_i\big[\supp\big(\mu_{n'}\big)\big]=\emptyset\]
for every $n\neq n'$ and $i\in\{0,1\}$ (here $\pi_i$ denotes the projection on the $i$-th coordinate). Note that for each $n\io_+$ and any two sets $A\in\wp(\Omega)$ and $B\in\wp(\Sigma)$ we have:
\[\tag{$\dagger$}\big|\mu_n\big([A]\times[B]\big)\big|\le\frac{\big|A\cap \Omega_n\big|}{2^n}\cdot\frac{\big|B\cap \Sigma_n\big|}{n},\]
where $[A]$ and $[B]$ always denote the clopen subsets of $\beta\Omega$ and $\beta\Sigma$ corresponding in the sense of the Stone duality to $A$ and $B$, respectively---since $\beta\Omega$ and $\beta\Sigma$ are extremely disconnected, we have $[A]=\ol{A}^{\beta\Omega}$ and $[B]=\ol{B}^{\beta\Sigma}$.

\medskip

Before we state and prove the main proposition of this section, we need to provide a bit of explanation of probability tools we use in the proof. For every $n\io_+$ and $i\in n$ define the function $X_i\colon\Omega_n\to\{0,1\}$ as follows: $X_i(r)=1$ if and only if $r(i)=1$, where $r\in\Omega_n$. Put $S_n=\sum_{i=0}^{n-1}X_i$, so $S_n\colon\Omega_n\to n$ is the function computing the number of $1$'s in the argument sequence $r\in\Omega_n$. For a finite set $A\in\fso$, let $P_A$ denotes the standard product probability on $2^A$ (assigning $1/2^{|A|}$ to each elementary event, i.e. $P_A(\{r\})=1/2^{|A|}$ for each $r\in2^{|A|}$). Recall that for every $k\le n$ it holds:
\[P_n(S_n=k)=P_n\big(\big\{r\in\Omega_n\colon\ S_n(r)=k\big\}\big)={n\choose k}1/2^n.\]
We will need the following fact estimating the probability that $S_n(r)$ has value ``far'' (with respect to $\eps$) from $n/2$, i.e. that ``$r$ contains \emph{significantly} more (with respect to $\eps$) $1$'s than $-1$'s, or \textit{vice versa}''.

\begin{fact}\label{fact:bollobas}
If $n\io_+$ and $\eps\in\big(0,1/12\big]$ are such numbers that $n\ge3/\eps$, then:
\[P_n\big(\big|S_n-n/2\big|\ge\eps n/2\big)\le%\frac{\sqrt{2}}{\eps\sqrt{n}}\cdot e^{-\eps^2n/6}\le
\frac{\sqrt{2}}{\eps\sqrt{n}}.\]
\end{fact}
\begin{proof}
See Bollob\'as \cite[Theorem 1.7.(i)]{Bol01}.
\end{proof}

\medskip

We are ready to prove the main result of this section.

%\begin{theorem}\label{thm:jnp_products}
%The sequence $\seqn{\mu_n}$ defined above is a JN-sequence. Consequently, $(\bo)^2$ has the JNP.
%\end{theorem}
\begin{proposition}\label{prop:fsjnp_product_omega_sigma}
The sequence $\seq{\mu_n}{n\io_+}$ defined above is an fsJN-sequence. Consequently, $\beta\Omega\times\beta\Sigma$ has the fsJNP.
\end{proposition}
\begin{proof}
Since $\beta\Omega\times\beta\Sigma$ is a totally disconnected compact space, to prove that $\seq{\mu_n}{n\io_+}$ is weakly* null it is enough to show that it converges to $0$ on every clopen subset of the form $[A]\times[B]$, where $A\in\wp(\Omega)$ and $B\in\wp(\Sigma)$. So let us fix two such sets $A$ and $B$.% For every $n\io_+$ let $k_n=\big|B\cap\Sigma_n\big|$. % Let thus $A$ and $B$ be two fixed subsets of $\Omega$ and $\Sigma$, respectively, and, for the sake of contradiction, let us assume that $\mu_n\big([A]\times[B]\big)\ge\eta$ for some $\eta>0$ and all $n\io$ (if necessary, we go to a subsequence).

Fix $\eps\in\big(0,1/12\big]$ and put:
\[I_0=\big\{n\io_+\colon\ \big|B\cap\Sigma_n\big|<2/\eps^4\big\}\]
and
\[I_1=\omega_+\sm I_0=\big\{n\io_+\colon \big|B\cap\Sigma_n\big|\ge2/\eps^4\big\}.\]
% $n\io$ such that $n\ge2/\eps^4$ (so $n\ge3/\eps$), and put:%$n\ge48/\eta$ {\color{red}($n\ge3/\eps$)}, and put:
For each $i\in\{0,1\}$ we will find $N_i\io$ such that for every $n\ge N_i$, $n\in I_i$, it holds
\[\big|\mu_n\big([A]\times[B]\big)\big|=\big|\mu_n\big(\big[A\cap\Omega_n\big]\times\big[B\cap\Sigma_n\big]\big)\big|<2\eps.\]

\medskip

If for some $i\in\{0,1\}$ the set $I_i$ is finite, then let immediately $N_i=1+\max I_i$. If $I_0$ is infinite, then by ($\dagger$) for every $n\in I_0$ we have:
\[\big|\mu_n\big([A]\times[B]\big)\big|\le\frac{\big|A\cap \Omega_n\big|}{2^n}\cdot\frac{\big|B\cap\Sigma_n\big|}{n}\le\frac{\big|B\cap\Sigma_n\big|}{n}<\frac{2}{n\eps^4},\]
so there exists $N_0\io$ such that for every $n\ge N_0$, $n\in I_0$, we have:
\[\big|\mu_n\big([A]\times[B]\big)\big|<2\eps.\]

\medskip

Let us now assume that $I_1$ is infinite. For every $n\in I_1$ define also the set $\Delta_{n,\eps}$ as follows:
\[\Delta_{n,\eps}=\Big\{s\in\Omega_n\colon\ \Big|\big|\big\{i\in B\cap\Sigma_n\colon\ s(i)=1\big\}\big|-\frac{\big|B\cap\Sigma_n\big|}{2}\Big|\ge\eps\frac{\big|B\cap\Sigma_n\big|}{2}\Big\},\]
so $\Delta_{n,\eps}$ denotes the event that $s\in\Omega_n$ is ``far'' (with respect to $\eps$) from having the same numbers of $1$'s and $-1$'s when restricted to the set $B$. If we put similarly:
\[\Gamma_{n,\eps}=\Big\{s\in2^{B\cap\Sigma_n}\colon\ \Big|\big|\big\{i\in B\cap\Sigma_n\colon\ s(i)=1\big\}\big|-\frac{\big|B\cap\Sigma_n\big|}{2}\Big|\ge\eps\frac{\big|B\cap\Sigma_n\big|}{2}\Big\},\]
then we trivially have:
\[\tag{$\times$}\Delta_{n,\eps}=\Gamma_{n,\eps}\times2^{\Sigma_n\sm B}.\]
Using this, for every $n\in I_1$ we will estimate the values of measures (see (2) and (3)):
\[\tag{$*$}\big|\mu_n\big(\big[A\cap\Delta_{n,\eps}\big]\times[B\cap\Sigma_n]\big)\big|=\Big|\sum_{\substack{s\in A\cap\Delta_{n,\eps}\\i\in B\cap\Sigma_n}}\frac{s(i)}{n2^n}\Big|\]
and
\[\tag{$**$}\big|\mu_n\big(\big[A\cap\big(\Omega_n\sm\Delta_{n,\eps}\big)\big]\times[B\cap\Sigma_n]\big)\big|=\Big|\sum_{\substack{s\in A\cap(\Omega_n\sm\Delta_{n,\eps})\\i\in B\cap\Sigma_n}}\frac{s(i)}{n2^n}\Big|.\]
Note that:
\[\big|\mu_n\big([A]\times[B]\big)\big|\le\big|\mu_n\big(\big[A\cap\Delta_{n,\eps}\big]\times[B\cap\Sigma_n]\big)\big|+\big|\mu_n\big(\big[A\cap\big(\Omega_n\sm\Delta_{n,\eps}\big)\big]\times[B\cap\Sigma_n]\big)\big|,\]
so obtaining ``good'' estimations of ($*$) and ($**$) will finish the proof.

\medskip

Fix $n\in I_1$ and let us start with the estimation of ($*$). Note that $\big|B\cap\Sigma_n\big|\ge3/\eps$, so recall ($\times$) and apply Fact \ref{fact:bollobas} with the set $B\cap\Sigma_n$ instead of the set $n$ to get that: %$P_n\big(\Delta_{n,\eps}\big)\le\frac{\sqrt{2}}{\eps\sqrt{n}}$ and hence:
%%%%%\[\tag{$0$}P_n\big(\Delta_{n,\eps}\big)\le\frac{\sqrt{2}}{\eps\sqrt{n}}.\]
\[\tag{$0$}P_n\big(\Delta_{n,\eps}\big)=P_{B\cap\Sigma_n}\big(\Gamma_{n,\eps}\big)\cdot P_{\Sigma_n\sm B}\big(2^{\Sigma_n\sm B}\big)=P_{B\cap\Sigma_n}\big(\Gamma_{n,\eps}\big)\le\frac{\sqrt{2}}{\eps\sqrt{\big|B\cap\Sigma_n\big|}}.\]
%\[\tag{$0$}P_n\big(\Delta_{n,\eps}\big)\le\frac{\sqrt{2}}{\eps\sqrt{\eta n}},\]
%since $\eta\le 1$.
It follows that for every $s\in\Omega_n\sm \Delta_{n,\eps}$ we have:
\[\big|\sum_{i\in B\cap\Sigma_n}s(i)\big|=\Big|\big|\big\{i\in B\cap\Sigma_n\colon\ s(i)=1\big\}\big|-\big|\big\{i\in B\cap\Sigma_n\colon\ s(i)=-1\big\}\big|\Big|\le\]
\[\Big|\big|\big\{i\in B\cap\Sigma_n\colon\ s(i)=1\big\}\big|-\frac{\big|B\cap\Sigma_n\big|}{2}\Big|+\]
\[\Big|\big|\big\{i\in B\cap\Sigma_n\colon\ s(i)=-1\big\}\big|-\frac{\big|B\cap\Sigma_n\big|}{2}\Big|<\]
\[2\cdot\eps\frac{\big|B\cap\Sigma_n\big|}{2}=\eps\big|B\cap\Sigma_n\big|,\]
so:
\[\tag{$1$}\big|\sum_{i\in B\cap\Sigma_n}s(i)\big|<\eps\big|B\cap\Sigma_n\big|.\]
Finally, it holds that:
\[\Big|\sum_{\substack{s\in A\cap\Delta_{n,\eps}\\i\in B\cap\Sigma_n}}\frac{s(i)}{n2^n}\Big|\le\sum_{\substack{s\in A\cap\Delta_{n,\eps}\\i\in B\cap\Sigma_n}}\frac{1}{n2^n}=\frac{\big|A\cap\Delta_{n,\eps}\big|\cdot\big|B\cap\Sigma_n\big|}{n2^n}\le\]
\[\frac{\big|\Delta_{n,\eps}\big|\cdot n}{n2^n}=P_n\big(\Delta_{n,\eps}\big),\]
so by (0):
\[\tag{$2$}\Big|\sum_{\substack{s\in A\cap\Delta_{n,\eps}\\i\in B\cap\Sigma_n}}\frac{s(i)}{n2^n}\Big|\le\frac{\sqrt{2}}{\eps\sqrt{\big|B\cap\Sigma_n\big|}}.\]

\medskip

We estimate ($**$) in a similar way:
\[\Big|\sum_{\substack{s\in A\cap(\Omega_n\sm\Delta_{n,\eps})\\i\in B\cap\Sigma_n}}\frac{s(i)}{n2^n}\Big|\le\frac{1}{n2^n}\Big|\sum_{\substack{s\in A\cap(\Omega_n\sm\Delta_{n,\eps})\\i\in B\cap\Sigma_n}}s(i)\Big|=\]
\[\frac{1}{n2^n}\Big|\sum_{s\in A\cap(\Omega_n\sm\Delta_{n,\eps})}\ \sum_{i\in B\cap\Sigma_n}s(i)\Big|\le\frac{1}{n2^n}\sum_{s\in A\cap(\Omega_n\sm\Delta_{n,\eps})}\big|\sum_{i\in B\cap\Sigma_n}s(i)\big|\le\]
\[\frac{1}{n}\max\Big\{\big|\sum_{i\in B\cap\Sigma_n}s(i)\big|\colon\ s\in\Omega_n\sm\Delta_{n,\eps}\Big\},\]
so by (1):
\[\tag{$3$}\Big|\sum_{\substack{s\in A\cap(\Omega_n\sm\Delta_{n,\eps})\\i\in B\cap\Sigma_n}}\frac{s(i)}{n2^n}\Big|<\frac{\eps\big|B\cap\Sigma_n\big|}{n}\le\frac{\eps n}{n}=\eps.\]

\medskip

We are ready to finish the proof---using (2), (3) and the fact that $\big|B\cap\Sigma_n\big|>2/\eps^4$, we conclude that:
\[\big|\mu_n\big([A]\times[B]\big)\big|=\Big|\sum_{\substack{s\in A\cap\Omega_n\\i\in B\cap\Sigma_n}}\frac{s(i)}{n2^n}\Big|\le\Big|\sum_{\substack{s\in A\cap\Delta_{n,\eps}\\i\in B\cap\Sigma_n}}\frac{s(i)}{n2^n}\Big|+\Big|\sum_{\substack{s\in A\cap(\Omega_n\sm\Delta_{n,\eps})\\i\in B\cap\Sigma_n}}\frac{s(i)}{n2^n}\Big|<\]
\[\frac{\sqrt{2}}{\eps\sqrt{\big|B\cap\Sigma_n\big|}}+\eps<\eps+\eps=2\eps.\]
It follows that if for $N_1$ we take any number from $I_1$, then for every $n\ge N_1$ we have:
\[\big|\mu_n\big([A]\times[B]\big)\big|<2\eps.\]

We finish the proof by denoting $N=\max\big(N_0,N_1\big)$ and seeing that for every $n\ge N$ we obviously have the same inequality, i.e.:
\[\big|\mu_n\big([A]\times[B]\big)\big|<2\eps.\]
Since $\eps\in\big(0,1/12\big]$ is arbitrary, it holds that $\lim_{n\to\infty}\mu_n\big([A]\times[B]\big)=0$ and hence $\seq{\mu_n}{n\io_+}$ is weakly* null.
\end{proof}

\begin{theorem}\label{theorem:products_fsjnp}
For every two infinite compact spaces $K$ and $L$, their product $K\times L$ has the fsJNP.
\end{theorem}
\begin{proof}
First, notice that $\omega$ is homeomorphic to both $\Omega$ and $\Sigma$, so $\bo$, $\beta\Omega$ and $\beta\Sigma$ are mutually homeomorphic. Consequently, by Proposition \ref{prop:fsjnp_product_omega_sigma}, $(\bo)^2$ has the fsJNP and an fsJN-sequence witnessing this fact may be defined with supports contained completely in $\omega^2$ (as measures $\mu_n$'s defined above on $\beta\Omega\times\beta\Sigma$ have supports contained in $\Omega\times\Sigma$).

Let $D$ and $E$ be discrete countable subsets of $K$ and $L$, respectively. Let $\varphi\colon\omega\to D$ and $\psi\colon\omega\to E$ be bijections. By the Stone Extension Property of $\bo$, there are continuous maps $\Phi\colon\bo\to K$ and $\Psi\colon\bo\to L$ such that $\Phi\rstr\omega=\varphi$ and $\Psi\rstr\omega=\psi$. Let $\seqn{\mu_n}$ be an fsJN-sequence of measures on $(\bo)^2$ with supports in $\omega^2$. For each $n\io$ define a measure $\nu_n$ on $K\times L$ as follows:
\[\nu_n=\sum_{(x,y)\in\supp(\mu_n)}\mu_n\big(\{(x,y)\}\big)\cdot\delta_{(\varphi(x),\psi(y))},\]
it follows that $\big\|\nu_n\big\|=1$ and $\supp\big(\nu_n\big)$ is finite. Since $\seqn{\mu_n}$ is weakly* null, for every $f\in C(K\times L)$ we have:
\[\lim_{n\to\infty}\nu_n(f)=\lim_{n\to\infty}\mu_n\big(f(\Phi,\Psi)\big)=0,\]
where $f(\Phi,\Psi)(x,y)=f(\Phi(x),\Psi(y))\in C(\bo\times\bo)$, so $\seqn{\nu_n}$ is also weakly* null. This proves that $K\times L$ has also the fsJNP.
\end{proof}

Theorem \ref{theorem:products_fsjnp} yields a strengthening of the result of Khurana, Cembranos and Freniche.

\begin{corollary}\label{cor:product_ell1_grothendieck}
For every two infinite compact spaces $K$ and $L$, their product $K\times L$ does not have the $\ell_1$-Grothendieck property. \noproof
\end{corollary}

Schachermayer \cite[Proposition 5.3]{Sch82} and Cembranos \cite[Corollary 2]{Cem84} proved that a $C(K)$-space is Grothendieck if and only if it does not contain any complemented copy of $c_0$. This implies Cembranos' result \cite[Corollary 3]{Cem84} stating that the space $C(K\times L)$ admits a complemented copy of $c_0$ for any two infinite compact spaces $K$ and $L$. Corollary \ref{cor:product_ell1_grothendieck}, together with the characterization of Banakh, K\k{a}kol and \'Sliwa \cite[Theorem 1]{BKS19}, implies the following stronger result.

\begin{corollary}\label{cor:product_c0}
For every two infinite compact spaces $K$ and $L$, the space $C_p(K\times L)$ admits $(c_0)_p$ as a quotient. \noproof
\end{corollary}

Let us note here that the results presented in this section %Theorem \ref{theorem:products_fsjnp} and Corollaries \ref{cor:product_ell1_grothendieck} and \ref{cor:product_c0}
were also studied and generalized by K\k{a}kol, Marciszewski, Sobota and Zdomskyy \cite{KMSZ20} in the class of pseudocompact spaces.

%\section{Semadeni's result}
%
%\begin{enumerate}
%    \item comments about Semadeni's paper and results
%    \item inverse results?
%    \item additive space or what was that? to recall and provide the counterexample in $(\bo)^2$!!!
%\end{enumerate}

\end{document}